\documentclass[oneside]{amsart}

\usepackage[english]{babel}
\usepackage[T1]{fontenc}
\usepackage[utf8]{inputenc}

\usepackage{amsmath,
			amssymb,
            amsfonts,
            amsthm}
\usepackage{algorithmic}
\usepackage{stmaryrd} 
\usepackage{mathrsfs} 
\usepackage{dsfont} 

\usepackage{mathtools}
\usepackage{enumerate}
\usepackage{todonotes}
\usepackage{color}
\usepackage{ifdraft}

\usepackage{etoolbox}

\makeatletter
\patchcmd{\@settitle}{\uppercasenonmath\@title}{\LARGE\scshape}{}{}
\patchcmd{\@setauthors}{\MakeUppercase}{\scshape}{}{}
\patchcmd{\@setauthors}{\footnotesize}{\normalsize}{}{}
\makeatother

\newcommand{\al}{\alpha}

\newcommand{\be}{\beta}

\newcommand{\ga}{\gamma}
\newcommand{\Ga}{\Gamma}
\newcommand{\de}{\delta}
\newcommand{\De}{\Delta}

\newcommand{\vep}{\varepsilon}
\newcommand{\ze}{\zeta}

\newcommand{\et}{\eta}

\newcommand{\ka}{\kappa}

\newcommand{\la}{\lambda}

\newcommand{\rh}{\rho}

\newcommand{\si}{\sigma}
\newcommand{\Si}{\Sigma}
\newcommand{\ta}{\tau}

\newcommand{\ch}{\chi}

\newcommand{\Om}{\Omega}

\newcommand{\otilde}[1]{\widetilde{#1}}

\newcommand{\mbf}[1]{\mathbf{#1}} 
\newcommand{\mbb}[1]{\mathbb{#1}} 
\newcommand{\mrm}[1]{\mathrm{#1}} 
\newcommand{\mcl}[1]{\mathcal{#1}} 
\newcommand{\msr}[1]{\mathscr{#1}} 
\newcommand{\bsm}[1]{\boldsymbol#1} 

\newcommand{\bo}{\mathcal{O}} 
\newcommand{\so}{o} 

\newcommand{\N}{\mathbb{N}} 
\newcommand{\Z}{\mathbb{Z}} 
\newcommand{\R}{\mathbb{R}} 
\newcommand{\T}{\mathbb{T}} 

\newcommand{\inter}[1]{\operatorname{int}#1} 
\newcommand{\bdy}[1]{\operatorname{bd}#1} 

\newcommand{\cspan}{\operatorname{co}} 

\newcommand{\tp}[1]{#1^{\!\mathsf{T}}} 

\newcommand{\absc}{\ll} 
\newcommand{\sing}{\perp} 
\newcommand{\equi}{\asymp} 

\newcommand{\Exp}[1][]{\mathbb{E}_{#1}} 
\newcommand{\Var}[1][]{\mathbb{V}_{#1}} 
\newcommand{\Pre}[1][]{\mathbb{V}^{-1}_{\!#1}} 

\newcommand{\dist}{\operatorname{dist}}
\newcommand{\bsh}{\backslash}

\DeclareMathOperator*{\argmin}{argmin}
\DeclareMathOperator{\supp}{supp}

\newcommand{\bra}{\langle}
\newcommand{\ket}{\rangle}

\newcommand{\img}{\operatorname{im}} 
\newcommand{\invimg}[1]{#1^{\shortleftarrow}\!} 
\newcommand{\from}{\colon}

\newcommand{\der}{\operatorname{d}\!} 
\newcommand{\pder}[1][]{\partial_{#1}} 
\newcommand{\sder}[1][]{\operatorname{D}_{#1}\!} 
\newcommand{\rnder}[2]{
	\mathchoice
    	{\frac{\operatorname{d}\!#1}
        	  {\operatorname{d}\!#2}}
        {\operatorname{d}\!#1/\!
         \operatorname{d}\!#2}
        {\operatorname{d}\!#1/\!
         \operatorname{d}\!#2}
        {\operatorname{d}\!#1/\!
         \operatorname{d}\!#2}
    }

\newcommand{\grad}[1][]{\nabla_{\!#1}}
\newcommand{\hess}[1][]{\nabla^2_{\!#1}}
\newcommand{\lap}[1][]{\Delta_{#1}}

\usepackage[a4paper,
		    text={145mm,240mm},
		    head=6mm,
		    marginparwidth=22mm,
		    \ifdraft{showframe}{} 
		    ]{geometry}

\usepackage[obeyspaces,hyphens,spaces]{url}
\ifdraft{
	\usepackage[notcite,
    			notref]{showkeys}
	\usepackage{seqsplit}
	
	}{}

\usepackage{thmtools}
\declaretheoremstyle[spaceabove=8pt,
					 spacebelow=8pt]{mythmstyle}

\declaretheorem[name=Theorem,
				numberwithin=section,
                style=mythmstyle]{thm}
\declaretheorem[name=Proposition,
				sibling=thm,
                style=mythmstyle]{pro}
\declaretheorem[name=Lemma,
				sibling=thm,
                style=mythmstyle]{lem}
\declaretheorem[name=Corollary,
				sibling=thm,
                style=mythmstyle]{cor}
\declaretheorem[name=Assumption,
				sibling=thm]{ass}

\declaretheorem[name=Definition,
				style=definition,
                sibling=thm]{dfn}
\declaretheorem[name=Algorithm,
				style=definition,
                sibling=thm]{alg}

\declaretheorem[name=Remark,
				style=remark,
                sibling=thm]{rem}
\declaretheorem[name=Example,
				style=remark,
                sibling=thm]{exa}

\usepackage{csquotes}
\usepackage[backend = biber,
			bibencoding=utf8,
		    style=numeric-comp,
		    url = false,
		    isbn = false,
		    doi = false,
		    eprint = true]{biblatex}
\title[Analysis of a micro-macro acceleration method]{Analysis of a micro-macro acceleration method with minimum relative entropy moment matching}



\newcommand{\consp}{\mbb{X}} 
\renewcommand{\P}{\msr{P}} 
\newcommand{\momsp}{\msr{M}} 
\newcommand{\leb}[1]{\msr{L}^{#1}} 
\newcommand{\Mb}{\msr{M}_{b}} 
\newcommand{\Bm}{\msr{B}_{m}} 
\newcommand{\cont}{\msr{C}} 
\newcommand{\contb}{\msr{C}_{b}} 
\newcommand{\conto}{\msr{C}_{0}} 
\newcommand{\cdiff}[1]{\msr{C}^{#1}} 
\newcommand{\bcdiff}[1]{\msr{C}_{b}^{#1}} 
\newcommand{\ocdiff}[1]{\msr{C}_{0}^{#1}} 
\newcommand{\comp}{\msr{K}} 

\newcommand{\lre}{\mcl{I}} 
\newcommand{\finf}{\mcl{J}} 
\newcommand{\match}{\mcl{M}} 
\newcommand{\dmatch}{\msr{D}} 
\newcommand{\res}{\mcl{R}} 
\newcommand{\resf}{R} 
\newcommand{\sem}{\mcl{S}} 
\newcommand{\gen}{\mcl{L}} 
\newcommand{\proj}[1]{\mcl{P}_{#1}} 
\newcommand{\inop}{\mcl{F}} 
\newcommand{\locerr}{\mbf{e}} 

\renewcommand{\Pr}{\mbb{P}} 
\newcommand{\Bor}{\msr{B}or} 
\newcommand{\Leb}{\der{x}}
\newcommand{\Norm}{\mcl{N}} 
\newcommand{\Law}[1]{\mu_{#1}} 
\newcommand{\expd}{\mcl{E}} 
\newcommand{\p}{\rh} 
\newcommand{\app}[1]{\overline{#1}} 

\numberwithin{equation}{section}
\begin{document}

\author[T.~Leli\`{e}vre]{Tony Leli\`{e}vre}
\address[T.~Leli\`{e}vre]{CERMICS (ENPC), Inria, Universit\'{e} Paris-Est, F-77455 Marne-La-Vall\'{e}e, France}
\email{tony.lelievre@enpc.fr}

\author[G.~Samaey]{Giovanni Samaey}
\address[G.~Samaey]{NUMA, Department of Computer Science, KU Leuven, 3001 Heverlee, Belgium}
\email{giovanni.samaey@kuleuven.be}

\author[P.~Zieli\'{n}ski]{Przemys{\l}aw Zieli\'{n}ski}
\address[P.~Zieli\'{n}ski]{NUMA, Department of Computer Science, KU Leuven, 3001 Heverlee, Belgium}
\email{przemyslaw.zielinski@kuleuven.be}

\date{\today}

\begin{abstract}
We analyse convergence of a micro-macro acceleration method for the Monte Carlo simulation of stochastic differential equations with time-scale separation between the (fast) evolution of individual trajectories and the (slow) evolution of the macroscopic function of interest. We consider a class of methods, presented in~\cite{DebSamZie2017}, that performs short bursts of path simulations, combined with the extrapolation of a few macroscopic state variables forward in time. After extrapolation, a new microscopic state is then constructed, consistent with the extrapolated variable and minimising the perturbation caused by the extrapolation. In the present paper, we study a specific method in which this perturbation is minimised in a relative entropy sense. We discuss why relative entropy is a useful metric, both from a theoretical and practical point of view, and rigorously study local errors and numerical stability of the resulting method as a function of the extrapolation time step and the number of macroscopic state variables. Using these results, we discuss convergence to the full microscopic dynamics, in the limit when the extrapolation time step tends to zero and the number of macroscopic state variables tends to infinity.
\end{abstract}

\maketitle

\section{Introduction}\label{sec:intro}
The considerations and results presented in this manuscript originate from the need to efficiently simulate the following expectations
\begin{equation}\label{eq:observ}
t\mapsto\Exp[][f(X_t)],
\end{equation}
for times $t\in[0,T]$, where $X_t$ is a given diffusion process and $f$ is a function of interest. In the present work, we focus on issues concerning the temporal discretisation of the underlying evolution of the random variable $X_t$, with time step $\de t>0$, for a large final time $T$. The full simulation requires also the consistent approximation of expectations in~\eqref{eq:observ}; this is usually achieved by Monte Carlo methods~\cite{Caflisch1998,Higham2001}.

From the computational perspective, we are interested in \emph{stiff} systems, with a separation between a (fast) time-scale, on which the individual trajectories of $X_t$ need to be simulated, and the (slow) time-scale, on which the expectations~\eqref{eq:observ} evolve. This feature leads to a stability constraint on the time discretisation methods that forces us to take very small steps $\de t$, compared to the desired time horizon $T$ for~\eqref{eq:observ}. The discrepancy between the minuscule leaps we have to make and the big times we want to arrive at, quickly makes the cost of Monte Carlo simulation prohibitive.
This problem led to the development of various general multiscale algorithmic approaches, such as \emph{heterogeneous multi-scale}~\cite{EEng2003,EEnqLiRenVan2007} or \emph{equation-free}~\cite{KevGeaHymKevRunThe2003,KevSam2009} methods, which try to overcome the scale separation, or even use it to one's advantage.

As a part of this study, we analyse the accuracy of a micro-macro acceleration method to efficiently simulate observables~\eqref{eq:observ}. The algorithm exploits the time-scale separation by operating with two time steps: a microscopic one~$\de t$, suited for the underlying stochastic process, and a macroscopic $\De t\gg \de t$, which we believe to be natural for the evolution of the expectations. To describe the coarse (macroscopic) behaviour of the process, we reduce the diffusion $X_t$ to a finite number of \emph{macroscopic state variables}, given as
\begin{equation}\label{eq:macvar}
m_l(t) \doteq \Exp[][\resf_l(X_t)],\quad l=1,\dotsc,L,
\end{equation}
for some appropriately chosen functions $\resf_l$ (cf.~\cite{KevGeaHymKevRunThe2003}). These variables store partial (statistical) information about the distributions of the stochastic process. Due to scale separation, we can expect that the variables $m_l$ evolve under the influence of a vector field with natural time scale $\De t$. Although we do not know this vector field in general, we can (and will) approximate it by directly estimating the time derivatives of every $m_l$, to move the simulation forward in time by $\De t$. One time step of the micro-macro acceleration method includes (i) microscopic \emph{simulation} of~$X_t$ for a small batch of time steps of size $\de t$; (ii) \emph{restriction}, i.\,e., extraction of an estimate of the macroscopic time derivative, based on the simulation in the first stage; (iii) forward in time \emph{extrapolation} of the macroscopic state; and (iv) \emph{matching} of the last microscopic state from (i) with the extrapolated macroscopic state. We provide a more detailed description in Section~\ref{sec:mM_alg}.

The most challenging stage is the matching. It amounts to an inference procedure to pick a distribution, having prescribed (extrapolated) \emph{macroscopic state} -- a particular point in the $L$-dimensional space of macroscopic state variables. This is an \emph{ill-posed} problem: there may be no solution, or the solution may not be unique. Both cases may depend sensitively on the prescribed macroscopic state that one wants to match with. 

Our strategy is to use a \emph{prior} distribution $\mu$, which comes from the last available microscopic state in the current step and alter it, so that it becomes consistent with the~extrapolated macroscopic state. Particularly, if~$m_1,\dotsc,m_L$ are the extrapolated macroscopic states, we obtain the matched distribution from the prior $\mu$ as the solution to the following optimisation problem
\begin{equation}\label{eq:match}
\argmin_{\nu}\;\lre(\nu\|\mu), \qquad \text{constrained on}\ \int\!\resf_l\,\der{\nu}=m_l,
\end{equation}
where
\[
\lre(\nu\|\mu)=\int\!\ln\rnder{\nu}{\mu}\,\der{\nu}
\]
and we minimise over all probability distributions $\nu$ absolutely continuous with respect to $\mu$. The objective function $\lre$ in~\eqref{eq:match} is the \emph{relative entropy} of $\nu$ with respect to $\mu$, also known as \emph{Kullback-Leibler} or \emph{information divergence} in the information theory literature~\cite{KulLei1951,Kullback1978}.

The analysis and intuition behind problem~\eqref{eq:match} relies on a geometric interpretation that views matching as a projection operator in the space of distributions, endowed with the topology generated by the relative entropy~\cite{Csiszar1975,PavFer2013}. This is not a metric topology~\cite{Harremoes2007}. Nevertheless, due to Pinsker's inequality, by which relative entropy dominates the square of total variation norm, and various analogies with Euclidean geometry, $\lre(\nu\|\mu)$ can be regarded as a ``square distance'' between two probability distributions. In particular, whenever $\mu^*$ is a solution to~\eqref{eq:match}, and $\nu$ satisfies the constraints, a so-called \emph{Pythagorean identity} holds:
\begin{equation}\label{eq:pyth}
\lre(\nu||\mu) = \lre(\nu||\mu^*) + \lre(\mu^*||\mu).
\end{equation}
We can intuitively understand the foregoing property as: the matching $\mu^*$ is an ``orthogonal projection'' of $\mu$ on the submanifold of probability densities that satisfy the constraints generated by the moments of $\nu$.

Before moving on to the technical content of the paper, we finalize this introduction with two important points. First, Section~\ref{sec:why-RE} discusses the reasons behind the choice for relative entropy as the quantity to be minimized in~\eqref{eq:match}. Second, Section~\ref{sec:outline} briefly sketches the main contributions of this work and the outline of the paper.

\subsection{On the usefulness of relative entropy matching\label{sec:why-RE}}
No rigorous justification exists why the relative entropy is the proper choice for the matching procedure. The first description of the micro-macro acceleration with matching in \cite{DebSamZie2017} contained multiple examples of metrics that could be used in the optimisation procedure~\eqref{eq:match}. Nevertheless, we identify below three reasons that motivate the choice for relative entropy: the first one from a ``physical'' point of view, the second one from a ``numerical'' point of view, and the third one from a ``theoretical'' point of view (related to error control and adaptivity).

\textbf{The physical point of view.} The choice for relative entropy, specified in~\eqref{eq:match}, is closely related to the \emph{maximum entropy principle}~\cite{Jaynes1957,Jaynes1957a}, which dictates that one should look for a distribution, consistent with available data, that maximises the entropy $\mcl{H}_{\nu}=-\lre(\,\cdot\,||\nu)$, see also~\cite{Jaynes1982}.
This convention has been extensively used for constructing closures of moment systems to derive constitutive equations for kinetic equations~\cite{IlgKarOtt2002,HauLevTit2008,SamLelLeg2011,MalBruDub2015}.
Moreover, in the context of data assimilation, procedure~\eqref{eq:match} serves as the risk-neutral approach for calibrating asset-pricing models~\cite{AveFriHolSam1997,Avellaneda1998} and an optimal approximation of spectral densities~\cite{GeoLin2003,FerRamTic2011}.

\textbf{The numerical point of view.} Relative entropy is also convenient numerically. The computational procedure to determine~\eqref{eq:match} is based on a dual formulation, see also~\cite{DebSamZie2017}, which looks for the vector of Lagrange multipliers $\la^*_1,\dotsc,\la^*_L$ that solve
\begin{equation}\label{eq:lagmul}
Z(\la^*_1,\dotsc,\la^*_L)^{-1}\int \resf_l\cdot\exp\!\Big(\sum_{p=1}^{L}\la^*_p\resf_p\Big)\,\der{\mu} = m_l,\qquad 1\le l \le L
\end{equation}
where 
\[
Z(\la^*_1,\dotsc,\la^*_L)=\int\exp\!\Big(\sum_{l=1}^{L}\la^*_l\resf_l\Big)\der{\mu}
\]
is the \emph{partition function}. As long as we can compute or estimate the integrals, \eqref{eq:lagmul} constitutes a finite-dimensional system of non-linear equations, which can be solved by a Newton procedure. Moreover, the density of the distribution $\mu_L^*$ satisfying~\eqref{eq:match} reads
\begin{equation}\label{eq:matchden}
\rnder{\mu^*_L}{\mu}= Z(\la^*_1,\dotsc,\la^*_L)^{-1} \exp\!\Big(\sum_{l=1}^{L}\la^*_l\resf_l\Big).
\end{equation}

There are two advantages to this representation of~$\mu^*_L$. First, because the exponential function is positive, $\mu^*_L$ is always \emph{equivalent} to the prior distribution $\mu$, that is, their supports are the same. Second, the exponential function serves as the \emph{likelihood ratio} for the importance sampling of~$\mu^*_L$~\cite[Ch.~V.1]{AsmGly2007}. Therefore, we can estimate the observables~\eqref{eq:observ} with respect to~$\mu^*_L$, by considering a number of replicas $X^j$, $j=1,\dotsc,J$ distributed according to the prior~$\mu$, and computing weighted averages with weights $w_j=\exp\big(\sum_{l=1}^{L}\la^*_l\resf_l(X^j)\big)$. For more details on the numerical implementation, we refer to~\cite{DebSamZie2017}.

\textbf{Error control and adaptivity.}
The properties of relative entropy provide also a convenient \emph{a posteriori} error analysis that allows appending the set of macroscopic state variables with new ones that reduce relative entropy in a greedy way. To illustrate this idea, assume that the macroscopic states are moments of an unknown target probability distribution $\nu^t$, that is, $m_l=\int\resf_l\der{\nu^t}$ for $l=1,\ldots,L$. Moreover, let $\mu^*_L$ be the matching of a prior distribution $\mu$ with $m_1,\dotsc,m_L$, which we already computed. We want to get an indication of the gain we can expect by adding a new macroscopic state variable, corresponding to a function $\resf_{L+1}$, to the matching procedure~\eqref{eq:match}.

Denote by $\mu^*_{L+1}$ the matching of the same prior $\mu$ with extended system $m_1,\dotsc,m_L,m_{L+1}$, where $m_{L+1}=\int\resf_{L+1}\,\der{\nu^t}$ in accordance with our assumption. By construction, the set of constraints in~\eqref{eq:match} generated by $m_1,\dotsc,m_{L+1}$ is a subset of those yielded by $m_1,\dotsc,m_L$. Therefore, by the \emph{transitivity property} of the relative entropy matching~\cite[Thm.~2.3]{Csiszar1975}, we can alternatively obtain $\mu^*_{L+1}$ by matching $\mu^*_L$ with $m_1,\dotsc,m_{L+1}$. This has two consequences. First, as we already computed $\mu^*_L$ that has correct first $L$ macroscopic states, using it instead of $\mu$ in~\eqref{eq:lagmul}, we can cheaply obtain the Lagrange multipliers $\otilde{\la}^*_1,\dotsc,\otilde{\la}^*_{L+1}$ for $\mu^*_{L+1}$. Second, applying the Pythagorean identity~\eqref{eq:pyth} to $\mu^*_{L+1}$, with $\nu^t$ in place of $\nu$ and $\mu^*_{L}$ as a prior, produces
\begin{equation*}
\lre(\nu^t||\mu^*_L) = \lre(\nu^t||\mu^*_{L+1}) + \lre(\mu^*_{L+1}||\mu^*_L),
\end{equation*}
from which we get 
\begin{equation}\label{eq:pos_est}
\lre(\nu^t||\mu^*_L) - \lre(\nu^t||\mu^*_{L+1}) = \lre(\mu^*_{L+1}||\mu^*_L).
\end{equation}

The left hand side of equality~\eqref{eq:pos_est} gives an indication of \emph{how much} accuracy one expects to gain by adding $\resf_{L+1}$ to the system of macroscopic state variables. The right hand side reads
\begin{equation*}
\lre(\mu^*_{L+1}||\mu^*_L) = 
\sum_{l=1}^{L+1}\otilde{\la}^*_lm_l-\sum_{l=1}^{L}\la^*_lm_l + \ln\frac{Z(\la^*_1,\dotsc,\la^*_L)}{Z(\otilde{\la}^*_1,\dotsc,\otilde{\la}^*_{L+1})}.
\end{equation*}
Note that $\lre(\mu^*_{L+1}||\mu^*_L)$ does not depend on the target density $\nu^t$ and can be evaluated numerically, as soon as we estimate the Lagrange multipliers by solving~\eqref{eq:lagmul}. Therefore, equality~\eqref{eq:pos_est} enables to develop an adaptive procedure selecting new macroscopic state variables that maximally reduce the relative entropy at a current time step of the micro-macro acceleration method.

\subsection{Main contributions and outline\label{sec:outline}}

The above arguments give ample motivation to study micro-macro acceleration methods with relative entropy matching. The micro-macro acceleration method was introduced in~\cite{DebSamZie2017} using a more general, axiomatic definition of the matching operator, and a convergence result was presented there based on some generic properties for all underlying components of the method. The assumptions in~\cite{DebSamZie2017} do not apply to the matching given by~\eqref{eq:match}, and only numerical results indicating the convergence are presented, for a non-trivial test case originating from the micro-macro simulation of dilute polymers. In this respect, the current paper expands the body of work initiated in~\cite{DebSamZie2017}. 

This paper investigates the numerical properties of the micro-macro acceleration method with relative entropy matching: (i) numerical stability, to establish bounds on the propagation of local errors; and (ii) local errors produced by the matching with finite number $L$ of macroscopic state variables. We achieve this goal by demonstrating how the properties of minimum relative entropy regularisation can be combined with the features of the underlying evolution of $X_t$, to provide a rigorous analysis of the micro-macro acceleration method. To establish convergence of the micro-macro acceleration method to the underlying microscopic dynamics, we then combine the above results and consider the limit when the extrapolation time step $\Delta t$ tends to zero and the number of macroscopic state variables $L$ tends to infinity. 

The remainder of this manuscript is organised as follows. Section~\ref{sec:mM_alg} gives a detailed account of the micro-macro acceleration method, keeping the exposition general enough so that it applies in a broader context than the one we study later. In Section~\ref{sec:prelim}, we start with the basic notions and assumptions on the underlying diffusion process $X_t$. 
In Section~\ref{sec:minxent}, we rigorously define the matching operator corresponding to~\eqref{eq:match} and study its properties, such as dependence on the prior distribution. We introduce the remaining constructions and gather all assumptions needed to complete the proof of convergence in Section~\ref{sec:conv}.
Section~\ref{sec:expansion} is devoted to the investigation of the~relation between the evolution of the diffusion and the relative entropy. Finally, the last two sections expose the convergence proof that relies on two main ingredients: the numerical stability of the~method (Section~\ref{sec:stab}), which reduces the global errors to local ones, and the consistency of local errors (Section~\ref{sec:locerr}), which implies the convergence.

\section{Micro-macro acceleration method}
\label{sec:mM_alg}
The micro-macro acceleration method aims at being faster than a full microscopic simulation, while converging to it when the extrapolation time step $\Delta t$ vanishes and the number of extrapolated macroscopic state variables $L$ goes to infinity. The underlying assumption for the method to be efficient is that the macroscopic state variables can be simulated on a much slower time scale than the microscopic dynamics, thus allowing the choice of a large extrapolation time step $\Delta t$ compared with the time step $\delta t$ for microscopic simulation. 


The main building blocks of the method can be grouped into two categories: \emph{propagators}, which move the simulation forward in time on the micro or macro time scales; and \emph{transition operators}, which connect two levels of description. The microscopic states are given by the random variables $X_t$, and the macroscopic states are described by vectors in the Euclidean space $\R^{L}$, with $L$ the number of macroscopic state variables used to preserve information about distributions. We now detail first the transition operators (Section~\ref{sec:transition}), after which we discuss the propagation operators (Section~\ref{sec:propagators}). All components are then collected in a description of the micro-macro acceleration method in Section~\ref{sec:mM-method}.

\subsection{Transition operators\label{sec:transition}}

To transition from microscopic to macroscopic states, we consider the \emph{restriction operator}~$\res$. It is determined by the vector $\mbf{\resf}$ of functions $\resf_1,\dotsc,\resf_L$, and for a random variable $X$, we define
\begin{equation}\label{eq:restr_rv}
\res(X) \doteq \Exp\big[\mbf{\resf}(X)\big].
\end{equation}
This formula is consistent with~\eqref{eq:macvar}, as $\res(X_t)_l=m_l(t)$ when $X_t$ is the diffusion generating the observables in~\eqref{eq:observ}. Note also that the vector $\res(X)$ depends only on the law of a random variable~$X$, which we denote~$\Law{X}$. Therefore for the analysis, it will turn out to be more convenient to consider $\res$ as acting on the family of probability measures, see~\eqref{eq:restr}.

\begin{rem}[On notation]\label{rem:notmac}
As we mention in Section~\ref{sec:intro}, we use the Euclidean space $\R^{L}$ to store the statistical (coarse) information of the underlying distributions. To visually highlight the elements of~$\R^{L}$ and $\R^{L}$-valued functions, we henceforward apply bold fonts for their symbols. We also use $\|\cdot\|$ to denote the Euclidean norm in $\R^{L}$.
\end{rem}

To proceed from the macroscopic state to the microscopic distributions, we face the inverse problem
\begin{equation}\label{eq:inv_prob}
\text{given}\ \mbf{m}\in\R^{L}\ \text{find}\ Y\ \text{such that}\ \res(Y)=\mbf{m}.
\end{equation}
This is an \emph{ill-posed} problem: there may be no solution or the solution may not be unique, and both cases may depend sensitively on~$\mbf{m}$. Usually, when~\eqref{eq:inv_prob} has a solution, it is \emph{underdetermined} in the sense that infinitely many consistent (laws of) random variables exist. 
As announced in the Introduction, we will regularize~\eqref{eq:inv_prob} by considering a prior random variable $X$, which is naturally available in the micro-macro acceleration method, and define the \emph{matching operator} as
\begin{equation}\label{eq:matching_intuitive}
\match(\mbf{m},X) = \argmin_{Y}\lre(\Law{Y}||\Law{X})\quad\text{constrained on}\ \res(Y)=\mbf{m}.
\end{equation}
To make sense of $\match(\mbf{m},X)$, we first consider the probability measure~$\mu^*$ that solves~\eqref{eq:match}, and next choose any random variable $Y$ so that $\Law{Y}=\mu^*$. There is no generic way to pick $Y$ but, as long as we are concerned with the expectations and measure the weak error, the particular choice of $Y$ is not important. 
\begin{rem}[Matching ensembles]
In practice, when performing Monte Carlo simulation, we always start with an ensemble $\{X^j\}_{j=1}^J$ of $J$ replicas sampled from $\Law{X}$. The formula~\eqref{eq:matchden}
then provides a convenient way to sample $Y$ with the weighted replicas $Y^j=(w^j,X^j)$, where the weights are $w^j=\exp\!\big(\sum\la_l^*\resf_l(X^j)\big)$ and the Lagrange multipliers $\la_1^*,\dotsc,\la_L^*$ satisfy~\eqref{eq:lagmul}. For more on the practical implementation of the matching operator with finite ensembles, we refer to~\cite{DebSamZie2017}.
\end{rem}

\subsection{Propagators\label{sec:propagators}}

The first propagator, operating on the micro time scale, is the one-step time discretisation of SDE
\begin{equation}\label{eq:sde}
\der{X}_t = a(X_t)\der{t}+b(X_t)\der{W_t},
\end{equation}
which generates the diffusion process $X_t$. It performs a full microscopic simulation on a time interval of length $\De\ta>0$. The computational cost of the simulation is usually high, but the time $\De\ta$ we devote to it is very short, compared to the time scale on which the averages~\eqref{eq:observ} evolve. In practice, we divide $\De\ta$ into $K$ steps of length $\de t$, thus obtaining a time mesh $\{t_k=k\de t:\ k=0,\dotsc,K\}$, and use a stochastic numerical method for SDE~\eqref{eq:sde}. For example, we can employ an Euler-Maruyama step to propagate a given initial random variable $\app{X}_0$ as
\begin{equation}\label{eq:sde_em}
\app{X}_{k} = \app{X}_{k-1} + a(\app{X}_{k-1})\de t + b(\app{X}_{k-1})(W_{t_{k}}-W_{t_{k-1}}),
\end{equation}
for $k=1,\dotsc,K$.

The second propagator is extrapolation, which moves \emph{only the macroscopic variables} forward in time over the macroscopic time step $\De t\gg\De\ta$. In this manuscript, we consider first order extrapolation of the macroscopic variables, called coarse forward Euler integration~\cite{GeaKevThe2002}. Assuming we have at our disposal two macroscopic variables $\mbf{m}_{0},\mbf{m}_{1}$ separated by $\De\ta$, which we obtain by averaging the microscopic states, the extrapolation proceeds as follows:
\begin{equation}\label{eq:cfe}
\mbf{m}^{\mrm{ext}} \doteq \mbf{m}_{0} + \De t\frac{\mbf{m}_{1} - \mbf{m}_{0}}{\De\ta}.
\end{equation}
Higher order versions of \eqref{eq:cfe}, which require macroscopic states at additional time instances, can be constructed in several ways: using polynomial extrapolation~\cite{GeaKevThe2002}; implementing Adams-Bashforth or Runge-Kutta methods~\cite{RicGeaKev2004,LeeGea2007,LafLejSam2016};
or trading accuracy for stability by designing a multistep state extrapolation method~\cite{VanRoo2008}.

\subsection{Micro-macro acceleration method\label{sec:mM-method}}
We now have all the ingredients to describe the complete method in Algorithm~\ref{alg:accel}. We introduce two indices, $k=0,\ldots,K$ and $n=0,\ldots,N$, to emphasise the fact that there are two time steps involved: the microscopic time step $\de t$, to evolve the full microscopic dynamics over $\De\ta$; and the macroscopic time step $\De t$, to perform extrapolation of the macroscopic state variables up to the final time $T$.
\begin{alg}\label{alg:accel}
Given a microscopic state $\app{X}_n$ at time~$t_n$, a number $L$ of macroscopic state variables, macroscopic step size $\De t>0$, microscopic step size $\de t>0$, and a number $K\in\mbb{N}$ of microscopic steps, with $K\de t=\De\ta\leq\De t$, compute the microscopic state $\app{X}_{n+1}$ at time $t_{n+1}=t_n+\De t$ via a four-step procedure:
\begin{enumerate}[(i)]
\item \emph{Simulate} the microscopic system over $\De\ta$ with $K$ time steps of size $\de t$ using a microscopic discretization scheme, such as~\eqref{eq:sde_em}, to obtain a sequence of microscopic states
\begin{equation*}
\app{X}_{n,0},\app{X}_{n,1},\dotsc,\app{X}_{n,K},
\end{equation*}
with $\app{X}_{n,0}\doteq\app{X}_n$.
\item \emph{Record} the $L$-dimensional macroscopic states $\mbf{m}_{n,k}=\res(\app{X}_{n,k})$ for $k=0,\dotsc,K$.
\item \emph{Extrapolate} the macroscopic states $\mbf{m}_{n,0},\dotsc,\mbf{m}_{n,K}$ over a step of size $\De t$, for instance using~\eqref{eq:cfe}, to a new macroscopic state $\mbf{m}_{n+1}$ at time $t_{n+1}$.
\item \emph{Match} the microscopic state $\app{X}_{n,K}$ at time $t_{n,K}$ with the extrapolated macroscopic state $\mbf{m}_{n+1}$
\begin{equation*}
\app{X}_{n+1} = \match(\mbf{m}_{n+1},\app{X}_{n,K}),
\end{equation*}
to obtain a new microscopic state $\app{X}_{n+1}$ at time $t_{n+1}$.
\end{enumerate}
\end{alg}

By successive application of~Algorithm~\ref{alg:accel}, we obtain after performing $N$ steps the random variable $\app{X}_N=\app{X}_N^{K\de t,\De t,L}$ that ``approximates'' the final value $X_T$ of the diffusion process. Because we are interested in estimating the averages given by~\eqref{eq:observ}, we measure the quality of $\app{X}_N$ by the weak error
\begin{equation*}
\Exp[][f(X_T)] - \Exp[][f(\app{X}_N)].
\end{equation*}
We find sufficient conditions, under which this error goes to zero as the time steps $\de t$ and $\De t$ go to zero, and the number of macroscopic states $L$, used for extrapolation, goes to infinity. The precise statement of the result we prove is the content of Section~\ref{sec:conv}.

\section{Mathematical setting}
\label{sec:prelim}
Throughout the manuscript, we consider diffusion processes that live on a configuration space denoted by $\consp$. To avoid technical complications that are unnecessary, in view of the goals of the paper, we make the following standing assumption on the configuration space:
\begin{ass}\label{ass:consp}
The configuration space $\consp$ is either the Euclidean space $\R^{d}$, or the torus $\T^d\doteq\R^{d}/\Z^{d}$, with dimension $d\in\mbb{N}$.
\end{ass}
This assumption avoids, for instance, the issue of proper boundary conditions on the involved diffusion processes on bounded subsets of~$\R^{d}$. Nevertheless, Assumption~\ref{ass:consp}
still contains two common settings for diffusions: 
\begin{itemize}
\item The whole space $\R^{d}$ acts as an example of a non-compact configuration space; 
\item The torus $\T^{d}$ acts as a physically relevant compact case, resulting from periodic boundary conditions.
\end{itemize}
It will turn out that the proofs and derivations for a non-compact configuration space will require additional assumptions, compared to the compact setting. We will point out these assumptions when relevant.

\begin{rem}[Basic notations]
The Lebesgue measure on $\consp$ is denoted by $\Leb$, and for any two points $x,y\in\consp$, $|x-y|$ stands for the distance between them. On $\R^d$, this is the usual metric generated by the Euclidean norm $|\cdot|_{\R^d}$; on $\T^d$ this distance is defined as $|x-y|\doteq\min\{|x-y+k|_{\R^d}:\ k\in\Z^d\}$, where, to make our notation more consistent, we do not distinguish between a representative and its equivalence class. If $x,y\in\R^{d}$, $\tp{x}$ is the \emph{transpose} and, consequently, $\tp{x}y$ and $x\tp{y}$ are the \emph{scalar product} and \emph{tensor product} of two vectors $x$ and $y$.
Throughout the paper, a \emph{smooth} function means a $\cdiff{\infty}$ function, and we use $\pder$, $\grad$, $\hess$ for the partial derivative, gradient and Hessian, respectively. For vector-valued functions, we write $\sder$ to denote the strong derivative (Jacobian matrix).
\end{rem}

\subsection{Spaces of measures and spaces of functions}\label{sec:spaces}
In what follows, we denote by $\P(\consp)$ the set of all probability measures on $\consp$ defined on the $\si$-field $\Bor(\consp)$ of Borel subsets of $\consp$. The symbols $\Exp[\mu]$ and $\Var[\mu]$ stand for the~expectation and variance(-covariance) with respect to~$\mu\in\P(\consp)$. 
We also consider the~Banach space $\Mb(\consp)$ of all bounded and signed Borel measures, of which $\P(\consp)$ is a convex subset. The norm on $\Mb(\consp)$ is the \emph{total variation} (TV), and for $\et\in\Mb(\consp)$ it reads
\begin{equation*}
\|\et\|_{TV}\doteq\sup_{B\in\Bor(\consp)}|\et(B)|+|\et(\consp\setminus B)|,
\end{equation*}
see, e.g., \cite{Bogachev2007}.
For $\mu,\nu\in\P(\consp)$, this norm induces the \emph{total variation distance} $\|\mu-\nu\|_{TV}$, 
which amounts to the $L^1$-norm of the difference between the densities
\begin{equation*}
\|\mu-\nu\|_{TV} = \int_{\consp}\Big|\rnder{\mu}{\et}-\rnder{\nu}{\et}\Big|\,\der\et,
 \end{equation*}
whenever $\mu,\nu$ are absolutely continuous (denoted $\absc$) with respect to a common measure $\et$, and $\rnder{\mu}{\et}$, $\rnder{\nu}{\et}$ are the corresponding densities (Radon-Nikodym derivatives). We also write $\mu\sing\nu$ whenever the measures are singular (their supports are disjoint) and $\mu\equi\nu$ when they are equivalent (have the same sets of measure zero).

Besides the spaces in which probability measures live, we also need to characterize the space of functions we want to consider as macroscopic state variables. We denote by $\Bm(\consp)$ the space of bounded, Borel measurable functions on $\consp$ equipped with the sup-norm $\|\cdot\|_{\infty}$. The symbol $\bra f|\et\ket\in\R$ stands for the pairing (congruence) between a function $f\in\Bm(\consp)$ and a signed measure $\et\in\Mb(\consp)$, and $f\et\in\Mb(\consp)$ stands for the measure having density $f$ with respect to $\et$. Note that, if $\mu\in\P(\consp)$, we have $\bra f|\mu\ket = \Exp[\mu]f$. We will also use two subspaces of $\Bm(\consp)$: $\conto(\consp)$, of all continuous functions ``vanishing at infinity''\footnote{$f\in\conto(\consp)$, if for all $\vep>0$ there is a compact $K\subset\consp$ such that $f(x)<\vep$ for every $x\in\consp\setminus K$}; and $\contb(\consp)$, of all bounded continuous functions on the configuration space $\consp$. Recall that, if $\consp$ is compact, $\conto(\consp)=\contb(\consp)$, and both consist of all continuous functions on $\consp$. When we need higher regularity, we consider the Banach space $\bcdiff{k}(\consp)$, of all $k$-times differentiable functions with bounded derivatives, with norm $\|f\|_{k,\infty}=\sum_{|\al|\leq k}\|\pder[\al]f\|_{\infty}$, where $\al$ is a multi-index, and in particular, see Section~\ref{sec:conv}, its subspace $\ocdiff{k}(\consp)$ of functions with vanishing derivatives. For a vector function with values in the space of macroscopic variables $\mbf{\resf}\from\consp\to\R^{L}$, such that $\resf_l\in\ocdiff{k}(\consp)$ for all $l=1,\dotsc,L$, we denote (see Remark~\ref{rem:notmac})
\begin{equation*}
\|\mbf{\resf}\|^2_{k,\infty} = \big\|\big(\|\resf_1\|^2_{k,\infty},\dotsc,\|\resf_L\|^2_{k,\infty}\big)\big\|^2 = \sum_{l=1}^{L}\|\resf_l\|^2_{k,\infty}.
\end{equation*}

Finally, we need to describe in what sense we will consider convergence of sequences of probability measures. In this paper, we will mainly be concerned with the \emph{weak convergence of probability measures} on $\P(\consp)$. A sequence $\{\mu_n\}_{n\geq1}$ of probability measures on $\consp$ converges weakly to $\mu\in\P(\consp)$, if $\lim_{n\to+\infty}\Exp[\mu_n]f=\Exp[\mu]f$ holds for every $f\in\contb(\consp)$. The usefulness of the \emph{weak topology on $\P(\consp)$}, induced by this convergence, stems from its metrizability (by the Prohorov metric) and the convenient characterisation of compactness~\cite{Dudley2002}: the weakly closed family of measures $\msr{A}\subset\P(\consp)$ is \emph{weakly compact in $\P(\consp)$} if and only if it is \emph{(uniformly) tight}, i.e.~given any $\vep>0$, there is a compact subset $K\subseteq\consp$ such that $\mu(\consp\setminus K)\leq\vep$ for all $\mu\in\msr{A}$. In particular, if $\consp$ is compact itself, $\P(\consp)$ is compact in the weak topology. In the non-compact case, a sufficient condition results from uniform control over the absolute first moment:
\begin{lem}\label{lem:tightness}
Let $\msr{A}\subseteq\P(\R^{d})$ be a family of probability measures such that there is a constant $M>0$ and $\Exp[\mu][|\cdot|]\leq M$ for all $\mu\in\msr{A}$, then $\msr{A}$ is tight. 
\end{lem}
\begin{proof}
Fix $\vep>0$ and consider a closed ball $K=\{x\in\R^{d}: |x|\leq r\}$, where $r$ is large enough so that $r\geq M/\vep$. From the Markov inequality we get
\begin{equation*}
\mu(\consp\setminus K) = \mu(\{x\in\R^{d}: |x|>r\})\leq \frac{\Exp[\mu][|\cdot|]}{r}\leq\vep.\qedhere
\end{equation*}
\end{proof}

Throughout this manuscript, we work mainly on $\P(\consp)$, but we introduce $\Mb(\consp)$ to utilize its elements as the ``directions'' for derivatives of mappings on $\P(\consp)$. We say that a direction $\et\in\Mb(\consp)$ is \emph{admissible for $\mu\in\P(\consp)$}, if there is an $\vep_0>0$ such that $\mu +\vep_0\et\in\P(\consp)$. (Note that this immediately implies that, for any admissible direction $\et$, we have $\langle 1 | \et\rangle=0$.)
\begin{dfn}\label{dfn:dirder}
Let $F\from\P(\consp)\to\R$ and $\mu\in\P(\consp)$. The mapping $F$ has a \emph{(one-sided) directional derivative} $\der{F(\mu;\et)}$ in the direction $\et\in\Mb(\consp)$, admissible for $\mu$, if the limit
\begin{equation*}
\der{F(\mu;\et)} \doteq \lim_{\vep\searrow0}\frac{F(\mu+\vep\et)-F(\mu)}{\vep}
\end{equation*}
exists.
\end{dfn}
We extend this definition, in an obvious way, when $F$ acts into a Banach space, like $\Mb(\Om)$ or $\Bm(\Om)$. In the case $F$ depends on other variables, we use the symbol $\pder F$ with an appropriate lower subscript on $\pder$. We summarise a few useful properties of directional derivatives below.
\begin{lem}\label{lem:dirder}
Let both $g\from\P(\consp)\to\Bm(\consp)$ and $\mcl{G}\from\P(\consp)\to\Mb(\consp)$ be continuous and have directional derivatives at $\mu$ in the direction $\et\doteq\mu-\nu$, with $\mu,\nu\in\P(\consp)$. Then
\begin{enumerate}[(i)]
\item $\der{(\exp\circ g)}(\mu;\et) = \exp(g(\mu))\der{g(\mu;\et)}$;\hfill (chain rule)
\item $\der{(g\cdot\mcl{G})}(\mu;\et) = \der{g}(\mu;\et)\mcl{G}(\mu) + g(\mu)\der{\mcl{G}}(\mu;\et)$;\hfill (product rule)
\item $\|\mcl{G}(\mu) - \mcl{G}(\nu)\|_{TV}\leq \|\der{\mcl{G}}(\al\mu+(1-\al)\nu;\et)\|_{TV}$,\hfill (mean value inequality)\\
for some $\al\in[0,1)$.
\end{enumerate}
Moreover, if $\mcl{S}\from\Mb(\consp)\to\Mb(\consp)$ is linear and bounded, for any $\mu\in\Mb(\consp)$ the directional derivative exists in every direction $\et\in\Mb(\consp)$, and $\der{\mcl{S}}(\mu;\et)=\mcl{S}\et$.
\end{lem}

\subsection{Diffusions and related concepts}\label{sec:diffusions}
In this Section, we expose our working hypotheses and necessary results from the theory of diffusion processes. We assume that the process $X_t$ satisfies on the configuration space $\consp$ the stochastic differential equation (SDE)
\begin{equation}\label{eq:sde_int}
X_{t} = \xi + \int_{0}^t a(X_{s})\,\der{s} + \sum_{j=1}^{m}\int_{0}^{t}b^{j}(X_s)\der{W^{j}_t},
\end{equation}
where $\tp{(W_t^1,\ldots,W_t^m)}$ is an $m$-dimensional Wiener process, $\xi\in\consp$ an initial value, and the functions $a\from\consp\to\R^{d}$, $b\from\consp\to\R^{d\times m}$ are given drift and diffusion fluxes. For $j=1,\ldots,m$, the \mbox{$j$-th} column of the $d\times m$ matrix-valued function $b$ is denoted by $b^{j}$. We also fix a time interval $[0,T]$, with $T>0$, on which we want to approximate the particular observable of~\eqref{eq:sde_int} and use the notation $(X_t)_{0:T}$ whenever we consider the process up to time $T$ only.

We assume that the coefficients $a$ and $b$ are time-homogeneous, but extension to the time-dependent case is straightforward. We impose two conditions on the coefficients: bounded differentiability, to guarantee the existence and smoothness of the laws of $X_t$, and uniform ellipticity, which is the simplest assumption to ensure a ``sufficient spreading'' of the randomness:
\begin{ass}\label{ass:smooth_elliptic}
The functions $a$ and $b$ are smooth with all derivatives bounded, and there exists $\ka>0$ such that
\begin{equation*}
\ka|y|^2\leq\tp{y}b(x)\tp{b}(x)y\leq\ka^{-1}|y|^2,
\end{equation*}
for all $x\in\consp$ and $y\in\R^{d}$.
\end{ass}
We refer to~\cite{Baudoin2014,Stroock2008a} for all the results we present in the remainder of this section, which we include to make the manuscript self-contained. In the following, we will denote by $c,C>1$ generic constants that can depend on $T$, $\ka$, $d$, and the bounds on the derivatives of $a$ and $b$. Note that we use the same constants for all the presented estimates. This is legitimate, since we can always increase one or both of them to relax the bounds. In later sections, during computations, we also allow the value of both $c,C$ to change (increase) from line to line.

Assumption~\ref{ass:smooth_elliptic} guarantees that the process $X_t$ is a unique solution to SDE~\eqref{eq:sde_int} for all $t\geq0$ and it admits a smooth transition probability density $p(t,x;\xi)$ -- the likelihood of finding $X_t$ at $x\in\consp$ when starting from $\xi$ at time $0$. Moreover, $p$ satisfies Aronson's estimates: there exists $c,C>1$ such that for all $t>0$ and $x,\xi\in\consp$
\begin{equation}\label{eq:gest}
\frac{C^{-1}}{t^{d/2}}\exp\Big(-\frac{c|x-\xi|^2}{t}\Big)\leq p(t,x;\xi)\leq \frac{C}{t^{d/2}}\exp\Big(-\frac{|x-\xi|^2}{ct}\Big).
\end{equation}
As we detail in Appendix~\ref{sec:estimates_dens}, under an additional assumption on the initial law, the bounds in~\eqref{eq:gest} result in Gaussian lower and upper estimates for the densities of the process $(X_t)_{0:T}$, uniformly in $t$ (see Lemma~\ref{lem:gauss_est}). These, in turn, provide us with a good control of the relative entropy between laws at different times, which we need for the analysis in Sections~\ref{sec:relent_tayl} and~\ref{sec:locerr_extrap}.

In the backward variable $\xi$, the transition densities generate the \emph{diffusion semigroup} given by
\begin{equation}\label{eq:sem_def}
(\sem_{t}f)(\xi) \doteq \int_{\consp}f(x)p(t,x;\xi)\,\der{x} = \Exp[]\big(f(X_t)|X_0=\xi\big),
\end{equation}
for every Borel function $f\from\consp\to\R$ with polynomial growth. In particular, for each $t\geq0$, the mapping $\sem_{t}\from\Bm(\consp)\to\Bm(\consp)$ is a continuous linear contraction with respect to the $\sup$-norm, and $\sem_{t}1=1$. The semigroup $(\sem_{t})_{t\geq0}$ leaves $\conto(\consp)$ invariant and is strongly continuous when restricted to this subspace\footnote{That is
$\lim_{t\searrow0}\|\sem_{t}f - f\|_{\infty}=0$ for every $f\in\conto(\consp)$.}.
The \emph{(infinitesimal) generator} $\gen\from\msr{D}(\gen)\subset\conto(\consp)\to\conto(\consp)$ of $(\sem_{t})_{t\geq0}$ is defined by
\begin{equation}\label{eq:gen_val}
\gen f \doteq \lim_{t\searrow0}t^{-1}(\sem_{t}f-f), 
\end{equation}
with limit taken in sup-norm, and the domain $\msr{D}(\gen)$ being the set of $f\in\conto(\consp)$ for which the limit exists. The space $\ocdiff{2}(\consp)$, of all twice differentiable functions with vanishing derivatives, is a core for $\gen$, on which $\gen$ acts as the second order differential operator
\begin{equation}\label{eq:gen_core}
\gen f = \tp{a}\grad f + \frac{1}{2}\mrm{trace}(b\tp{b}\hess f),\quad f\in\ocdiff{2}(\consp).
\end{equation}

In the forward variable $x$, the transition densities provide the fundamental solution to the \emph{Kolmogorov's forward equation}
\begin{equation}\label{eq:fwd_kolm}
\pder[t]p(t,x;\xi)=[\gen^*p(t,\cdot\,;\xi)](x),\quad \lim_{t\to0}p(t,\cdot\,;\xi)=\de(\xi),
\end{equation}
where $\gen^*$ is the adjoint of $\gen$, with $\msr{D}(\gen^*)$ a subset of $\Mb(\consp)$, the dual of $\conto(\consp)$. Accordingly, the laws of the process $(X_t)_t$ are propagated forward in time by the \emph{adjoint semigroup} $(\sem_{t}^*)_{t\geq0}$, defined via relation
\begin{equation}\label{eq:dual_sem}
\Exp[\sem_{t}^*\!\mu]f = \Exp[\mu][\sem_{t}f],
\end{equation}
for all $\mu\in\P(\consp)$ and $f\in\Bm(\consp)$, see also~\cite[\S 8.1.15]{Bobrowski2005}. The family $(\sem_{t}^*)_{t\geq0}$ can be extended to a conservative semigroup on $\Mb(\consp)$ that leaves positive measures invariant.

\subsection{Euler scheme}\label{sec:em_scheme}
For the analysis of the microscopic step, we approximate $(X_t)_{0:\De\ta}$, on a small time horizon $\De\ta>0$, by the Euler scheme~\eqref{eq:sde_em} on a time mesh $\{t_k=k\de t:\ k=0,\dotsc,K\}$ with $K>1$ time steps $\de t=\De\ta/K$.
The approximate solution $\{\app{X}_{k}:\ k=0,\ldots,K\}$ we obtain is a time-homogeneous Markov chain with $k$-step transition probability kernels $(\xi,B)\mapsto\Pr(\app{X}_k\in B|\app{X}_0=\xi)$~\cite{YuaMao2004}, where $\xi\in\consp,B\in\Bor(\consp)$, which, owing to Assumption~\ref{ass:smooth_elliptic}, have a density, which we denote as $\app{p}(t_k,x;\xi)$, for any $k=1,\dotsc,K$~\cite{LemMen2010}.

Using these kernels, we can define the transition operator and its adjoint
\begin{gather}\label{eq:em_trans}
\begin{aligned}
(\app{\sem}_{t_k}f)(\xi) &\doteq \int_{\consp}f(x)\app{p}(t_k,x;\xi)\,\der{x}, &&f\in\Bm(\consp),\ \xi\in\consp,\\
(\app{\sem}_{t_k}^{*}\mu)(B) &\doteq \int_{\consp}(\app{\sem}_{t_k}\ch_{B})(\xi)\,\der{\mu}(\xi), &&\mu\in\Mb(\consp),\ B\in\Bor(\consp),
\end{aligned}
\end{gather}
where $\ch_{B}$ is the characteristic function of a set $B$ and $k=1,\dotsc,K$. For every probability measure $\mu\in\P(\consp)$, the two operators satisfy relation~\eqref{eq:dual_sem}. Both $\app{\sem}_{t_k}$ and $\app{\sem}_{t_k}^{*}$ are, for each fixed $k$, linear in $f$ and $\mu$ respectively.

In parallel with~\eqref{eq:gest}, we also have the following Gaussian estimates for the transition densities~\cite{LemMen2010}: there exists $c,C>1$ such that for all $k=1,\ldots, K$, and $x,\xi\in\consp$
\begin{equation}\label{eq:gest_em}
\frac{C^{-1}}{t_k^{d/2}}\exp\Big(-\frac{c|x-\xi|^2}{t_k}\Big)\leq \app{p}(t_k,x;\xi)\leq \frac{C}{t_k^{d/2}}\exp\Big(-\frac{|x-\xi|^2}{ct_k}\Big).
\end{equation}
The generic constants $c,C$ are uniform with respect to the discretisation parameter $K$. In later sections, we employ the following sharp estimate in the difference between the transition density of the process $(X_t)_{0:\De\ta}$ and the scheme~\eqref{eq:sde_em}, see~\cite[Thm.~2.3]{GobLab2008}.
\begin{thm}
If Assumption~\ref{ass:smooth_elliptic} holds, then for every $\De\ta_0>0$, there are constants $c,C>1$ such that
\begin{equation}\label{eq:errest_em}
|p(t_k,x;\xi) - \app{p}(t_k,x;\xi)|\leq C\frac{\De\ta}{Kt_k^{(d+1)/2}}\exp\Big(-\frac{c|x-\xi|^2}{t_k}\Big),
\end{equation}
for every $\De\ta\leq\De\ta_0$.
\end{thm}
We use this result in Section~\ref{sec:locerr_extrap} to control the error in TV distance between the densities and the weak error between expectations, see also Appendix~\ref{sec:estimates_dens}.

\section{Minimum relative entropy moment matching}
\label{sec:minxent}
In this Section, we will study the properties of relative entropy $\lre$ (see equation~\eqref{eq:match}) and the minimum relative entropy matching operator, which we denote by $\match$ (see equation~\eqref{eq:matching_intuitive}). We provide a precise definition and characterization of $\match$ in Section~\ref{sec:matchop}, together with an investigation of the continuity and the differentiability on each coordinate in Section~\ref{sec:match-prop}. In particular, we treat directional derivatives with respect to the~prior measure, which constitute a crucial element in the study of the numerical stability of the~micro-macro acceleration method in Section~\ref{sec:stab}. Before that, we introduce the elements that we will use to obtain a convenient description of the matching procedure. In Section~\ref{sec:restr}, we elaborate on the restriction operator and the moment space, to extract the macroscopic variables~\eqref{eq:macvar} and control the feasibility of the statistical constraints. In Section~\ref{sec:expfam}, we discuss exponential families, which will turn out to be convenient to represent the density obtained through the matching.


\subsection{Restriction operator and moment space}\label{sec:restr}
Fix $L\in\N$ and a vector $\mbf{\resf}=\tp{(\resf_1,\ldots,\resf_L)}$ of functions $\resf_l\in\Bm(\consp)$. To accelerate the simulation of SDE~\eqref{eq:sde_int}, we will use the statistical information contained in the vectors $\Exp[\mu]\mbf{\resf}$, where $\mu\in\P(\consp)$ is the law of the solution at some time instance. We formalize this by introducing the \emph{restriction operator}, $\res\from\P(\consp)\to\R^{L}$, generated by $\mbf{\resf}$, that reads
\begin{equation}\label{eq:restr}
\res\mu=\Exp[\mu]\mbf{\resf}.
\end{equation}
The restriction operator is continuous in the weak topology on $\P(\consp)$, and it is linear when extended, in an obvious way, to $\Mb(\consp)$. When $\mu=\Law{X}$, the law of a random variable $X$, formula~\eqref{eq:restr} is consistent with the restriction~\eqref{eq:restr_rv} that was introduced in the algorithmic context of Section~\ref{sec:mM_alg}.

The \emph{moment space} corresponding to $\consp$ and $\mbf{\resf}$ is a convex subset of $\R^{L}$ defined as
\begin{equation}\label{eq:momsp}
\momsp(\consp,\mbf{\resf}) \doteq \img\res = \{\Exp[\mu]\mbf{\resf}:\ \mu\in\P(\consp)\}.
\end{equation}
Whenever the configuration space and the vector of restriction functions are fixed, we write $\momsp=\momsp(\consp,\mbf{\resf})$.
This set will serve to check the feasibility of constraints for optimisation in~\eqref{eq:match}. Obviously, when the vector of macroscopic states $\mbf{m}=(m_1,\ldots,m_l)\in\R^{L}$ does not belong to $\momsp$, we cannot reconstruct a~probability measure having these moments. However, even if $\mbf{m}\in\momsp$, the entropy problem~\eqref{eq:match} need not have a~solution (see~\cite{Junk2000}). The results presented in this Section and Section~\ref{sec:matchop} will demonstrate that $\mbf{m}\in\inter\momsp$ is a~sufficient condition for the existence of the~minimiser to~\eqref{eq:match}, provided the system $\{\resf_1,\ldots,\resf_L\}$ and the prior distribution $\mu$ satisfy the following strengthening of algebraic independence~\cite{Lewis1995}:

\begin{dfn}\label{dfn:psH}
We say that functions $\resf_1,\ldots,\resf_{L}\in\Bm(\consp)$ are \emph{linearly independent modulo $\mu$}
if they are linearly independent on every subset of $\consp$ with positive $\mu$-measure, or, equivalently, if
\begin{equation*}
\mu\big(\{x\in\consp:\ \tp{\bsm{\la}}\mbf{\resf}(x)=0\}\big)=0
\end{equation*}
for all $\bsm{\la}\in\R^{L}\setminus\{\bsm{0}\}$.
\end{dfn}
In particular, when $\consp$ is compact, any linearly independent set of real-analytic functions on $\consp$ will be linearly independent modulo $\mu$. Note also that this property persists whenever we switch to any measure $\nu$ that is absolutely continuous with respect to $\mu$.

With Definition~\ref{dfn:psH} at our disposal, we acquire the following property of the interior of the moment space~$\momsp$, which will turn out to be essential for the definition of the matching operator in Section~\ref{sec:matchop}:
\begin{thm}[{\cite[Thm.~2.9]{BorLew1991b}}]\label{thm:momsp_inter}
Assume that $1,\resf_1,\ldots,\resf_L\in\Bm(\consp)$ are linearly independent modulo a fully supported\footnote{The support of measure $\mu$ is defined as $\supp(\mu)\doteq\{x\in\consp:\ \mu(U)>0\ \text{for each open}\ U\subseteq\consp\ \text{with}\ x\in U\}$, and the measure has full support if $\supp(\mu)=\consp$.} measure $\mu\in\P(\consp)$. For every $\mbf{m}\in\inter\momsp$ there exists a probability measure $\nu\absc\mu$ such, that $\Exp[\nu]\mbf{\resf}=\mbf{m}$ and $\ln(\rnder{\nu}{\mu})\in\leb{\infty}(\mu)$.
\end{thm}
Note that, in the hypothesis of Theorem~\ref{thm:momsp_inter}, we require the system of restriction functions to be independent from the constant function as well. This is the natural situation in our framework. As we are working with probability measures, the constant statistics do not bring new information, and any linear dependence of components of $\mbf{\resf}$ on constants makes the vector of expectations reducible.

In practice, one would consider a fixed, dominating measure on $\consp$, e.g.~the Lebesgue measure $\Leb$, and choose $1,\resf_1,\dotsc,\resf_L$ linearly independent modulo $\Leb$. See also the final paragraph of Section~\ref{sec:matchop}. Then, the conclusion of Theorem~\ref{thm:momsp_inter} holds for all fully supported measures $\mu\absc\Leb$.

We finish this Section with a general description of the moment space:
\begin{lem}[{\cite[Thm.~2.1]{Lewis1995}}]\label{lem:momsp_prop}
If $\mu\in\P(\consp)$ has full support, and $1,\resf_1,\ldots,\resf_L\in\Bm(\consp)$ are linearly independent modulo $\mu$, the following relations hold:
\begin{enumerate}[(i)]
\item $\inter\momsp=\{\Exp[\nu]\mbf{\resf}:\ \nu\in\P(\consp)\ \text{and}\ \nu\absc\mu\}\neq\emptyset$;
\item $\bdy\momsp\subseteq\{\Exp[\nu]\mbf{\resf}:\ \nu\in\P(\consp)\ \text{and}\ \nu\sing\mu\}$.
\end{enumerate}
\end{lem}


\subsection{Exponential families}
\label{sec:expfam}
For a vector $\bsm{\la}\in\R^{L}$ and a measure $\mu\in\P(\consp)$, with a fixed vector $\mbf{\resf}=\tp{(\resf_1,\ldots,\resf_L)}$ of moment functions $\resf_l\in\Bm(\consp)$, define
\begin{equation*}
A(\bsm{\la},\mu) \doteq \ln Z(\bsm{\la},\mu) \doteq \ln\Exp[\mu]\big[e^{\tp{\bsm{\la}}\mbf{\resf}}\big].
\end{equation*}
We call $Z$ the \emph{partition function} and $A$ the \emph{log-partition function}.
For fixed $\mu\in\P(\consp)$, the log-partition function determines a family of probability distributions that reads
\begin{equation*}
\expd(\bsm{\la},\mu) = \exp\!\big(\tp{\bsm{\la}}\mbf{\resf} - A(\bsm{\la},\mu)\big)\cdot\mu\in\P(\consp),\quad \bsm{\la}\in\R^{L}.
\end{equation*}
The function $\bsm{\la}\mapsto\expd(\bsm{\la},\mu)$ is called the~\emph{exponential family with respect to $\mu$}~\cite{AmaNag2000,LehRom2005}.
\begin{lem}\label{lem:lpart}
\begin{enumerate}[(i)]
\item\label{lem:lpart_lam} For each $\mu\in\P(\consp)$, the function $\bsm{\la}\mapsto A(\bsm{\la},\mu)$ is convex and smooth on $\R^{L}$, with
\begin{equation*}
\grad[\bsm{\la}]A(\bsm{\la},\mu)=\Exp[\expd(\bsm{\la},\mu)]\mbf{\resf},\quad \hess[\bsm{\la}]A(\bsm{\la},\mu)=\Var[\expd(\bsm{\la},\mu)](\mbf{\resf}).
\end{equation*}
\item\label{lem:lpart_mu} For every $\bsm{\la}\in\R^{L}$, the function $\mu\mapsto A(\bsm{\la},\mu)$ is concave and weakly continuous on $\P(\consp)$.
\item\label{lem:lpart_comb} The mapping $(\bsm{\la},\mu)\mapsto A(\bsm{\la},\mu)$ is continuous on $\R^{L}\times\P(\consp)$ with $\|\cdot\|\times{weak}$ topology.
\end{enumerate}
\end{lem}
\begin{proof}
Item~\eqref{lem:lpart_lam} follows from differentiation under the integral sign, valid due to the Lebesgue dominated convergence theorem. The first two derivatives of partition function read
\begin{equation*}
\pder[\la_l]Z(\bsm{\la},\mu) = \Exp[\mu]\big[\resf_l\cdot e^{\tp{\bsm{\la}}\mbf{\resf}}\big],\quad \pder[\la_l\la_k]Z(\bsm{\la},\mu) = \Exp[\mu]\big[\resf_l\resf_k\cdot e^{\tp{\bsm{\la}}\mbf{\resf}}\big],
\end{equation*}
from which the formulas for the gradient and the Hessian of $A$ follow. The details can be found, for example, in~\cite[Sec.~2.7]{LehRom2005}.

The proof of claim~\eqref{lem:lpart_mu} is straightforward.

The conclusion of item~\eqref{lem:lpart_comb} follows from the estimate
\begin{equation*}
|Z(\bsm{\la}_n,\mu_n) - Z(\bsm{\la},\mu)|\leq 
\|e^{\tp{\bsm{\la}_n}\mbf{\resf}}-e^{\tp{\bsm{\la}}\mbf{\resf}}\|_{\infty} + |\bra e^{\tp{\bsm{\la}}\mbf{\resf}},\mu_n-\mu\ket|.
\end{equation*}
Thus if $(\bsm{\la}_n,\mu_n)\to(\bsm{\la},\mu)$, we have $Z(\bsm{\la}_n,\mu_n)\to Z(\bsm{\la},\mu)$, and the same holds for the log-partition function $A$. The sequential continuity implies the continuity in $\R^{L}\times\P(\consp)$, due to the metrizability of the weak topology on $\P(\consp)$~\cite[Thm.~11.3.3.]{Dudley2002}.
\end{proof}

Note that the measures $\expd(\bsm{\la},\mu)$ and $\mu$ are equivalent, the Radon-Nikodym derivative of $\expd(\bsm{\la},\mu)$ with respect to $\mu$ is $\rnder{\expd(\bsm{\la},\mu)}{\mu}=\exp\!\big(\tp{\bsm{\la}}\mbf{\resf}-A(\bsm{\la},\mu)\big)\in\leb{\infty}(\mu)$ with norm bounded by $\exp\!\big(2\|\bsm{\la}\|\!\cdot\!\|\mbf{\resf}\|_{\infty}\big)$. 
According to Lemma~\ref{lem:lpart}, this density is differentiable in $\bsm{\la}$ and, by the chain rule, we have a simple estimate on this derivative, which we will need in Section~\ref{sec:locerr_extrap}:
\begin{lem}\label{lem:matchdens_bound}
For any fixed $\bsm{\la}\in\R^L$ and $\mu\in\P(\consp)$
\begin{equation*}
\Big\|\grad[\bsm{\la}]\rnder{\expd(\bsm{\la},\mu)}{\mu}\Big\|_{\infty}\leq 2\|\mbf{\resf}\|_{\infty}\,e^{2\|\bsm{\la}\|\cdot\|\mbf{\resf}\|_{\infty}}\leq2\|\mbf{\resf}\|_{\infty}e^{\|\mbf{\resf}\|^2_{\infty}}\cdot e^{\|\bsm{\la}\|^2}.
\end{equation*}
\end{lem}
One nice feature the assumption of linear independence modulo $\mu$ (Definition~\ref{dfn:psH}) guarantees is the invertibility of the Hessian matrix of the log-partition function.
\begin{lem}\label{lem:hess_nonsing}
If the functions $\resf_1,\ldots,\resf_L$ are linearly independent modulo $\mu$, the Hessian $\hess[\bsm{\la}]A(\bsm{\la},\mu)$ is positive definite.
\end{lem}
\begin{proof}
We can assume (up to changing $\mbf{\resf}$ to $\mbf{\resf}-\Exp[\expd(\bsm{\la},\mu)]\mbf{\resf}$) that $\Exp[\expd(\bsm{\la},\mu)]\mbf{\resf}=\bsm{0}$, so the Hessian is
\begin{equation*}
\hess[\bsm{\la}]A(\bsm{\la},\mu)=\Exp[\expd(\bsm{\la},\mu)]\big[\mbf{\resf}\tp{\mbf{\resf}}\big].
\end{equation*}
Take a vector $\mbf{v}\in\R^{L}$. The variance-covariance matrix is always positive-semidefinite so we already know that $\tp{\mbf{v}}\Exp[\expd(\bsm{\la},\mu)][\mbf{\resf}\tp{\mbf{\resf}}]\mbf{v}\geq0$. Suppose now that this form is equal to zero. By the linearity of expectation, this reads as
\begin{equation*}
\Exp[\expd(\bsm{\la},\mu)]\Big[\Big(\sum_{l=1}^{L}v_l\resf_l\Big)^{\!2}\Big] = 0
\end{equation*}
Since the exponential distribution $\expd(\bsm{\la},\mu)$ is a probability measure equivalent to $\mu$, this equality can hold only if $\mu\big(\sum_{l=1}^{L}v_l\resf_l=0\big)=1$ and we get a contradiction with the linear independence modulo $\mu$.
\end{proof}

Finally, we find the directional derivatives of the log-partition function with respect to the underlying measure.
\begin{lem}\label{lem:lpart_dirder}
For each fixed $\bsm{\la}\in\R^{L}$ the function $\P(\consp)\ni\mu\mapsto A(\bsm{\la},\mu)$ has the directional derivative $\pder[\mu]A(\bsm{\la},\mu;\et)$, see Definition~\ref{dfn:dirder}, in every admissible direction $\et\in\Mb(\consp)$ for $\mu$, with
\begin{equation}\label{eq:lpart_dirder}
\pder[\mu]A(\bsm{\la},\mu;\et) = \big\bra\exp\!\big(\tp{\bsm{\la}}\mbf{\resf}-A(\bsm{\la},\mu)\big)\big|\et\big\ket.
\end{equation}
\end{lem}
\begin{proof}
On one hand, since the functional $\mu\mapsto Z(\bsm{\la},\mu)$ extends linearly to $\Mb(\consp)$, we have
\begin{equation*}
\pder[\mu]Z(\bsm{\la},\mu;\et) = \bra e^{\tp{\bsm{\la}}\mbf{\resf}}|\et\ket.
\end{equation*}
On the other hand, we compute
\begin{align*}
\pder[\mu]Z(\bsm{\la},\mu;\et) 
&= \lim_{\vep\searrow0}\frac{\exp\!A\,(\bsm{\la},\mu+\vep\et)- \exp\!A\,(\bsm{\la},\mu)}{\vep} \\[0.7em]
&=\exp\big(A(\bsm{\la},\mu)\big) \lim_{\vep\searrow0}\frac{\exp\big( A(\bsm{\la},\mu+\vep\et) - A(\bsm{\la},\mu)\big) - 1}{\vep}\\[0.7em]
&= \exp\big(A(\bsm{\la},\mu)\big) \lim_{\vep\searrow0}\frac{A(\bsm{\la},\mu+\vep\et) - A(\bsm{\la},\mu)}{\vep} = \exp\big(A(\bsm{\la},\mu)\big)\cdot \pder[\mu]A(\bsm{\la},\mu;\et).
\end{align*}
The limits exist according to the concavity of $A$ in the second variable. From these two formulas, we obtain~\eqref{eq:lpart_dirder}.
\end{proof}

\subsection{Definition of matching operator}\label{sec:matchop}
In this Section, we combine the results from Sections~\ref{sec:restr} and~\ref{sec:expfam} to define and characterize the matching operator based on the minimisation of relative entropy.
We begin with the definition of relative entropy:
\begin{dfn}\label{dfn:lre}
The \emph{(logarithmic) relative entropy} of a measure $\nu\in\P(\consp)$ with respect to a measure~$\mu\in\P(\consp)$ is given by
\begin{equation*}
\lre(\nu\|\mu)=\left\{\begin{array}{cl}
\displaystyle\Exp[\mu]\!\Big[\rnder{\nu}{\mu}\ln\rnder{\nu}{\mu}\Big], & \text{if}\ \nu\absc\mu,\\[1em]
+\infty, & \text{otherwise}.
\end{array}\right.
\end{equation*}
\end{dfn}
The boundedness from below of the function $s\mapsto s\ln s$ guarantees that the expectation is well defined, even though its value may be infinite. The convexity of $s\ln s$ yields $\lre(\nu\|\mu)\geq0$ for all $\nu,\mu\in\P(\consp)$, with equality if and only if $\nu=\mu$ (this follows from Jensen's inequality). However, even if $\nu\equi\mu$, the two entropies $\lre(\nu\|\mu)$ and $\lre(\mu\|\nu)$ are not equal in general. The function $\lre(\cdot\|\cdot)$ is convex on the product $\P(\consp)\times\P(\consp)$, but the triangle inequality does not hold for $\lre$. The lack of usual properties associated with metric functions makes the study of the geometry induced by $\lre$ on $\P(\consp)$ more involved.

Before proceeding, let us elaborate on some conditions for the finiteness of the relative entropy. Note first, that the absolute continuity $\nu\absc\mu$ is necessary but not sufficient, as can be seen by taking $\mu\sim\exp(-1/x)\Leb$ and $\nu=\Leb$ on $(0,1)$. Since $\rnder{\nu}{\mu}\cdot\mu=\nu$, by changing the integration in Definition~\ref{dfn:lre}, we see that the necessary and sufficient condition is 
\begin{equation}\label{eq:cond_nu}
\Exp[\nu]\big[\ln\rnder{\nu}{\mu}\big]<+\infty.
\end{equation} 
However, the condition~\eqref{eq:cond_nu} involves an expectation with respect to $\nu$, while we are interested in the expectations with respect to the prior measure $\mu$. In this direction, the following Lemma, of which the proof follows easily from the H\"{o}lder inequality, gives a simple sufficient condition for the finiteness of the relative entropy:
\begin{lem}\label{lem:lre_fin}
If $\nu\absc\mu$, $\Exp[\mu][\rnder{\nu}{\mu}^2]<+\infty$, and $\Exp[\mu][\ln^2\rnder{\nu}{\mu}]<+\infty$, the relative entropy $\lre(\nu\|\mu)$ is finite. In particular, if $\ln\rnder{\nu}{\mu}$ is bounded, the relative entropy is finite.
\end{lem}

Let us now recall two well-known facts about optimal solutions for the minimisation of relative entropy, as in~\eqref{eq:match}. The first result provides a sufficient condition for the existence of an optimal solution and is a consequence of~\cite[Thm.~2.1]{Csiszar1975}. By $\invimg{\res}(\cdot)$ we denote the inverse image of $\res$.
\begin{pro}\label{pro:minxent_exist}
Let $\mu\in\P(\consp)$, $\resf_1,\ldots,\resf_L\in\Bm(\consp)$ and $\mbf{m}\in\momsp(\consp,\mbf{\resf})$. If there is $\nu\in\invimg{\res}(\{\mbf{m}\})$ such, that $\lre(\nu\|\mu)<+\infty$, there exists a unique measure $\mu^*\in\P(\consp)$ such that
\begin{equation}\label{eq:minxent_distr}
\mu^* = \argmin_{\nu\in\res^{\leftarrow}(\{\mbf{m}\})}\lre(\nu\|\mu).
\end{equation}
\end{pro}

The next result gives an explicit formula for the density, with respect to the prior $\mu$, of every measure that minimises the relative entropy with moment constraints~\cite[Thm.~3.1]{Csiszar1975}.
\begin{pro}\label{pro:minxent_dens}
Let $\mu\in\P(\consp)$, $\resf_1,\ldots,\resf_L\in\Bm(\consp)$ and $\mbf{m}\in\momsp(\consp,\mbf{\resf})$. If $\mu^*\in\P(\consp)$ fulfils~\eqref{eq:minxent_distr}, it reads
\begin{equation}\label{eq:minxent_dens}
\mu^* = \expd(\bsm{\la}^{\!*},\mu)=\exp\!\big(\tp{(\bsm{\la}^{\!*})}\mbf{\resf} - A(\bsm{\la}^{\!*},\mu)\big)\cdot\mu,
\end{equation}
where $\bsm{\la}^{\!*}\in\R^{L}$ satisfies
\begin{equation}\label{eq:minxent_lagr}
\grad[\bsm{\la}]A(\bsm{\la}^{\!*},\mu) = \mbf{m}.
\end{equation}
Moreover, the minimized value of relative entropy is
\begin{equation}\label{eq:minxent_lre}
\lre(\mu^*\|\mu) = \tp{(\bsm{\la}^{\!*})}\mbf{m} - A(\bsm{\la}^*,\mu).
\end{equation}
\end{pro}
Here, we can see the connection between optimal solutions of entropy minimisation and the exponential families of Section~\ref{sec:expfam}. The parameters $\bsm{\la}^{\!*}$ are obtained as the solution to the first-order optimality system~\eqref{eq:minxent_lagr}, and thus we will call them Lagrange multipliers from this point on. The assumptions in Proposition~\ref{pro:minxent_dens} do not lead to the uniqueness of the multipliers $\bsm{\la}^{\!*}$, but we can guarantee this via Lemma~\ref{lem:hess_nonsing} by imposing the linear independence modulo $\mu$ on the restriction functions.

Now that we have discussed the properties of relative entropy, we are ready to rigorously define the matching operator that we have intuitively introduced in~\eqref{eq:matching_intuitive}. We aim at defining $\mathcal{M}$ as an operator acting on the pairs $(\mbf{m},\mu)$, where $\mbf{m}$ is a given vector of moments (macroscopic state), and $\mu$ is a prior probability measure. We predetermined that the result of matching will be given by the optimal solution to~\eqref{eq:minxent_distr}, and Proposition~\ref{pro:minxent_dens} yields the exponential form for the matched distribution. What is left, is to depict an admissible set for the pairs we can match -- the domain of the operator. To this end, by Proposition~\ref{pro:minxent_exist}, it suffices to establish the existence of a probability measure $\nu$ with (i) $\Exp[\nu]\mbf{\resf}=\mbf{m}$ and (ii) $\lre(\nu\|\mu)<+\infty$. The first condition is just $\mbf{m}\in\momsp(\consp,\mbf{\resf})$. The second will be true if, additionally, $\mbf{m} \in \inter\momsp(\supp(\mu),\mbf{\resf})$, and $\{1,\resf_1,\ldots,\resf_L\}$ is independent modulo $\mu$. Indeed, Theorem~\ref{thm:momsp_inter} guarantees in this case the existence of $\nu\absc\mu$ having the right moments (macroscopic state) and with bounded $\ln\rnder{\nu}{\mu}$. This, together with Lemma~\ref{lem:lre_fin}, gives the finiteness of relative entropy. These considerations lead to the following definition:

\begin{dfn}[Matching operator]\label{dfn:matchop}
Consider a restriction vector $\mbf{\resf}=\tp{(\resf_1,\ldots,\resf_L)}$ such, that $\resf_l\in\Bm(\consp)$ for each $l=1,\ldots,L$. We define the \emph{matching operator} $\match\from\dmatch(\consp,\mbf{\resf})\to\P(\consp)$ with
\begin{align*}
\dmatch(\consp,\mbf{\resf}) 
&= \left\{(\mbf{m},\mu):\begin{array}{c}
\{1,\resf_1,\ldots,\resf_L\}\ \text{is independent mod}\ \mu\\[0.5em]
\text{and}\ \mbf{m}\in\inter\momsp\big(\supp(\mu),\mbf{\resf}\big)
\end{array}\quad \text{or}\quad \mbf{m} = \Exp[\mu][\mbf{\resf}]\right\}\\[1em]
\match(\mbf{m},\mu) 
&= \argmin_{\nu\in\invimg{\res}(\{\mbf{m}\})}\lre(\nu\|\mu) = \expd(\bsm{\la}(\mbf{m},\mu),\mu),
\end{align*}
where $\bsm{\la}(\mbf{m},\mu)\in\R^{L}$ satisfies
\begin{equation}\label{eq:match_lagr}
\grad[\bsm{\la}]A(\bsm{\la}(\mbf{m},\mu),\mu) = \mbf{m}.
\end{equation}
\end{dfn}
While $\dmatch(\consp,\mbf{\resf})$ succinctly gives a range of possible pairs $(\mbf{m},\mu)$ we can match, it posits requirements on the measure $\mu$ that can be cumbersome to check. Particularly, the independence in the first condition requires $\mu$ to be an atomless measure, since for any atom $x\in\consp$ of $\mu$ the system of numbers $\{1,\resf_1(x),\dotsc,\resf_{L}(x)\}$ cannot be linearly independent. Moreover, even if $\mu$ is atomless, it is not always true that $\momsp\big(\supp(\mu),\mbf{\resf}\big)$ has non-empty interior, as $\mu$ can be concentrated on a lower dimensional submanifold of $\consp$. For this reason, we additionally include all pairs $(\mbf{m},\mu)$ that $\mbf{m}=\Exp[\mu][\mbf{\resf}]$ in the definition of the matching domain $\dmatch(\consp,\mbf{\resf})$, with the matching $\match(\mbf{m},\mu)$ being equal to $\mu$ in this case.

In some cases, we can restrict the range of prior measures to simplify the situation. 
For example, if the system $\{1,\resf_1,\ldots,\resf_L\}$ is independent modulo a positive, ``dominating'' atomless measure $\mu_0$ with full support, such as the Lebesgue or Gaussian measure on $\consp$, it is linearly independent modulo all probability distributions $\mu\absc\mu_0$. Thus, we have the inclusion
\begin{equation}\label{eq:matchdom_simple}
\inter\momsp(\consp,\mbf{\resf})\times\{\mu\in\P(\consp):\ \mu\absc\mu_0\ \text{and}\ \supp(\mu)=\consp\}\subset\dmatch(\consp,\mbf{\resf}),
\end{equation}
and the set on the left is easier to work with, since we decoupled the moment condition from the priors. This is the setting we exploit in the remainder of the paper, where the prior measures are the time marginal distributions of a diffusion process, and the standard assumptions on the coefficients of SDE~\eqref{eq:sde_int} guarantee the absolute continuity with respect to $\mu_0=\Leb$, as well as the positivity of their densities with respect to the Lebesgue measure.


\subsection{Properties of the matching operator\label{sec:match-prop}}
In the final part of this Section, we gather the continuity and differentiability properties of the matching operator from Definition~\ref{dfn:matchop}. We fix $\consp$ and $\mbf{\resf}=\tp{(\resf_1,\ldots,\resf_L)}$ such, that $\resf_l\in\Bm(\consp)$ for $l=1,\ldots,L$, and denote $\dmatch\doteq\dmatch(\consp,\mbf{\resf})$, which we consider as a subset of $\R^{L}\times\P(\consp)$. The proofs are presented in Appendix~\ref{sec:match-prop_proofs}.

First, we consider continuity and differentiability with respect to the macroscopic state and the prior measure: 
\begin{thm}\label{thm:matchop_prop}The matching operator $\match$, from Definition~\ref{dfn:matchop}, has the following properties:
\begin{enumerate}[(i)]
\item\label{thm:matchop_prop_pyth}
For any $(\mbf{m},\mu)\in\dmatch$ and $\nu\absc\mu$ such, that $\Exp[\nu]\mbf{\resf}=\mbf{m}$, we have the Pythagorean identity
\begin{equation*}
\lre(\nu\|\mu) = \lre(\nu\|\match(\mbf{m},\mu)) + \lre(\match(\mbf{m},\mu)\|\mu).
\end{equation*}
\item\label{thm:matchop_prop_cont}
The mapping $(\mbf{m},\mu)\mapsto\match(\mbf{m},\mu)\in\P(\consp)$ is \mbox{$\|\cdot\|\times weak$} to~$weak$
continuous on $\dmatch$, and the functionals $(\mbf{m},\mu)\mapsto\bsm{\la}(\mbf{m},\mu)$ and $(\mbf{m},\mu)\mapsto\lre\big(\match(\mbf{m},\mu)\|\mu\big)$ are \mbox{$\|\cdot\|\times weak$} continuous on $\dmatch$.
\item\label{thm:matchop_prop_momdiff}
For every probability measure $\mu\in\P(\consp)$ for which the system $\{1,\resf_1,\ldots,\resf_L\}$ is linearly independent modulo $\mu$, the function $\mbf{m}\mapsto\bsm{\la}(\mbf{m},\mu)$ is differentiable on $\inter\momsp\big(\supp(\mu),\mbf{\resf}\big)$ with
\begin{equation}\label{eq:lagr_sder_m}
\sder[\mbf{m}]\bsm{\la}(\mbf{m},\mu)=\Big(\hess[\bsm{\la}]A\big(\bsm{\la}(\mbf{m},\mu),\mu\big)\Big)^{-1}. 
\end{equation}
\item\label{thm:matchop_prop_priordiff}
For every 
$(\mbf{m},\mu),(\mbf{m},\nu)\in\dmatch$, the directional derivative of $\bsm{\la}(\mbf{m},\mu)$ exists in the direction $\et=\nu-\mu$ and
\begin{equation}\label{eq:lagr_dirder}
\pder[\mu]\bsm{\la}(\mbf{m},\mu;\et) = -\sder[\mbf{m}]\bsm{\la}(\mbf{m},\mu)\big\bra \exp\!\big(\tp{\bsm{\la}(\mbf{m},\mu)}\mbf{\resf}-A\big(\bsm{\la}(\mbf{m},\mu),\mu\big)\big)\big(\mbf{\resf}-\mbf{m}\big)\big|\et\big\ket.
\end{equation}
\end{enumerate}
\end{thm}

In the following Theorem, we establish the continuity of $\match$ in TV norm with respect to the prior distribution. 
\begin{thm}\label{thm:lip_match}
Let $\Ga\subseteq\dmatch$ be compact in the \mbox{$\|\cdot\|\times weak$} topology on $\R^{L}\times\P(\consp)$. Then, there exists a constant $C=C(\Ga,\mbf{\resf})$ such, that
\begin{equation}\label{eq:lip_match}
\|\match(\mbf{m},\mu_1)-\match(\mbf{m},\mu_2)\|_{TV}\leq C\|\mu_1-\mu_2\|_{TV},
\end{equation}
for all $(\mbf{m},\mu_i)\in\Ga$, $i=1,2$.
\end{thm}
In Section~\ref{sec:locerr}, we will need the Lipschitz condition~\eqref{eq:lip_match} to control the error due to the difference in prior measures, while the moments are kept fixed.

\section{The convergence result}\label{sec:conv}
In this Section, we formulate the remaining assumptions (on top of those made in Section~\ref{sec:prelim}) that we use to prove convergence of the micro-macro acceleration method in the limit when the macroscopic time step $\Delta t$ tends to zero and the number of extrapolated moments $L$ tends to infinity. First, we consider an initial random variable $X_0$ with law $\mu_0$, satisfying the following assumption:
\begin{ass}\label{ass:init_gauss_est}
The probability measure $\mu_0$ is absolutely continuous with respect to the Lebesgue measure and satisfies
\begin{equation*}
C^{-1}\exp(-c|x|^2)\leq\rnder{\mu_0}{x}(x)\leq C\exp(-|x|^2/c),
\end{equation*}
for some constants $c,C>1$.
\end{ass}
Recall that, if $x\in\T^{d}$, $|x|$ measures the distance from the equivalence class of the lattice points, and in this case Assumption~\ref{ass:init_gauss_est} is equivalent to requiring the boundedness of $\ln(\rnder{\mu_0}{x})$. Appendix~\ref{sec:estimates_dens} contains the proofs of some properties that result from Assumption~\ref{ass:init_gauss_est} in the case $\consp=\R^{d}$. 

We also recall that $(X_t)_{0:T}$ denotes the exact solution of equation~\eqref{eq:sde_int} with initial condition $X_0$ on a fixed (macroscopic) time horizon $T>0$. Now, let us discuss the conditions that we impose on the building blocks of Algorithm~\ref{alg:accel}.

On the macroscopic level, we use $N(\De t)=\min\{N\in\N:\ N\De t\geq T\}$ steps of Algorithm~\ref{alg:accel}. For the microscopic simulation, we employ the Euler scheme~\eqref{eq:sde_em} with $K$ steps of size $\de t\ll\De t$, and denote $\De\ta =K\de t$. The analysis of convergence requires a consistent way of building restriction operators with an increasing number of macroscopic variables. To this end, we consider a sequence $\{\resf_l:\ l\geq 1\}$ of functions $\resf_l\from\consp\to\R$, which serve as the macroscopic state variables, and consider a hierarchy of restriction operators defined as follows
\begin{equation}\label{eq:res_hier}
\res_{L}\mu\doteq\Exp[\mu][\mbf{\resf}_L],\quad \mbf{\resf}_L=\tp{(\resf_1,\ldots,\resf_L)},\ L\geq 1.
\end{equation}

\begin{ass}\label{ass:res_hier}
The restriction functions $\resf_l\in\ocdiff{2}(\consp)$, $l\geq1$, satisfy the conditions:
\begin{enumerate}[(i)]
\item
the sequence $\|\mbf{\resf}_L\|_{2,\infty} =\sum_{l=1}^{L}\|\resf_l\|^2_{2,\infty}$ is bounded in $L$,
\item
the system $\{1,\resf_1,\dotsc,\resf_L\}$ is independent modulo the Lebesgue measure on $\consp$ for all $L\geq 1$,
\item
the (infinite sequence of) moments $\{\Exp[\mu(t)][\resf_1],\Exp[\mu(t)][\resf_2],\dotsc\}$ uniquely determine the exact solution $\mu(t)=\sem_{t}^*\mu_0$, for all $0\leq t\leq T$,
\end{enumerate}
\end{ass}
In Remark~\ref{rem:scal} below, we show that, by rescaling the restriction functions, item $(i)$ can be relaxed to requiring only uniform boundedness in $l$ of the norms $\|\resf_l\|_{2,\infty}$. 
Item $(ii)$ ensures that we can match with any number $L$ of macroscopic variables (see Definition~\ref{dfn:matchop}), and $(iii)$ guarantees that we can approximate the laws $\mu(t)$ of the solution $X_t$ by matching with exact moments $\res_{L}\mu(t)$ as $L$ tends to $+\infty$. (We will use this property in Section~\ref{sec:locerr_cons}). In the case $\consp=\T^{d}$, since $\cont(\T^d)$ is separable, a sufficient condition to guarantee $(iii)$ is that the sequence $\{\resf_l:\ l\geq1\}$ is dense in $\cont(\T^d)$~\cite{Kruk2004}. The space $\contb(\R^{d})$ is not separable, so this argument does not hold when $\consp=\R^{d}$. In this case, we can resort to other functional spaces. For example, due to Assumption~\ref{ass:smooth_elliptic}, we know that $\mu(t)\absc\mu_0$ and if the densities satisfy $\p(t)\in\leb{2}(\R^{d},\mu_0)$, we can choose $\resf_l$ to constitute an orthogonal basis of $\leb{2}(\R^{d},\mu_0)$. Then $(iii)$ follows from the uniqueness of the Fourier coefficients.

\begin{rem}(Power moments)
The set of possible restriction functions that are allowed under Assumption~\ref{ass:res_hier} is quite restricted when the configuration space is $\R^d$. Consider, for instance, a typical one-dimensional setting in which the restriction functions are given by $\resf_l(x)=x^l/l$, $l=1,2,\dotsc$. This hierarchy of functions considered on the torus $\T$, defined by identification of $\T$ with $(0,1]$, satisfies the conditions in Assumption~\ref{ass:res_hier} (see also Remark~\ref{rem:scal}). However, these functions are not bounded on $\R$, and thus are not encompassed by Assumption~\ref{ass:res_hier} and the results of this manuscript. The unboundedness of restriction functions poses new challenges in the analysis of the micro-macro acceleration method, mainly because the existence of the matching and its properties are much harder to establish, see for example~\cite{Junk2000}. The extension of the results of this paper to such cases turns out to require additional analysis of the properties of relative entropy itself, and is therefore left for future research.
\end{rem}

To simplify the notation, we omit the index $L$ from the restrictions \eqref{eq:res_hier}, whenever this number is fixed. With this in mind, we now define the extrapolation operator used in the remainder of the paper.
\begin{dfn}\label{dfn:extr_proj}
For a given restriction operator $\res$, the \emph{extrapolation over $\De t$}, with $0<\De\ta\leq\De t$, of the moments of an initial law $\mu\in\P(\consp)$ is given by
\begin{equation}\label{eq:extr}
\mbf{m}(\De t,\De\ta,\mu)\doteq \res\mu+\De t\frac{\res(\app{\sem}_{\!\De\ta}^{*}\mu) - \res\mu}{\De\ta}.
\end{equation}
Moreover, let $\comp\subset\R^{L}$ be a closed convex set such that
\begin{equation*}
\big\{\res\big(\sem_{t}^{*}\mu_0\big):\ 0\leq t\leq T\big\}\subset\inter\comp\subset\comp\subset\inter\momsp(\consp,\mbf{\resf}),
\end{equation*}
and let $\proj{\comp}$ be the metric projection on $\comp$. The \emph{projected extrapolation} reads
\begin{equation}\label{eq:extr_proj}
\mbf{m}_{\comp}(\De t,\De\ta,\mu)\doteq \proj{\comp}\big(\mbf{m}(\De t,\De\ta,\mu)\big).
\end{equation}
\end{dfn}
The projection onto the set $\comp$ is a technical assumption, related to the moment problem -- we can match only when the macroscopic states belong to $\inter\momsp(\consp,\mbf{\resf})$. Without the projection, the linear extrapolation $\mbf{m}(\De t,\De\ta,\mu)$ does not necessarily respect this constraint in general. However, when the extrapolation step $\De t$ becomes small enough, we will have $\mbf{m}(\De t,\De\ta,\mu)\in\comp$, if $\mu\in\inter\comp$. In consequence, since we require that all the moments of the exact solution belong to the interior of $\comp$, the projection part of extrapolation~\eqref{eq:extr_proj} becomes less relevant in the limit $\De t\to0$, which is the focus of our analysis of convergence. It is also clear that we can always make $\comp$ compact, by intersecting it with a large enough ball, and we detail how to fix $\comp$ in Section~\ref{sec:inop}.

\begin{rem}[Adaptive extrapolation step]\label{rem:adapt}
Note that we can also consider the method with variable macroscopic step $\De t$. This would make the projection $\proj{\comp}$ redundant, as we already pointed out, and is equivalent to~\eqref{eq:extr_proj} in the limit when $\De t$ tends to $0$. Algorithm~\ref{alg:accel} with adaptive time stepping is also more practical for actual simulations, since we cannot always guarantee the separation of time scales during the entire simulation. For more on this issue and an implementation with a criterion for the selection of an appropriate step size $\Delta t$, we refer to~\cite{DebSamZie2017}.
\end{rem}
\begin{rem}[Scaling restriction functions]\label{rem:scal}
Suppose that instead of item $(i)$ in Assumption~\ref{ass:res_hier}, we have only a constant $C>0$ such that $\|\resf_l\|_{2,\infty}\leq C$ for all $l\geq1$. Then, setting $\otilde{\resf}_l=\resf_l/l$ we get
\begin{equation*}
\sum_{l=1}^{L}\|\otilde{\resf}_l\|^2_{2,\infty}\leq C^2\sum_{l=1}^{L}\frac{1}{l^2}\leq 2C^2,
\end{equation*}
so the system $\{\otilde{\resf}_l\}$ satisfies the condition in item $(i)$. Such scaling does not have an impact on the matching procedure; if the vectors $\otilde{\mbf{m}}$ and $\mbf{m}$ are related by $\otilde{m}_l=m_l/l$, the constraints $\Exp[\nu][\otilde{\mbf{\resf}}]=\otilde{\mbf{m}}$ and $\Exp[\nu][\mbf{\resf}]=\mbf{m}$ generate the same set of probability measures, and thus the matching $\otilde{\match}(\mbf{m},\mu)$, based on $\otilde{\mbf{\resf}}$, gives the same results as $\match(\mbf{m},\mu)$. As the extrapolations $\otilde{\mbf{m}}(\De t,\De\ta,\mu)$ and $\mbf{m}(\De t,\De\ta,\mu)$, given by~\eqref{eq:extr}, are also related by the same scaling, we see that this procedure does not affect the output of Algorithm~\ref{alg:accel}.
\end{rem}

Our strategy to demonstrate convergence of the micro-macro acceleration method can be briefly described as follows. In Section~\ref{sec:stab}, we perform a forward error analysis by studying the propagation of local errors in the TV distance. We obtain a Lipschitz estimate for the one-step propagator of the micro-macro acceleration scheme that allows controlling the accumulation of local errors. This constitutes the numerical stability of the method. Then, in Section~\ref{sec:locerr}, we investigate the limiting behaviour of local errors when $\Delta t$ tends to zero and $L$ tends to infinity. A crucial step in this process involves replacing, through Pinsker's inequality, the TV distance between the law of $X_{t_n}$ and the matching $\match(\res_L(X_{t_n}),X_{t_{n-1}})$ by the relative entropy of these two distributions. Therefore, we first proceed to the study of this particular relative entropy in Section~\ref{sec:expansion}.

During the analysis, it will turn out that some additional assumptions are required, on which we briefly comment below. We then obtain the following theorem, which gives the exact statement of convergence that we will prove in the remaining part of the manuscript:
\begin{thm}\label{thm:conv}
Let $\consp$ be given by Assumption~\ref{ass:consp}, and let the drift $a$ and diffusion $b$ coefficients be as in Assumption~\ref{ass:smooth_elliptic}. Consider the solution $(X_t)_{0:T}$ of SDE~\eqref{eq:sde_int} with initial law that satisfies Assumption~\ref{ass:init_gauss_est} and such that condition~\eqref{eq:integr_cond} below holds for $\mrm{Law}(X_t)$. Fix also the regular time mesh $\{t_n=n\De t$, $n=0,\ldots,N(\De t)\}$.

Let $\app{X}_{n}^{\De\ta,\De t,L}$, for $n=0,\dotsc,N(\De t)$, be the sequence of the (laws of the) random variables obtained from Algorithm~\ref{alg:accel} with:
\begin{itemize}
\item
the Euler scheme~\eqref{eq:sde_em} with step $\de t$ proportional to $(\De\ta)^2$,
\item
the restriction operator $\res_{L}$ such that Assumption~\ref{ass:res_hier} holds,
\item
the extrapolation $\mbf{m}_{\comp}$ given in Definition~\ref{dfn:extr_proj}, and
\item
the matching operator from Definition~\ref{dfn:matchop}.
\end{itemize}
Moreover, assume that there is a constant $A=A(\mu_0,T)$, independent of $\De\ta$, $\De t$ and $L$, such that (see Lemma~\ref{lem:tightness})
\begin{equation}\label{eq:mombound}
\sup_{n\leq N(\De t)}\Exp[]\big[|\app{X}_{n}^{\De\ta,\De t,L}|\big]\leq A,
\end{equation}
and 
\begin{equation}\label{eq:Lbound_ent}
\sup_{n\leq N(\De t)}\lre\big(\app{X}_{n+1}^{\De\ta,\De t,L}\|\app{\sem}^*_{\!\De\ta}\app{X}_{n}^{\De\ta,\De t,L}\big)\leq A.
\end{equation}
Then, for all $f\in\contb(\consp)$
\begin{equation}\label{eq:conv}
\lim_{L\to+\infty} \limsup_{\substack{\De\ta,\De t\to0 \\ 0<\De\ta\leq\De t}}\, \sup_{n\leq N(\De t)} \big|\Exp[][f(X_{t_n})] - \Exp[][f(\app{X}_{n}^{\De\ta,\De t,L})]\big|=0.
\end{equation}

\end{thm}

The discussion in Section~\ref{sec:relent_tayl} clarifies the nature of the integrability condition~\eqref{eq:integr_cond}. This assumption, as well as the additional assumption~\eqref{eq:mombound}, is automatically satisfied when $\consp=\T^{d}$. Whether \eqref{eq:integr_cond} holds is a property of SDE~\eqref{eq:sde} itself, and does not rely on the features of micro-macro acceleration method. Assumptions~\eqref{eq:mombound} and~\eqref{eq:Lbound_ent}, on the other hand, are directly concerned with the method, with~\eqref{eq:mombound} being active only when $\consp=\R^d$. In particular, \eqref{eq:Lbound_ent} is essential in controlling the numerical stability of the method as $L$ goes to infinity, see Section~\ref{sec:Lbound}. Definition~\ref{dfn:matchop} of the matching operator and its properties listed in Section~\ref{sec:match-prop} imply numerical stability for every fixed $L$, which we demonstrate in Section~\ref{sec:inop_der}, but are not sufficient to deal with the limit ($L\to+\infty$). At the level of generality we consider in this manuscript, we could not infer these two bounds from more basic principles. Therefore, the validity of~\eqref{eq:Lbound_ent} and~\eqref{eq:mombound} should be checked in a more specific setting, and we restrict ourselves to pointing out the importance of these two bounds.

\section{Entropy expansion in $\De t$}\label{sec:expansion}
Throughout this Section, $\p(t)=\p(t,\cdot)$ stands for the density of the process $(X_t)_{0:T}$ at time $t$. This density is given by the Radon-Nikodym derivative of $\sem_{t}^*\mu_0$ with respect to the Lebesgue measure on $\consp$, where $\sem_{t}^*$ is the adjoint semigroup introduced in Section~\ref{sec:prelim}, see equation~\eqref{eq:dual_sem}, and $\mu_0$ is the law of the initial random variable that satisfies Assumption~\ref{ass:init_gauss_est}.

We are interested in the behaviour of the relative entropy between the probability density $\p(t+\De t)$, for small $\De t>0$, and the density of the matching $\match(\res\p(t+\De t),\p(t))$, which we denote by the same symbol.
The value of this entropy quantifies the error we make when approximating the exact distribution by the matched distribution based on $L$ moments of the exact distribution. Thus, no extrapolation is considered at this stage. According to Theorem~\ref{thm:matchop_prop}\eqref{thm:matchop_prop_pyth}, we can decompose the relative entropy as follows
\begin{equation}\label{eq:lre_pyth}
\lre\big(\p(t+\De t)\big\|\match(\res\p(t+\De t),\p(t))\big) = \lre\big(\p(t+\De t)\big\|\p(t)) - \lre\big(\match(\res\p(t+\De t),\p(t))\big\|\p(t)\big).
\end{equation}
We will study the expansion in $\De t$ around $t$ of the first term on the right-hand side of~\eqref{eq:lre_pyth} in Section~\ref{sec:relent_tayl}, and the expansion of the second term in Section~\ref{sec:match_tayl}.

\subsection{Entropy expansion for a diffusion process}
\label{sec:relent_tayl}
For concreteness, let us first consider the simple example of pure diffusion on the real line, before turning to the more general case.
\begin{exa}\label{exa:widening_gauss}
Assume that the laws of the corresponding stochastic process follow the heat equation, so $\gen=\lap$. If the initial condition at time $t=0$ is the normal distribution with mean $0$ and variance $\Si$, the solution is given by the so-called \emph{widening Gaussian}
\begin{equation}\label{eq:wide_gauss}
\p(t,x) = \frac{1}{\sqrt{2\pi(\Si+2t)}}e^{-x^2 / 2(\Si+2t)},\quad t\geq0,\ x\in\R.
\end{equation}
Thus, the mean stays at $0$ for all times and the variance is $\Si(t)=\Si+2t$. The relative entropy between two solutions separated by $0<\De t\ll1$ is
\begin{equation*}
\lre(\p(t+\De t)\|\p(t)) = \frac{1}{2}\Big\{\frac{\Si(t+\De t)}{\Si(t)} -1 - \ln\frac{\Si(t+\De t)}{\Si(t)}\Big\} = 
\frac{1}{2}\Big\{\frac{2\De t}{\Si(t)} - \ln\Big(1 + \frac{2\De t}{\Si(t)}\Big)\Big\}.
\end{equation*}
Application of the formula $\ln(1+h) = h - h^2/2 + \bo(h^3)$, with $h= 2\De t / \Si(t)$, gives us the expansion in~$\De t$
\begin{equation*}
\lre(\p(t+\De t)\|\p(t)) = (\De t)^2\frac{1}{\Si(t)^2} + \bo_{\Si(t)}\big((\De t)^3\big).
\end{equation*}
The fact that the expansion starts from the second order term is in accordance with the intuition of relative entropy being a "square distance" (cf.~Section~\ref{sec:intro}). Moreover, since $h\leq2\De t/\Si$ for all $t\geq0$, we can argue that the coefficient by the third order term is bounded by $4/(3\Si^3)$, uniformly for all times.
\end{exa}

In this section, our goal is to perform the same expansion in a general case of densities propagated by the dual of the diffusion semigroup $\sem_{t}$ given in~\eqref{eq:sem_def}. Fix $t\in[0,T)$ and $\De t>0$ such, that $t+\De t\leq T$. In the case $\consp=\T^d$, the heat kernel estimates~\eqref{eq:gest} imply that the logarithm $\ln\!\big(\p(t+\De t) / \p(t)\big)$ is bounded on $\T^d$. When $\consp=\R^d$, in view of Lemma~\ref{lem:gauss_est}, we have the following pointwise estimate for the ratio
\begin{equation*}
\frac{\p(t+\De t,x)}{\p(t,x)}\leq \frac{C}{(1+2t)^{d/2}} \exp\Big(2c|x|^2 - \frac{|x|^2}{c(1+2(t+\De t))}\Big)\leq C\exp\Big(\Big(2c -\frac{1}{c(1+2T)}\Big)|x|^2\Big),
\end{equation*}
where $c,C>1$, and the logarithm of this ratio is bounded by $C|x|^2$, uniformly for all $t,\De t$. Thus in both cases, applying the upper bound from Lemma~\ref{lem:gauss_est} to $\p(t+\De t)$ once more when $\consp=\R^{d}$, we can see that the following entropy is finite:
\begin{equation}\label{eq:lre_diff}
\lre(\p(t+\De t)\|\p(t)) = \int_{\consp}\p(t+\De t)\,\ln\frac{\p(t+\De t)}{\p(t)}.
\end{equation}
We aim at expanding~\eqref{eq:lre_diff} with respect to $\De t>0$. Since the entropy vanishes as $\De t$ approaches zero, there will be no zeroth order term. As we will show, the first order term also disappears, due to the conservation of mass by the adjoint semigroup $\sem_{t}^*$.


Let us begin with the Taylor expansion of $\p$ about $t$
\begin{equation}\label{eq:p_tayl}
\p(t+\De t) = \p(t) + \De t\cdot\pder[t]\p(t) + \frac{1}{2}\int_{0}^{\De t}\!\!(\De t-s)\cdot \pder[t]^2\p(t+s)\,\der{s}.
\end{equation}
Kolmogorov's equation~\eqref{eq:fwd_kolm} for the transition kernels implies that the density $\p(t)$ satisfies the \emph{Fokker-Planck equation} $\pder[t]\p(t)=\gen^*\p(t)$. Moreover, since all $\p(t)$ are probability densities, the total mass is conserved and it holds 
\begin{equation}\label{eq:mass_cons}
\int_{\consp}\pder[t]\p(t) = \int_{\consp}\pder[t]^2\p(t)=0,
\end{equation}
for all $t\in[0,T]$. Next, we use another Taylor expansion about $t$ to obtain
\begin{equation}\label{eq:logp_tayl}
\ln\frac{\p(t+\De t)}{\p(t)} = \De t\cdot\pder[t]\ln\!\p\,(t) + \frac{1}{2}(\De t)^2\cdot\pder[t]^2\ln\!\p\,(t) + \frac{1}{6}\int_{0}^{\De t}\!\!(\De t-s)^2\cdot\pder[t]^3\ln\!\p\,(t+s)\,\der{s}.
\end{equation}
Inserting both~\eqref{eq:p_tayl} and~\eqref{eq:logp_tayl} into~\eqref{eq:lre_diff} gives
\begin{align*}
\lre(\p(t+\De t)\|\p(t))
&= \De t\int_{\consp}\p(t)\,\pder[t]\ln\!\p(t)\\[1em]
&+ (\De t)^2\!\int_{\consp}\Big\{\frac{1}{2}\p(t)\pder[t]^2\ln\!\p(t)+\pder[t]\p(t)\pder[t]\ln\!\p(t)\Big\}\\[1em]
&+ (\De t)^3\int_{\consp}\frac{1}{2}\pder[t]\p(t)\pder[t]^2\ln\!\p(t)\\[1em]
&+\int_{0}^{\De t}\!\!(\De t-s)^2\Big\{\int_{\consp}\p(t)\pder[t]^3\ln\!\p(t+s) + \frac{1}{2}\int_{\consp}\pder[t]^2\p(t+s)\pder[t]\ln\!\p(t)\Big\}\der{s}\\[1em]
&+ \int_{0}^{\De t}\!\!(\De t-s)^3\Big\{\frac{1}{4}\int_{\consp}\pder[t]^2\p(t+s)\pder[t]^2\ln\!\p(t) + \frac{1}{6}\int_{\consp}\pder[t]\p(t)\pder[t]^2\ln\!\p(t+s)\Big\}\der{s}\\[1em]
&+ \int_{0}^{\De t}\!\int_{0}^{\De t}\!\!(\De t-s)(\De t-s')^2 \Big\{\frac{1}{12} \int_{\consp}\pder[t]^2\p(t+s)\pder[t]^3\ln\!\p(t+s')\Big\}\der{s'}\der{s}
\end{align*}
First, note that the identity $\p\,\pder[t]\!\ln\!\p=\pder[t]\p$, together with~\eqref{eq:mass_cons}, implies that the integral by $\De t$ in the first line vanishes. In the second line, according to the identity $\p\,\pder[t]^2\ln\!\p=\pder[t]^2\p-\pder[t]\p\,\pder[t]\!\ln\!\p$, the integral reads
\begin{equation}\label{eq:finf_sde}
\finf(t) \doteq \int_{\consp}\pder[t]\p(t)\,\pder[t]\!\ln\!\p(t) = \Exp[\p(t)]\big[|\pder[t]\ln\!\p(t)|^2\big] = \Exp[\p(t)]\big[\big|\gen^*\p(t) / \p(t)\big|^2\big],
\end{equation}
the last equality obtained by using the Fokker-Planck equation. $\finf(t)$ is the so called \emph{Fisher information}~\cite[Ch.~2.6]{Kullback1978} with respect to the time parameter.

\begin{rem}[On refining the expansion of $\lre(\p(t+\De t)\|\p(t))$]\label{rem:ent_expan}
To guarantee that $\finf(t)$ is finite and to establish a uniform in time bound on the higher order terms in the above expansion, we need to control the integrals $\int_{\consp}\pder[t]^i\p(s)\pder[t]^j\ln\p(s')$, with $i,j=0,\dotsc,3$, as $s,s'$ ranges in $[0,T]$. A~simple calculation reveals a~recursive formula $\pder[t]^j\ln\p = \pder[t]^j\p/\p+P_j(\pder[t]^{j-1}\ln\p,\dotsc,\pder[t]\ln\p)$, where $P_j$ is a polynomial of degree $j$. Therefore, we need only to ensure that
\begin{equation}\label{eq:integr}
\int_{\consp}\pder[t]^i\p(s)\pder[t]^j\p(s')/\p(s') \leq \mrm{const},
\end{equation}
for all $s,s'\in[0,T]$. In the compact case $\consp=\T^d$, the lower Gaussian estimate in~\eqref{eq:gest} guarantees that $\p(s')$ is bounded away from $0$ uniformly in $s'\in[0,T]$, and the regularity of drift and diffusion coefficients imply the boundedness of time derivatives $\pder[t]^k\p$ on $[0,T]\times\consp$. These two fact are enough to justify~\eqref{eq:integr}. In the non-compact case $\consp=\R^d$, the situation is more complicated. The Gaussian estimates~\eqref{eq:gest} and the related upper bounds on the derivatives of transition densities, see~\cite[Thm.~3.3.11]{Stroock2008a}, are not sufficient to obtain~\eqref{eq:integr}. 
\end{rem}
Motivated by considerations from information theory~\cite[27]{Kullback1978}, we introduce the following integrability condition, which clearly yields~\eqref{eq:integr}:
\begin{align}\label{eq:integr_cond}
\begin{split}
&|\pder[t]^{i}\p(t)|/\p(t)\leq H\ \text{for every}\ t\in[0,T],\ i=1,2,3,\ \text{where $H$ is a function on $\consp$}\ \text{such that}\\
&\int_{\consp}|\pder[t]^{i}\p(t)|H<M<+\infty\ \text{for}\ i=0,1,2,3,\ \text{with constant $M$ independent of $t$.}
\end{split}
\end{align}

Condition~\eqref{eq:integr_cond} ensures, as indicated in Remark~\ref{rem:ent_expan}, that $\finf(t)$ in~\eqref{eq:finf_sde} is well-defined and that all terms from the third line on in the expansion of $\lre(\p(t+\De t)\|\p(t))$, containing at least three powers of $\De t$ (including $\De t$ in the upper integral limit), can be bounded by $\mrm{const}\cdot(\De t)^3$ uniformly in $t$ and $s$. We summarize the result in the following statement.

\begin{lem}\label{lem:exp_diff}
Assume that $\p(t)$ solves on $\consp$ the Fokker-Planck equation $\pder[t]\p(t) = \gen^*\p(t)$, with the drift and diffusion coefficient such that Assumption~\ref{ass:smooth_elliptic} holds, and with initial density $\p_0$ as in Assumption~\ref{ass:init_gauss_est}. When $\consp=\R^d$, assume moreover that $\p(t)$ satisfies~\eqref{eq:integr_cond}. Then, for every fixed final time $T>0$, we have
\begin{gather}\label{eq:entprod}
\begin{aligned}
\lre(\p(t+\De t)\|\p(t)) 
&= \frac{1}{2}(\De t)^2\,\finf(t) + \bo\big((\De t)^3\big),\\
&= \frac{1}{2}(\De t)^2\,\Exp[\p(t)]\big[\big|\gen^*\p(t) / \p(t)\big|^2\big] + \bo\big((\De t)^3\big),
\end{aligned}
\end{gather}
as $\De t$ converges to $0$, uniformly in $t\in[0,T]$.
\end{lem}

Before we finish this Section, let us quickly revisit the case of pure diffusion from Example~\ref{exa:widening_gauss}. The adjoint generator is $\gen^*=-\pder[xx]$ and, using the fact that $\p(t)$ is normal with mean $0$ and variance $\Si(t)$, we have
\begin{multline*}
\Exp[\p(t)]\Big[\Big|\frac{\pder[xx]\p(t)}{\p(t)}\Big|^2\Big] = \Exp[\Norm(0,\Si(t))]\Big[\Big|\frac{X^2}{\Si(t)^2}-\frac{1}{\Si(t)}\Big|^2\Big] = \frac{1}{\Si(t)^4}\Exp[\Norm(0,\Si(t))]\big[(X^2-\Si(t)^2)^2\big]\\[1em]
= \frac{1}{\Si(t)^4}\big(3\Si(t)^2 - 2\Si(t)\cdot\Si(t) + \Si(t)^2\big) = \frac{2}{\Si(t)^2}.
\end{multline*}
Inserting this into~\eqref{eq:entprod} gives us exactly the expansion we obtained directly in Example~\ref{exa:widening_gauss}. In this example we can check, by a direct calculation, that~\eqref{eq:integr_cond} is satisfied with $H(x)$ given by a polynomial of second degree in $|x|$.

\subsection{Entropy expansion with the matching}
\label{sec:match_tayl}
In this Section, we will use the properties of the matching operator to derive an expansion in $\De t$ for the relative entropy
\begin{equation*}
\lre\big(\match\big(\res\p(t+\De t),\p(t)\big)\|\p(t)),
\end{equation*}
for time $t\in[0,T-\De t]$. To this end, let us fix $\De t_0>0$, denote $Q=[0,T-\De t_0]\times[0,\De t_0]$ and define an $\R^{L}$-valued function on $Q$
\begin{equation*}
\bsm{\ze}(t,s) \doteq\bsm{\la}\big(\res\p(t+s)),\p(t)\big).
\end{equation*}
Using (i) the smoothness of densities $(t,x)\mapsto\p(t,x)$, which results from Assumption~\ref{ass:smooth_elliptic}; (ii) the differentiability of $\mbf{m}\mapsto\bsm{\la}(\mbf{m},\mu)$, elucidated in Theorem~\ref{thm:matchop_prop}\eqref{thm:matchop_prop_momdiff}; and (iii) the smoothness of the log-partition function from Lemma~\ref{lem:lpart}, we infer that the functions $\res\p(t+s)$, $\bsm{\ze}(t,s)$, and $A(\bsm{\ze}(t,s),\p(t))$ are smooth with respect to $t$ and $s$, with bounded partial derivatives on $Q$. On this basis, and using~\eqref{eq:minxent_lre}, we obtain an expansion for $\De t<\De t_0$ as follows
\begin{gather}
\begin{aligned}\label{eq:entprod_match2}
\lre\big(\match\big(\res\p(t+\De t),\p(t)\big)\|\p(t))
&= \tp{\bsm{\ze}(t, \De t)}\res\p(t+\De t) - A(\bsm{\ze}(t,\De t),\p(t))\\[0.7em]
&= \De t\cdot\pder[s]\!\Big(\tp{\bsm{\ze}(t,s)}\res\p(t+s) - A(\bsm{\ze}(t,s),\p(t))\Big)_{\!|s=0}\\[0.7em]
&+ \frac{1}{2}(\De t)^2\cdot\pder[s]^2\!\Big(\tp{\bsm{\ze}(t,s)}\res\p(t+s) - A(\bsm{\ze}(t,s),\p(t))\Big)_{\!|s=0}\\[0.7em]
&+ \bo\big((\De t)^3\big),
\end{aligned}
\end{gather}
in which the coefficients in the third order term are bounded uniformly with respect to $t\in[0,T-\De t_0]$. It remains to compute the derivatives. In the following computations, and also later in the text, for a matrix $M$ and a vector $\mbf{v}$, we use the notation $M[\mbf{v}]^2\doteq\tp{\mbf{v}}M\mbf{v}$.

First, note that $\bsm{\ze}(t,0)=\bsm{0}$, and, from~\eqref{eq:lagr_sder_m} and Lemma~\ref{lem:lpart},
\begin{equation}\label{eq:entprod_match1}
\pder[s]\bsm{\ze}(t,s)_{|s=0}=\Big(\sder_{\mbf{m}}\bsm{\la}\big(\res\p(t+s)),\p(t)\big)\res\big(\pder[s]\p(t+s)\big)\Big)_{|s=0} = \Pre[\p(t)](\mbf{\resf})\,\res(\gen^*\p(t)).
\end{equation}
The first derivative equals
\begin{multline*}
\pder[s]\!\Big(\tp{\bsm{\ze}(t,s)}\res\p(t+s) - A(\bsm{\ze}(t,s),\p(t))\Big)\\
=\tp{\big(\pder[s]\bsm{\ze}(t,s)\big)}\res\p(t+s) + \tp{\bsm{\ze}(t,s)}\res\big(\pder[s]\p(t+s)\big)-\tp{\big(\pder[s]\bsm{\ze}(t,s)\big)}\grad[\bsm{\la}]A\big(\bsm{\ze}(t,s),\p(t)\big),
\end{multline*}
and it vanishes at $s=0$ since $\grad[\bsm{\la}]A\big(\bsm{\ze}(t,0),\p(t)\big)=\res\p(t)$. For the second derivative, we have
\begin{multline*}
\pder[s]^2\!\Big(\tp{\bsm{\ze}(t,s)}\res\p(t+s) - A(\bsm{\ze}(t,s),\p(t))\Big)\\[1em]
=\tp{\big(\pder[s]^2\bsm{\ze}(t,s)\big)}\res\p(t+s) + 2\tp{\big(\pder[s]\bsm{\ze}(t,s)\big)} \res\big(\pder[s]\p(t+s)\big) + \tp{\bsm{\ze}(t,s)} \res\big(\pder[s]^2\p(t+s)\big)\\[1em]
- \hess[\bsm{\la}]A\big(\bsm{\ze}(t,s),\p(s)\big)\big[\pder[s]\bsm{\ze}(t,s)\big]^2 - \tp{\big(\pder[s]^2\bsm{\ze}(t,s)\big)} \grad[\bsm{\la}]A\big(\bsm{\ze}(t,s),\p(t)\big),
\end{multline*}
and, for $s=0$, it reduces to
\begin{equation*}
2\tp{\big(\pder[s]\bsm{\ze}(t,0)\big)}\res\big(\gen^*\p(t)\big) - \Var[\!\p(t)](\mbf{\resf})\big[\pder[s]\bsm{\ze}(t,0)\big]^2.
\end{equation*}
Combining this with~\eqref{eq:entprod_match1} and~\eqref{eq:entprod_match2}, we finally get the following Lemma:
\begin{lem}\label{lem:exp_match}
Assume that $\p(t)$ solves on $\consp$ the Fokker-Planck equation $\pder[t]\p(t) = \gen^*\p(t)$, with the drift and diffusion coefficient such that Assumption~\ref{ass:smooth_elliptic} holds, and with initial density $\p_0$ as in Assumption~\ref{ass:init_gauss_est}. Moreover, let $\mbf{\resf}\in\contb(\consp,\R^L)$ with $\{1,\resf_1,\dotsc,\resf_L\}$ independent modulo the Lebesgue measure on $\consp$. Then, for every fixed final time $T>0$, we have
\begin{equation}\label{eq:entprod_match}
\lre\big(\match\big(\res\p(t+\De t),\p(t)\big)\|\p(t)) = \frac{1}{2}(\De t)^2\,\Pre[\p(t)](\mbf{\resf})\big[\res(\gen^*\p(t))\big]^2 + \bo\big((\De t)^3\big),
\end{equation}
as $\De t$ goes to $0$, uniformly in $t\in[0,T]$.
\end{lem}
For the exponential family $\expd(\bsm{\la},\mu)$ from Section~\ref{sec:expfam}, the Fisher information matrix is defined as~\cite[Sec.~2.1]{AmaNag2000}
\begin{equation*}
\finf_{\bsm{\la}}(\mu) \doteq \Exp[\expd(\bsm{\la},\mu)]\Big[\tp{\Big(\grad[\bsm{\la}]\ln e^{\tp{\bsm{\la}}\mbf{\resf}-A(\bsm{\la},\mu)}\Big)}\Big(\grad[\bsm{\la}]\ln e^{\tp{\bsm{\la}}\mbf{\resf}-A(\bsm{\la},\mu)}\Big)\Big],\quad\text{cf.~\eqref{eq:finf_sde},}
\end{equation*}
which simplifies, by evaluating the gradients, to
\begin{equation*}
\finf_{\bsm{\la}}(\mu) = \Exp[\expd(\bsm{\la},\mu)]\big[\tp{\big(\mbf{\resf}-\grad[\bsm{\la}]A(\bsm{\la},\mu)\big)}\big(\mbf{\resf}-\grad[\bsm{\la}]A(\bsm{\la},\mu)\big)\big] = \hess[\bsm{\la}]A(\bsm{\la},\mu),
\end{equation*}
according to Lemma~\ref{lem:lpart}. Thus, for $\bsm{\la}=\bsm{0}$ we have $\finf_{\bsm{0}}(\p(t)) = \Var[\p(t)](\mbf{\resf})$, and we can express the coefficient accompanying $(\De t)^2$ in~\eqref{eq:entprod_match} as $(1/2)\finf_{\bsm{0}}(\p(t))^{-1}[\res(\gen^*\p(t))]^2$.

\subsection{Summary}
To sum up the results of this Section, we combine the expansions from Lemmas~\ref{lem:exp_diff} and~\ref{lem:exp_match} with identity~\eqref{eq:lre_pyth} to obtain
\begin{gather}\label{eq:entr_match_expan}
\begin{aligned}
\lre\big(\p(t+\De t)\big\|\match(\res\p(t+\De t),\p(t))\big) 
&= \frac{(\De t)^2}{2} \Big(\Exp[\p(t)]\big[\big|\gen^*\p(t) / \p(t)\big|^2\big] - \Pre[\p(t)](\mbf{\resf})\big[\res(\gen^*\p(t))\big]^2\Big)\\[0.5em]
&+ \bo\big((\De t)^3\big).
\end{aligned}
\end{gather}
As we discussed, the coefficient by $(\De t)^2$ can be identified with the difference between the Fischer information $\finf(t)$, corresponding to the time parametrized family of densities generated by $\sem_{t}^*$, and the quadratic form $\finf_{\bsm{0}}(\p(t))^{-1}[\res(\gen^*\p(t))]^2$, where $\finf_{\bsm{0}}(\p(t))$ is the Fisher information matrix of the exponential family $\expd(\bsm{0},\p(t))$. We will employ this expansion in Section~\ref{sec:locerr} to estimate the infinitesimal error due to the extrapolation with finite number of moments, see~\eqref{eq:locerr_zero_ts}, and to prove the consistency of local errors as the number of moments grows to infinity, see Section~\ref{sec:locerr_cons}.

\section{Numerical stability}
\label{sec:stab}
In this Section, we investigate the numerical stability of the micro-macro acceleration method that will allow us to move from the global error to a cumulative sum of local errors. To be more precise, in Section~\ref{sec:inop} we define, for any macroscopic step $\De t>0$, microscopic window $\De\ta>0$, and fixed number of macroscopic state variables $L$, the \emph{increment operator} $\mu\mapsto\inop(\mu)$ that encodes one step of Algorithm~\ref{alg:accel}, as described in Section~\ref{sec:conv}. With this mapping at hand, the distribution $\mu_{n}$ of the random variable $\app{X}_{n}^{\De\ta,\De t,L}$, obtained from the numerical procedure after $n\leq N(\De t)$ steps, writes as the iterate
\begin{equation*}
\mu_{n} = \inop^{n}(\mu_0),
\end{equation*}
with $\mu_0$ the initial law. The increment operator $\inop$ depends of course on all the parameters of the micro-macro acceleration method, and we indicate them in later sections as appropriate.

The numerical stability of the micro-macro acceleration method reduces to proving the following Lipschitz estimate
\begin{equation}\label{eq:inop_lip}
\|\inop(\nu) - \inop(\mu)\|_{TV} \leq (1+ C_{L}\!\cdot\De t)\|\nu-\mu\|_{TV},
\end{equation}
with a constant $C_{L}>0$ that does not depend on $\De t$, $\De\ta$, nor on $\mu,\nu$. After some preparatory considerations in Section~\ref{sec:stab_lip}, which link~\eqref{eq:inop_lip} with a bound on the directional derivative of $\inop$, we present a detailed construction of the increment operator $\inop$ in Section~\ref{sec:inop}. Then, in Section~\ref{sec:inop_der}, we demonstrate that~\eqref{eq:inop_lip} holds and prove the uniformity of the constant $C_{L}$ for an appropriate family of triples $(\De t,\De\ta,\mu)$.

Note that with estimate~\eqref{eq:inop_lip} at hand, by the use of a telescopic sum, we can bound the error in total variation as
\begin{gather}\label{eq:globerr_locerr}
\begin{aligned}
\sup_{\mathclap{n\leq N(\De t)}}\|\mu_{n} - \mu(n\De t)\|_{TV}
&\leq\sup_{n\leq N(\De t)}\sum_{n'=1}^{n}\big\|\inop^{n-n'}\big(\mu(n'\De t)\big) - \inop^{n-n'-1}(\mu(n'\De t))\big\|_{TV}\\[0.5em]
&\leq\sup_{n\leq N(\De t)} e^{n\De tC_{L}} \sum_{n'=1}^{n}\big\|\inop\big(\mu(n'\De t)\big)-\mu(n'\De t)\big\|_{TV}\\[0.5em]
&\leq e^{TC_{L}}\sum_{n=1}^{N(\De t)}\big\|\inop\big(\mu(n\De t)\big)-\mu(n\De t)\big\|_{TV},
\end{aligned}
\end{gather}
where $\mu(n\De t)$ are the laws of the exact solution to~\eqref{eq:sde_int} evaluated on the time mesh. Since the left-hand side of~\eqref{eq:globerr_locerr} dominates the weak error in Theorem~\ref{thm:conv}, we reduce the study of convergence to the consistency of local errors $\inop\big(\mu(n\De t)\big)-\mu(n\De t)$ in the total variation distance. We analyse the behaviour of local errors as $\De t$ goes to $0$ in Section~\ref{sec:locerr_extrap}, but for convergence we also need to consider the limit as $L$ goes to $+\infty$. To this end, we discuss in Section~\ref{sec:Lbound} when we can have a uniform in $L$ bound on the Lipschitz constants $C_L$. With such bound at hand, the question of convergence reduces to the study of the sum on the right-hand side of~\eqref{eq:globerr_locerr}, investigated in Section~\ref{sec:locerr_cons}.

\subsection{Lipschitz condition for general operators}\label{sec:stab_lip}
In this short Section, we consider a mapping $\mcl{F}\from[0, h_0]\times\P(\consp)\to\P(\consp)$ and depict generic conditions so that it satisfies the appropriate Lipschitz estimate. Our objective is to use these conditions in the case of the increment operator $\mcl{F}(h,\mu)=\mcl{F}_h(\mu)$, and rigorously recover~\eqref{eq:inop_lip}, the numerical stability of the micro-macro acceleration method.
\begin{lem}\label{lem:liplike_est}
Let $h_0>0$ and consider a mapping $\mcl{F}\from[0, h_0]\times\P(\consp)\to\P(\consp)$. Assume that for all $\mu,\nu\in\P(\consp)$ and $h\in[0,h_0]$, it holds that
\begin{enumerate}[(i)]
\item $\mcl{F}(0,\mu)=\mu$,
\item the directional derivative $\pder[\mu]\mcl{F}(h,\mu;\nu-\mu)$ exists,
\item the Fr\'{e}chet derivative $\pder[h]\pder[\mu]\mcl{F}(h,\mu;\nu-\mu)$ exists.
\end{enumerate}
Then, we have
\begin{equation*}
\|\mcl{F}(h,\nu)-\mcl{F}(h,\mu)\|_{TV}
\leq\sup_{0\leq\al,\be\leq1} \|\pder[h]\pder[\mu]\mcl{F}(h_\al,\mu_\be;\nu-\mu)\|_{TV}\cdot h + \|\nu-\mu\|_{TV},
\end{equation*}
where $h_\al=(1-\al)h_0$ and $\mu_\be=\be\mu+(1-\be)\nu$.
\end{lem}
\begin{proof}
Define the mapping $\mcl{G}\from[0,h_0]\times\P(\consp)\to\Mb(\consp)$ by putting $\mcl{G}(h,\mu)=\mcl{F}(h,\mu)-\mu$. Then $\mcl{G}(0,\cdot)\equiv0$ and $\pder[h]\pder[\mu]\mcl{G} = \pder[h]\pder[\mu]\mcl{F}$. According to the mean value inequality for directional derivatives (see Lemma~\ref{lem:dirder}),  we get
\begin{equation*}
\|\mcl{G}(h,\nu)-\mcl{G}(h,\mu)\|_{TV}\leq\sup_{0\leq\be\leq1}\|\pder[\mu]\mcl{G}(h,\mu_\be;\nu-\mu)\|_{TV},
\end{equation*}
for every $h\in[0,h_0]$. Since $\pder[\mu]\mcl{G}(0,\mu_\be;\nu -\mu)=0$, the mean value theorem for vector-valued functions of real variable gives
\begin{equation*}
\|\pder[\mu]\mcl{G}(h,\mu_\be;\nu-\mu)\|_{TV} \leq\sup_{0\leq\al\leq1}\|\pder[h]\pder[\mu]\mcl{G}(h_\al,\mu_\be;\nu-\mu)\|_{TV}\cdot h,
\end{equation*}
for every $0\leq\be\leq1$. Combining these estimates and using the equivalence of mixed derivatives, we obtain
\begin{equation*}
\|\mcl{G}(h,\nu)-\mcl{G}(h,\mu)\|_{TV}\leq \sup_{0\leq\al,\be\leq1}\|\pder[h]\pder[\mu]\mcl{F}(h_\al,\mu_\be;\nu-\mu)\|_{TV}\cdot h,
\end{equation*}
which, together with
\begin{equation*}
\|\mcl{F}(h,\nu)-\mcl{F}(h,\mu)\|_{TV}
\leq\|\mcl{G}(h,\nu)-\mcl{G}(h,\mu)\|_{TV} +
\|\nu-\mu\|_{TV},
\end{equation*}
leads to the conclusion.
\end{proof}

The directional derivatives are not necessarily linear with respect to the direction, so in general we cannot say more about the total variation of the mixed derivative in Lemma~\ref{lem:liplike_est}. However, if we can demonstrate, for other reasons, that the mixed derivative $\pder[h]\pder[\mu]\mcl{F}$ is at least sublinear with respect to the direction, uniformly in $h$ and $\mu$, we can derive a Lipschitz estimate for $\mcl{F}$. More precisely, the following result holds:
\begin{cor}\label{cor:lip_est}
Let $\mcl{F}$ be the mapping from Lemma~\ref{lem:liplike_est}, and assume additionally that there is a convex set $\msr{C}\subset\P(\consp)$ and a constant $C_{\mrm{Lip}}>0$ such that
\begin{equation*}
\|\pder[h]\pder[\mu]\mcl{F}(h,\mu;\nu-\mu)\|_{TV}
\leq C_{\mrm{Lip}}\|\nu-\mu\|_{TV}
\end{equation*}
for all $h\in[0,h_0]$ and $\mu,\nu\in\msr{C}$. Then, we have
\begin{equation*}
\|\mcl{F}(h,\nu)-\mcl{F}(h,\mu)\|_{TV}\leq (1+ C_{\mrm{Lip}}\cdot h)\|\nu-\mu\|_{TV}.
\end{equation*}
\end{cor}
The last estimate is exactly~\eqref{eq:inop_lip} when $\mcl{F}(\De t,\cdot)=\inop_{\De t}$, the increment operator. We devote the remainder of this Section to the proper definition of $\inop_{\De t}$ and the confirmation of all assumptions in Lemma~\ref{lem:liplike_est} and Corollary~\ref{cor:lip_est}.

\subsection{One-step increment operator}\label{sec:inop}


Let us now detail the construction of the increment operator $\inop$. Throughout this Section and Section~\ref{sec:inop_der}, we fix $L$ and the vector of restriction functions $\mbf{\resf}\in\contb(\consp,\R^L)$, so we do not indicate this parameter. Denote by $\momsp=\momsp(\consp,\mbf{\resf})$ the corresponding moment space. To deal with the moment problem, see Definition~\ref{dfn:extr_proj}, we first define the appropriate projection operator into $\inter\momsp$. 

To this end, consider the compact curve in the moment space
\begin{equation*}
\ga(\mu_0,T)=\{\res\big(\sem_{t}^*\mu_0\big):\ t\in[0,T]\}\subset\momsp,
\end{equation*}
generated be the exact trajectory of the adjoint diffusion semigroup. The estimates on the density in~\eqref{eq:init_gauss_est} imply, in particular, that $\mu_0$ has full support on $\consp$ and is equivalent to the Lebesgue measure. Therefore, according to Lemma~\ref{lem:momsp_prop}, $\res\mu_0$ is in the interior of the moment space, and, using Lemma~\ref{lem:gauss_est}, we can see that the same holds true for the whole curve $\ga(\mu_0,T)$. In consequence, it is possible to choose a compact convex set $\comp\subset\R^{L}$ with smooth boundary such that
\begin{equation*}
\ga(\mu_0,T)\subset\inter \comp\subset \comp\subset\inter\momsp.
\end{equation*}
To see this, note that as $\ga(\mu_0,T)$ is a compact subset of $\inter\momsp$, the convex hull $\cspan\ga(\mu_0,T)$ is a compact convex subset of $\inter\momsp$. Since the distance function $d=\dist(\,\cdot\,,\cspan\ga(\mu_0,T))$ is convex and non-expansive, the convolutions $d_\vep\doteq d\ast\de_\vep$, where $\de_\vep\geq0$ is the standard mollifier, are non-negative smooth functions that converge to $d$ uniformly on compact subsets of $\R^L$. Thus, the level sets $\{d_\vep<r_\vep\}$, where $r_\vep\doteq\max_{\,\cspan\!\ga(\mu_0,T)}d_\vep$, are convex supersets of $\cspan\ga(\mu_0,T)$ with smooth boundary. By taking $\vep$ small enough, the uniform convergence on $\cspan\ga(\mu_0,T)$ guarantees that
$\comp\doteq\{d_\vep<r_\vep\}$ is contained in $\inter\momsp$, and we henceforth fix such $\comp$. Since the boundary of $\comp$ is smooth, the metric projection $\proj{\comp}$, from Definition~\ref{dfn:extr_proj}, is also smooth on $\R^{L}$~\cite{Holmes1973}.

Before we proceed to the formula for the increment operator, let us establish two lemmas.
Recall that by $\invimg{\res}(\comp)$ we denote the inverse image of $\comp$ under $\res$. 
\begin{lem}\label{lem:invres_prop}
The set $\invimg{\res}(\comp)$ is a convex and weakly closed subset of $\P(\consp)$. It is weakly compact in $\P(\T^{d})$.
\end{lem}
\begin{proof}
The convexity of $\invimg{\res}(\comp)$ follows directly from the convexity of $\comp$. Note that the condition $\resf_l\in\contb(\consp)$ from Assumption~\ref{ass:res_hier} implies the continuity of the restriction $\res$ on $\P(\consp)$ with respect to the weak convergence of probability measures. Thus $\invimg{\res}(\comp)$ is a weakly closed subset of $\P(\consp)$.

If $\consp=\T^d$, $\P(\consp)$ is weakly compact and the inverse image $\invimg{\res}(\comp)$ as well.
\end{proof}

For the next Lemma, recall that $\mbf{m}_{\comp}(\De t,\De\ta,\mu)$ is given by~\eqref{eq:extr_proj}, the adjoint transition operator $\app{\sem}^{*}$ of the Euler scheme by~\eqref{eq:em_trans}, and the domain of matching $\dmatch(\consp,\mbf{\resf})$ is depicted in Definition~\ref{dfn:matchop}.
\begin{lem}\label{lem:matchdom_inop}
For every $0<\De\ta\leq\De t$ and $\mu\in\P(\consp)$, $\big(\mbf{m}_{\comp}(\De t,\De\ta,\mu),\app{\sem}_{\!\De\ta}^{*}\mu\big) \in\dmatch(\consp,\mbf{\resf})$.
\end{lem}
\begin{proof}
From~\eqref{eq:em_trans} we see that
\begin{equation*}
\app{\sem}_{\!\De\ta}^{*}\mu(A) = \int_{A}\,\int_{\consp}\app{p}(\De\ta,x;\xi)\der{\mu}(\xi)\,\der{x},
\end{equation*}
so it has a density with respect to the Lebesgue measure. This, together with Assumption~\ref{ass:res_hier}, implies that the system $\{1,\resf_1,\ldots,\resf_L\}$ is independent modulo $\app{\sem}_{\!\De\ta}^{*}\mu$. Moreover, the density is always positive, since, from the lower bound in~\eqref{eq:gest_em}, we have
\begin{equation*}
\int_{\consp}\app{p}(\De\ta,x;\xi)\der{\mu}(\xi)\geq \frac{C^{-1}}{\De\ta^{d/2}}\int_{\consp}\exp\Big(-\frac{c|x-\xi|^2}{\De\ta}\Big)\der{\mu}(\xi) > 0,
\end{equation*}
for all $x\in\consp$. Thus, the fact that $\supp(\app{\sem}_{\!\De\ta}^{*}\mu)=\consp$ and the way we fixed $\comp$ guarantee that the pair $\big(\mbf{m}_{\comp}(\De t,\De\ta,\mu),\app{\sem}_{\De\ta}^{*}\mu\big)$ belongs to the domain of matching, see the discussion after Definition~\ref{dfn:matchop}.
\end{proof}

\begin{dfn}\label{dfn:inop}
The family of one-step increment operators $\inop_{\!\De t,\De\ta}\from\invimg{\res}(\comp)\to\invimg{\res}(\comp)$, with parameters $\De t,\De\ta$ such that $0<\De\ta\leq\De t$, reads as follows
\begin{equation*}
\inop_{\!\De t,\De\ta}(\mu) = \match\big(\mbf{m}_{\comp}(\De t,\De\ta,\mu),\app{\sem}_{\!\De\ta}^{*}\mu\big).
\end{equation*}
\end{dfn}

Let us note that the invariance of $\invimg{\res}(\comp)$ under $\inop_{\!\De t,\De\ta}$ follows from the construction, since we have for every $\mu\in\invimg{\res}(\comp)$
\begin{equation*}
\res\inop_{\!\De t\De\ta}(\mu)=\res\match\big(\mbf{m}_{\comp}(\De t,\De\ta,\mu),\app{\sem}_{\!\De\ta}^{*}\mu\big)=\mbf{m}_{\comp}(\De t,\De\ta,\mu)\in \comp.
\end{equation*}
The properties of matching imply $\inop_{\!\De t,\De\ta}(\mu)\equi\app{\sem}_{\!\De\ta}^{*}\mu$. In particular, $\inop_{\!\De t,\De\ta}(\mu)$ has full support in $\consp$.

\subsection{Derivatives and Lipschitz constant of one-step increment operator}\label{sec:inop_der}
For the purpose of this Section, let us first define
\begin{equation}\label{eq:compactif}
\msr{C}\doteq \invimg{\resf}(\comp)\cap\{\mu\in\P(\consp):\ \Exp[\mu][|\cdot|]\leq A\},
\end{equation}
where $A=A(\mu_0,T)$ is a constant on the right-hand side of assumption~\eqref{eq:mombound}, postulated in the hypotheses of Theorem~\ref{thm:conv}. Note that, by taking $A$ large enough, we can have $\Exp[\mu][|\cdot|]\leq A$ for all $\mu\in\P(\T^{d})$, so $\msr{C}=\invimg{\resf}(\comp)$ when $\consp=\T^{d}$. This shows that~\eqref{eq:compactif} is redundant for the considerations on the torus. Nevertheless, for both cases of $\consp$, by Lemmas~\ref{lem:tightness} and~\ref{lem:invres_prop}, the set $\msr{C}$ is a convex, weakly compact subset of $\P(\consp)$ and all laws $\mu(n\De t)$ from estimate~\eqref{eq:globerr_locerr} belong to $\msr{C}$.

Without loss in generality, we consider also a compact, path connected subset $\msr{T}$ of $(\De t,\De\ta)$-space such that $(0,0)\in\msr{T}$ and $\msr{T}\setminus\{0,0\}\subset\msr{T}^0=\{(\De t,\De\ta):\ 0<\De\ta<\De t<\De t_0\}$ with some fixed maximal extrapolation time step $\De t_0$. This is a technical assumption that allows us to take advantage of the continuity of the increment operator on the compact domain $\msr{T}\times\msr{C}$. To see that it does not confine our considerations, note first that we can naturally define $\inop_{0,0}$ as the identity operator on $\P(\consp)$. Moreover, having established the limit~\eqref{eq:conv} with $(\De t,\De\ta)$ ranging only in $\msr{T}$, by the freedom in the choice of $\msr{T}$, we obtain the same limiting behaviour for all $0<\De\ta<\De t$ as written in~\eqref{eq:conv}.

With these assumptions at hand, we devote the remainder of this Section to proving that the increment operator $(\De t,\mu)\mapsto\inop_{\!\De t,\De\ta}(\mu)$, with parameter $\De\ta<\De t$, satisfies the assumptions of Lemma~\ref{lem:liplike_est} and Corollary~\ref{cor:lip_est} on the set $[0,\De t_0]\times\msr{C}$, and a constant $C_{L}$ is uniform when $(\mu,\De t,\De\ta)$ range in $\msr{T}\times\msr{C}$. This leads, as we discussed in Section~\ref{sec:stab_lip}, to the desired Lipschitz estimate~\eqref{eq:inop_lip}. 

First, note that assumption (i) from Lemma~\ref{lem:liplike_est} holds for the increment operator $\inop_{\!\De t,\De\ta}(\mu)$, a~consequence of the~projective property $\match(\res\mu,\mu)=\mu$ of the~matching operator. To show the validity of all the other hypothesis, let us define $p\from\msr{T}^0\times\invimg{\res}(\comp)\to\Bm(\consp)$ as
\begin{equation}\label{eq:power}
p(\De t,\De\ta,\mu) = \tp{\bsm{\la}\big(\mbf{m}_{\comp}(\De t,\De\ta,\mu),\app{\sem}_{\!\De\ta}^{*}\mu\big)}\mbf{\resf} - A\big(\bsm{\la}\big(\mbf{m}_{\comp}(\De t,\De\ta,\mu),\app{\sem}_{\!\De\ta}^{*}\mu\big),\app{\sem}_{\!\De\ta}^{*}\mu\big),
\end{equation}
with extension $p(\mu,0,0)=0$. The main result that we establish reads as follows.
\begin{thm}\label{thm:stab_der_bounds}
For every $(\De t,\De\ta)\in\msr{T}^0$, the directional derivative of $p(\,\cdot\,,\De t,\De\ta)$ exists in all admissible directions $\et\in\Mb(\consp)$ and it reads
\begin{equation*}
\pder[\mu]p(\De t,\De\ta,\mu;\et) = \tp{\bra\mbf{p}_1(\De t,\De\ta,\mu)|\et\ket}\mbf{\resf}\ +\ \bra p_2(\De t,\De\ta,\mu)|\et\ket, 
\end{equation*}
for some functions $\mbf{p}_1\from\msr{T}^0\times\invimg{\res}(\comp)\to\Bm(\consp)^{L}$ and $p_2\from\msr{T}^0\times\invimg{\res}(\comp)\to\Bm(\consp)$ that are Fr\'{e}chet differentiable with respect to $\De t$. Moreover, the norms $\|p\|_{\infty}$, $\|\mbf{p}_1\|_{\infty}$, $\|p_2\|_{\infty}$ and the norms of the derivatives $\|\pder[\De t]\mbf{p}_1\|_{\infty}$, $\|\pder[\De t]p_2\|_{\infty}$ are bounded on the set $\msr{T}\times\msr{C}$.
\end{thm}

The first part of Theorem~\ref{thm:stab_der_bounds}, clearly yields assumptions~(ii) and~(iii) in Lemma~\ref{lem:liplike_est}. Before we proceed to the proof, let us depict how the second part establishes the boundedness of the mixed derivative from Corollary~\ref{cor:lip_est}. Note that, in line with~\eqref{eq:minxent_dens} and Definition~\ref{dfn:matchop}, we can write
\begin{equation*}
\mcl{F}_{\!\De t,\De\ta}(\mu) = e^{p(\De t,\De\ta,\mu)}\app{\sem}_{\!\De\ta}^{*}\mu,
\end{equation*}
and by the differentiability of $p$ and the linearity of $\app{\sem}_{\!\De\ta}^{*}$, the directional derivative of $\mcl{F}_{\!\De t,\De\ta}(\mu)$ exists in all admissible directions $\et\in\Mb(\consp)$. According to Lemma~\ref{lem:dirder}, the product rule gives
\begin{equation*}
\pder[\mu]\mcl{F}_{\!\De t,\De\ta}(\mu;\et) = e^{p(\De t,\De\ta,\mu)}\big[\pder[\mu]p(\De t,\De\ta,\mu;\et)\app{\sem}_{\!\De\ta}^{*}\mu + \app{\sem}_{\!\De\ta}^{*}\et\big].
\end{equation*}
Recall that $\app{\sem}_{\!\De\ta}^{*}$ can be extended to $\Mb(\consp)$ by employing formula~\eqref{eq:em_trans}. Next, using the chain rule from Lemma~\ref{lem:dirder}, we calculate
\begin{align*}
\pder[\De t]\pder[\mu]\mcl{F}_{\!\De t,\De\ta}(\mu;\et) 
&= e^{p(\De t,\De\ta,\mu)}\\
&\hphantom{=}\ \times\Big\{
\big[\pder[\De t]p(\De t,\De\ta,\mu)\cdot\pder[\mu]p(\De t,\De\ta,\mu;\et) + \pder[\De t]\pder[\mu]p(\De t,\De\ta,\mu;\et)\big]\app{\sem}_{\!\De\ta}^{*}\mu\\ 
&\hphantom{=\times\Big\{}\ +\big[\pder[\De t]p(\De t,\De\ta,\mu;\et)+1\big]\app{\sem}_{\!\De\ta}^{*}\et\Big\}.
\end{align*}
From the contractivity of $\app{\sem}_{\!\De\ta}^{*}$ in TV norm, which reads~$\|\app{\sem}_{\!\De\ta}^{*}\et\|_{TV}\leq\|\et\|_{TV}$, and the bound $\|\mu\|_{TV}\leq 1$, valid for all probability distributions, we obtain
\begin{align*}
\|\pder[\De t]\pder[\mu]\mcl{F}_{\!\De t,\De\ta}(\mu;\et)\|_{TV}
&\leq e^{\|p(\De t,\De\ta,\mu)\|_{\infty}}\\
&\hspace{-1.5em}\times\Big\{\|\pder[\De t]p(\De t,\De\ta,\mu)\|_{\infty}
\big(\|\mbf{p}_1(\De t,\De\ta,\mu)\|_{\infty}\|\mbf{\resf}\|_{\infty}+\|p_2(\De t,\De\ta,\mu)\|_{\infty}+1\big)\\
&+\|\pder[\De t]\mbf{p}_1(\De t,\De\ta,\mu)\|_{\infty}\|\mbf{\resf}\|_{\infty}+\|\pder[\De t]p_2(\De t,\De\ta,\mu)\|_{\infty}\ +1\Big\}\|\et\|_{TV}.
\end{align*}
Hence, Corollary~\ref{cor:lip_est}, with $\msr{C}$ as in~\eqref{eq:compactif}, and the boundedness of all norms in Theorem~\ref{thm:stab_der_bounds}, grant the estimate on the norm of the mixed derivative $\|\pder[\De t]\pder[\mu]\mcl{F}\|_{TV}$ on the set $\msr{T}\times\msr{C}$.

\begin{proof}
To simplify the formulas consider first the case when the projection is the identity, that is, let us assume that $\mbf{m}_{\comp}(\De t,\De\ta,\mu)=\mbf{m}(\De t,\De\ta,\mu)\in \comp$. At the end of proof, we will indicate what changes in the general situation and why it does not alter the results. 

\underline{Part 1.}
The value of $p(\De t,\De\ta,\mu)$, as we can see from formula~\eqref{eq:power}, is an affine function in $\Bm(\consp)$ with coefficients $\bsm{\la}\big(\mbf{m}(\De t,\De\ta,\mu),\app{\sem}_{\!\De\ta}^{*}\mu\big)$ and $ A\big(\bsm{\la}\big(\mbf{m}(\De t,\De\ta,\mu),\app{\sem}_{\!\De\ta}^{*}\mu\big),\app{\sem}_{\!\De\ta}^{*}\mu\big)$. Thus, the directional derivative $\pder[\mu]p(\De t,\De\ta,\mu;\et)$ exists, and it clearly has the form given in the statement of Theorem~\ref{thm:stab_der_bounds}, as soon as the directional derivatives of the coefficients exist. That these derivatives exist can be seen from Theorem~\ref{thm:matchop_prop}(\ref{thm:matchop_prop_priordiff}), and we only need to compute the functions $\mbf{p}_1$ and $p_2$.

From~\eqref{eq:lagr_dirder} and~\eqref{eq:extr_proj}, we obtain
\begin{multline*}
\pder[\mu]\bsm{\la}(\mbf{m}(\De t,\De\ta,\mu),\app{\sem}_{\!\De\ta}^{*}\mu;\et)\\[0.5em] =\sder_{\mbf{m}}\bsm{\la}(\mbf{m}(\De t,\De\ta,\mu),\app{\sem}_{\!\De\ta}^{*}\mu)\Big\bra\big[I+\frac{\De t}{\De\ta}(\app{\sem}_{\!\De\ta}-I)\big]\mbf{\resf} - \app{\sem}_{\!\De\ta}e^{p(\De t,\De\ta,\mu)}\big[\mbf{\resf}-\mbf{m}(\De t,\De\ta,\mu)\big]\big|\et\Big\ket,
\end{multline*}
and from~\eqref{eq:lpart_dirder}
\begin{align*}
\pder[\mu]
&A\big(\bsm{\la}\big(\mbf{m}(\De t,\De\ta,\mu),\app{\sem}_{\!\De\ta}^{*}\mu\big),\app{\sem}_{\!\De\ta}^{*}\mu;\et\big)\\[0.5em]
&=\tp{\grad[\bsm{\la}]A\big(\bsm{\la}\big(\mbf{m}(\De t,\De\ta,\mu),\app{\sem}_{\!\De\ta}^{*}\mu\big),\app{\sem}_{\!\De\ta}^{*}\mu\big)}\pder[\mu]\bsm{\la}(\mbf{m}(\De t,\De\ta,\mu),\app{\sem}_{\!\De\ta}^{*}\mu;\et) + \big\bra\app{\sem}_{\!\De\ta}e^{p(\De t,\De\ta,\mu)}\big|\et\big\ket\\[0.5em]
&=\tp{\mbf{m}(\De t,\De\ta,\mu)}\pder[\mu]\bsm{\la}(\mbf{m}(\De t,\De\ta,\mu),\app{\sem}_{\!\De\ta}^{*}\mu;\et) + \big\bra\app{\sem}_{\!\De\ta}e^{p(\De t,\De\ta,\mu)}\big|\et\big\ket.
\end{align*}
Comparing these two derivatives with~\eqref{eq:power} yields
\begin{align*}
\phantom{\mbf{p}_1(\De t,\De\ta,\mu) =}
&\begin{aligned}\mathllap{\mbf{p}_1(\De t,\De\ta,\mu) =}
\ \sder_{\mbf{m}}\bsm{\la}(\mbf{m}(\De t,\De\ta,\mu),\app{\sem}_{\!\De\ta}^{*}\mu)\Big\{
&\big[I+\frac{\De t}{\De\ta}(\app{\sem}_{\!\De\ta}-I)\big]\mbf{\resf}\\
& - \app{\sem}_{\!\De\ta}e^{p(\De t,\De\ta,\mu)}\big[\mbf{\resf}-\mbf{m}(\De t,\De\ta,\mu)\big]\Big\},
\end{aligned}\\[0.7em]
p_2(\De t,\De\ta,\mu) =&\ \tp{\mbf{m}(\De t,\De\ta,\mu)}\mbf{p}_1(\De t,\De\ta,\mu) + \app{\sem}_{\!\De\ta}e^{p(\De t,\De\ta,\mu)}.
\end{align*}

We can now infer the existence of the (Fr\'{e}chet) derivatives $\pder[\De t]\mbf{p}_1$ and $\pder[\De t]p_2$ by a careful inspection of the formulas for $\mbf{p}_1$ and $p_2$. We do not need to write down the complete derivatives explicitly; it is enough for our purpose to delineate their main components.

First, we look closely at differentiating $\sder_{\mbf{m}}\bsm{\la}$. To this end, note that
\begin{equation*}
\pder[\De t]\bsm{\la}\big(\mbf{m}(\De t,\De\ta,\mu),\app{\sem}_{\!\De\ta}^{*}\mu\big) = \sder_{\mbf{m}}\bsm{\la}\big(\mbf{m}(\De t,\De\ta,\mu),\app{\sem}_{\!\De\ta}^{*}\mu\big)\pder[\De t]\mbf{m}(\De t,\De\ta,\mu).
\end{equation*}
By the successive application of calculus' rules, we can observe that the derivatives $\pder[\De t]\mbf{p}_1$ and $\pder[\De t]p_2$ contain the following expressions (here we use the abbreviation $\mbf{x}\doteq(\mbf{m}(\De t,\De\ta,\mu),\app{\sem}_{\!\De\ta}^{*}\mu$) for the pair we match):
\begin{align*}
\sder_{\mbf{m}}\bsm{\la}(\mbf{x})
&= \Big(\hess[\bsm{\la}]A\big(\bsm{\la}(\mbf{x}),\app{\sem}_{\!\De\ta}^{*}\mu\big)\Big)^{-1}\quad (\text{see~\eqref{eq:lagr_sder_m}})\\[1em]
\pder[\De t]\sder_{\mbf{m}}\bsm{\la}(\mbf{x})
&= -\sder_{\mbf{m}}\bsm{\la}(\mbf{x})\Big(\pder[\De t] \hess[\bsm{\la}]A\big(\bsm{\la}(\mbf{x}),\app{\sem}_{\!\De\ta}^{*}\mu\big)\Big)\sder_{\mbf{m}}\bsm{\la}(\mbf{x}),\\[1em]
\pder[\De t] \hess[\bsm{\la}]A\big(\bsm{\la}(\mbf{x}),\app{\sem}_{\!\De\ta}^{*}\mu\big)
&= \Exp[\mu]\Big[\big(\pder[\De t]e^{p(\De t,\mu)}\big)\mbf{\resf}\,\tp{\mbf{\resf}}\Big] - 2\,\mbf{m}(\De t,\De\ta,\mu)\tp{\big(\pder[\De t]\mbf{m}(\De t,\De\ta,\mu)\big)}\\[1em]
&= \Exp[\mu]\Big[e^{p(\De t,\mu)}\big(\pder[\De t]p(\De t,\mu)\big)\mbf{\resf}\,\tp{\mbf{\resf}}\Big] - 2\,\mbf{m}(\De t,\De\ta,\mu)\tp{\big(\pder[\De t]\mbf{m}(\De t,\De\ta,\mu)\big)},\\[1em]
\pder[\De t]p(\De t,\De\ta,\mu)
&= \tp{\big(\pder[\De t]\bsm{\la}(\mbf{x})\big)}\mbf{\resf} -\tp{\grad[\bsm{\la}]A\big(\bsm{\la}(\mbf{x}),\app{\sem}_{\!\De\ta}^{*}\mu\big)}\pder[\De t]\bsm{\la}(\mbf{x}),\\[1em]
&= \tp{\big(\mbf{\resf}-\mbf{m}(\De t,\De\ta,\mu)\big)}\!\sder_{\mbf{m}}\bsm{\la}(\mbf{x})\pder[\De t]\mbf{m}(\De t,\De\ta,\mu).
\end{align*}
Moreover, the function $\hess[\bsm{\la}]A$ is jointly continuous in both of its arguments.

Second, the partial derivative in $\De t$ of $\big[I+(\De t/\De\ta)(\app{\sem}_{\!\De\ta}-I)\big]\mbf{\resf}$ reads simply
\begin{equation*}
\pder[\De t]\Big\{\big[I+\frac{\De t}{\De\ta}(\app{\sem}_{\!\De\ta}-I)\big]\mbf{\resf}\Big\} = (\De\ta)^{-1}(\app{\sem}_{\!\De\ta}-I)\mbf{\resf}.
\end{equation*}
Note that both this function and its derivative do not depend on the probability measure $\mu$.

Third, $\mbf{m}(\De t,\De\ta,\mu)$ is given, according to~\eqref{eq:extr}, as $\Exp[\mu]\big[\big[I+(\De t/\De\ta)(\app{\sem}_{\!\De\ta}-I)\big]\mbf{\resf}\big]$. By the linearity of the expectation we get from the previous formula
\begin{equation*}
\pder[\De t]\mbf{m}(\De t,\De\ta,\mu) = (\De\ta)^{-1}\Exp[\mu]\big[(\app{\sem}_{\!\De\ta}-I)\mbf{\resf}\big].
\end{equation*}

To sum up, the $\|\cdot\|_{\infty}$-norms of $p$, $\mbf{p}_1$, $p_2$, and their (Fr\'{e}chet) derivatives with respect to $\De t$ exist, and all depend on the the combination of:
\begin{itemize}
\item
the $\|\,\cdot\,\|$-norms of $\bsm{\la}\big(\mbf{m}(\De t,\De\ta,\mu),\app{\sem}_{\!\De\ta}^*\mu\big)$,
\\ where $(\De t,\De\ta,\mu)\in\msr{T}^0\times\invimg{\res}(\comp)$;
\item
the $\|\,\cdot\,\|_{\infty}$-norms of $\big[I+(\De t/\De\ta)(\app{\sem}_{\!\De\ta}-I)\big]\mbf{\resf}$ and $(\De\ta)^{-1}(\app{\sem}_{\!\De\ta}-I)\mbf{\resf}$,\\ where $(\De t,\De\ta)\in\msr{T}^0$; and
\item
the $\|\,\cdot\,\|$-norms of $\mbf{m}(\De t,\De\ta,\mu)$ and $\pder[\De t]\mbf{m}(\De t,\De\ta,\mu)$,\\ where $(\De t,\De\ta,\mu)\in\msr{T}^0\times\invimg{\res}(\comp)$.
\end{itemize}
Moreover, when we consider the general case with the projection $\proj{\comp}$, the resulting formulas involve additionally the expressions containing $\proj{\comp}(\mbf{m}),\sder\proj{\comp}(\mbf{m})$ and $\sder^2\proj{\comp}(\mbf{m})$, where $\mbf{m}$ belongs to the tubular neighbourhood of $\comp$ within distance $\De t_0$.

\underline{Part 2.}
We now restrict $\invimg{\resf}(\comp)$ to $\msr{C}$ and replace $\msr{T}^0$ with $\msr{T}$. The main obstacle at this point is that, by switching to $\msr{T}$, we need to study the behaviour of functions listed above around $(\De t,\De\ta)=(0,0)$ and investigate what happens when $(\De t,\De\ta)$ converges to $(0,0)$ in the set~$\msr{T}$.

To this end, we first look at the mapping $\De\ta\mapsto(\De\ta)^{-1}(\app{\sem}_{\!\De\ta}-I)\mbf{\resf}$. By adding and subtracting the diffusion semigroup $\sem_{\!\De\ta}$, we split
\begin{equation*}
(\De\ta)^{-1}(\app{\sem}_{\!\De\ta}-I)\mbf{\resf} =(\De\ta)^{-1}(\app{\sem}_{\!\De\ta}-\sem_{\!\De\ta})\mbf{\resf} + (\De\ta)^{-1}(\sem_{\!\De\ta}-I)\mbf{\resf}.
\end{equation*}
For the first summand, we use~\eqref{eq:errest_em} to estimate (see also Lemma~\ref{lem:euler_err_est})
\begin{align*}
\|(\app{\sem}_{\!\De\ta}-\sem_{\!\De\ta})\mbf{\resf}\|_{\infty}
&\leq C\|\mbf{\resf}\|_{\infty}\frac{(\De\ta)^{(1-d)/2}}{K}\sup_{\xi\in\consp}\int_\consp\exp\Big(-\frac{c|x-\xi|^2}{\De\ta}\Big)\der{x}\\[0.5em]
&\leq C\|\mbf{\resf}\|_{\infty}\frac{\sqrt{\De\ta}}{K}\leq C\|\mbf{\resf}\|_{\infty}(\De\ta)^{3/2},
\end{align*}
where in the last bound, we used the fact that $K$ is proportional to $(\De\ta)^{-1}$. Since the coordinates of $\mbf{\resf}\in\ocdiff{2}(\consp,\R^{L})$ belong to the domain of the generator $\gen$, the second summand, as $\De\ta$ goes to $0$, convergences in the $\|\,\cdot\,\|_{\infty}$-norm to $\gen\mbf{\resf} = \tp{(\gen\resf_1,\dotsc,\gen\resf_L)}$. Moreover, on $\ocdiff{2}(\consp)$ the generator $\gen$ is a second order differential operator, see~\eqref{eq:gen_core}, thus we can estimate the norm $\|\gen\mbf{\resf}\|_{\infty}$ by $\mrm{const}\cdot\|\mbf{\resf}\|_{2,\infty}$, with a constant that depends only on the bounds on the drift and diffusion coefficients. We conclude that the mapping $\De\ta\mapsto(\De\ta)^{-1}(\app{\sem}_{\!\De\ta}-I)\mbf{\resf}$ is continuous on $[0,\De t_0]$ with norm bounded by $C\|\mbf{\resf}\|_{2,\infty}$ and a constant that depends only on $\De t_0$ and the drift and diffusion coefficients of equation~\eqref{eq:sde_int}. The same is obviously true for $\big[I+(\De t/\De\ta)(\app{\sem}_{\!\De\ta}-I)\big]\mbf{\resf}$ with $(\De t,\De\ta)$ ranging in~$\msr{T}$.

Consequently, we can also infer the boundedness of $\mbf{m}(\De t,\De\ta,\mu)$ and $\pder[\De t]\mbf{m}(\De t,\De\ta,\mu)$ on $\msr{T}\times\msr{C}$, with uniform bound on their $\|\,\cdot\,\|$-norms by $C\|\mbf{\resf}\|_{2,\infty}$. Recall that $\mbf{m}(\mu;0,0)=\res\mu$ and $\app{\sem}_{0}\mu=\mu$. Thus, by the continuity of $(\mbf{m},\mu)\mapsto\bsm{\la}(\mbf{m},\mu)$ on $\dmatch(\consp,\mbf{\resf})$ (Theorem~\ref{thm:matchop_prop}) we have
\begin{equation}
\bsm{\la}\big(\mbf{m}(\De t,\De\ta,\mu),\app{\sem}_{\!\De\ta}^*\mu\big)\to\bsm{0} = \bsm{\la}(\bsm{0},\mu)\quad\text{as}\ (\De t,\De\ta)\to(0,0),
\end{equation}
for every $\mu\in\msr{C}$. Accordingly, the mapping $(
\De t,\De\ta,\mu)\mapsto\bsm{\la}\big(\mbf{m}(\De t,\De\ta,\mu),\app{\sem}_{\!\De\ta}^*\mu\big)$ is continuous at all points $(0,0,\mu)$, where $\mu\in\msr{C}$. Lemma~\ref{lem:matchdom_inop} guarantees that this mapping is continuous on $\msr{T}\setminus\{(0,0)\}\times\msr{C}$. Therefore, it is bounded on the compact set $\msr{T}\times\msr{C}$. 

Finally, the boundedness of $\proj{\comp}$ and its derivatives is a consequence of the smoothness of $\comp$ and the compactness of its tubular neighbourhood.
\end{proof}

\subsection{Boundedness of Lipschitz constants in the number of macroscopic variables}\label{sec:Lbound}
The goal of this Section is to extend the boundedness of norms on $\msr{T}\times\msr{C}_L$ from Theorem~\ref{thm:stab_der_bounds} to the uniform boundedness in $L$, the number of macroscopic state variables. 

First, let us set some notation. We consider a family of convex sets $\comp_L\subset\R^L$, $L\geq1$, with smooth boundaries and uniformly bounded diameters, such that
\begin{subequations}\label{eq:Lcomp}
\begin{align}
&\ga_L(\mu_0,T) = \{\res_L(\sem^*_t\mu_0):\ t\in[0,T]\}\subset\inter\comp_L\subset\comp_L\subset\inter\momsp_L,\label{eq:Lcomp_inc}\\[0.5em]
&\invimg{\res_L}(\comp_L)\supseteq\invimg{\res_{L+1}}(\comp_{L+1})\supseteq\dotso\supseteq \{\sem^*_t\mu_0:\ t\in[0,T]\}.\label{eq:Lcomp_res}
\end{align}
\end{subequations}
We indicate at the end of Section~\ref{sec:Lbound_pro} how to modify the procedure for constructing $\comp$ from Section~\ref{sec:inop} to obtain such a family, and how to construct appropriate projections.
As in the beginning of Section~\ref{sec:inop_der}, we accordingly define the decreasing family of weakly compact sets
\begin{equation*}
\msr{C}_0 = \{\mu\in\P(X):\ \Exp[\mu][|\cdot|]\leq A\},\quad \msr{C}_L = \invimg{\res_L}(\comp_L)\cap\{\mu\in\P(X):\ \Exp[\mu][|\cdot|]\leq A\},
\end{equation*}
with constant $A=A(\mu_0,T)$ postulated in Theorem~\ref{thm:conv}.

In this Section, we will work with assumption~\eqref{eq:Lbound_ent} that gives a uniform bound (with constant $A$) on the value of relative entropy. By the second part of the proof of Theorem~\ref{thm:stab_der_bounds}, it is enough to ensure the boundedness of Lagrange multipliers $\bsm{\la}_L$ and the norms $\|\mbf{\resf}_L\|_{2,\infty}$ of the restriction functions. The norms of restriction functions are bounded in $L$ due to Assumption~\ref{ass:res_hier}, so we concentrate on the Lagrange multipliers. In Section~\ref{sec:Lbound_adm}, we consider the case when the extrapolated moments are admissible for all $L$. Then, in Section~\ref{sec:Lbound_pro}, we construct a sequence $\comp_L$, used for projections, that satisfies the properties listed in~\eqref{eq:Lcomp} and show how these properties help to establish boundedness in the general case.

\subsubsection{Boundedness with admissible moments}\label{sec:Lbound_adm}
Fix $(\De t,\De\ta,\mu)\in\msr{T}\times\msr{C}_0$ and assume that $\mbf{m}_L=\mbf{m}_L(\De t,\De\ta,\mu)\in\msr{K}_L$ for all $L$. By the definition of extrapolation in~\eqref{eq:extr}, $\mbf{m}_L=(m_1,\dotsc,m_L)$ is a hierarchy of admissible macroscopic states, that is $\mbf{m}_{L+1}=(\mbf{m}_L,m_{L+1})$. The bound in~\eqref{eq:Lbound_ent} implies that $\limsup_{L\to+\infty}\lre(\mu_L\|\app{\sem}^*_{\!\De\ta}\mu)<+\infty$, where $\mu_L=\inop_L(\De t,\De\ta,\mu)\in\msr{C}_L$ is the value of increment operator with $L$ macroscopic variables.
Since the level sets of $\lre(\,\cdot\,\|\app{\sem}^*_{\!\De\ta}\mu)$ are compact in the weak topology, $\{\mu_L\}$ converges weakly, up to a subsequence, to a probability measure $\nu$. Here and in what follows, we do not change the index while passing to a subsequence. As all restriction functions are bounded and continuous, we have that
\begin{equation*}
\Exp[\nu][\mbf{\resf}_L]=\mbf{m}_L, 
\end{equation*}
for all $L\geq1$. Moreover, by the lower semicontinuity of relative entropy, we get
\begin{equation*}
\lre(\nu||\app{\sem}^*_{\!\De\ta}\mu)\leq\liminf_{L\to+\infty}\lre(\mu_L\|\app{\sem}^*_{\!\De\ta}\mu)<+\infty.
\end{equation*}
Therefore, the assumptions of Proposition~\ref{pro:minxent_exist} hold, and there exist a unique measure $\mu_\infty$ that minimises relative entropy to $\app{\sem}^*_{\!\De\ta}\mu$ constrained on $\Exp[\nu][\mbf{\resf}_L]=\mbf{m}_L$, for all $L\geq1$.

The properties of minimum relative entropy moment matching, see~\cite[Thm.~4]{Kruk2004}, imply that the optimal solutions $\mu_L = \match(\mbf{m}_L,\app{\sem}^*_{\!\De\ta}\mu)$ converge in total variation distance to $\mu_\infty$. The measure $\mu_\infty$ is absolutely continuous with respect to $\app{\sem}^*_{\!\De\ta}\mu$, and if we denote by $\exp(p_\infty)\in\leb{1}(\app{\sem}^*_{\!\De\ta}\mu)$ its density, we get
\begin{equation}\label{eq:Lconv1}
\exp\big(\tp{\bsm{\la}_L}\mbf{\resf}_L-A(\bsm{\la}_L,\app{\sem}^*_{\!\De\ta}\mu)\big)\longrightarrow \exp(p_\infty)\quad\text{in}\ \leb{1}(\app{\sem}^*_{\!\De\ta}\mu)\ \text{as}\ L\to+\infty,
\end{equation}
where $\bsm{\la}_L=\bsm{\la}
\big(\mbf{m}_L,\app{\sem}^*_{\!\De\ta}\mu\big)$ are Lagrange multipliers corresponding to $\mu_L$. Moreover, since the relative entropy $\lre(\mu_\infty||\mu_L)$ goes to zero as well, from the Pythagorean identity, Theorem~\ref{thm:matchop_prop}\eqref{thm:matchop_prop_pyth}, we get
\begin{equation}\label{eq:Lconv2}
0\leq\tp{\bsm{\la}_L}\mbf{m}_L - A(\bsm{\la}_L,\app{\sem}^*_{\!\De\ta}\mu)\big) =\lre(\mu_L||\app{\sem}^*_{\!\De\ta}\mu)\nearrow \lre(\mu_\infty||\app{\sem}^*_{\!\De\ta}\mu)\quad \text{as}\ L\to+\infty.
\end{equation}
Setting $c=-\lre(\mu_\infty||\app{\sem}^*_{\!\De\ta}\mu)$ and dividing~\eqref{eq:Lconv1} by the exponential of~\eqref{eq:Lconv2} we arrive at
\begin{equation}\label{eq:Lconv}
\exp\big(\tp{\bsm{\la}_L}(\mbf{\resf}_L-\mbf{m}_L)\big)\longrightarrow\exp\big(c\cdot p_\infty\big)\quad\text{in}\ \leb{1}(\app{\sem}^*_{\!\De\ta}\mu)\ \text{as}\ L\to+\infty.
\end{equation}

Suppose now that the sequence $\|\bsm{\la}_L\|$ is unbounded and restrict to a subsequence for which $\|\bsm{\la}_L\|$ increases monotonically to infinity as $L$ increases. We will show that this subsequence itself always contains a bounded subsequence, which leads to a contradiction and establishes boundedness of the initial sequence.

Passing to yet another subsequence, we can assume that the convergence in~\eqref{eq:Lconv} occurs $\app{\sem}^*_{\!\De\ta}\mu$ almost surely, and by the continuity of $\exp$, the sequence $\tp{\bsm{\la}_L}(\mbf{\resf}_L-\mbf{m}_L)$ converges $\app{\sem}^*_{\!\De\ta}\mu$-a.s.~to $c\cdot p_\infty$. We can now restrict the configuration space $\consp$ to a subset $\otilde{\consp}$ of positive $\app{\sem}^*_{\!\De\ta}\mu$-measure on which the convergence is uniform and all these functions are uniformly bounded. Indeed, Egoroff's theorem~\cite[Thm.~2.2.1]{Bogachev2007} yields the uniform convergence outside a set of arbitrarily small $\app{\sem}^*_{\!\De\ta}\mu$-measure. Since $\exp(c\cdot p_\infty)\in\leb{1}(\app{\sem}^*_{\!\De\ta}\mu)$, Chebyshev inequality~\cite[Thm.~2.5.3]{Bogachev2007} implies that this function is bounded outside a set of arbitrarily small $\app{\sem}^*_{\!\De\ta}\mu$-measure, thus $c\cdot p_\infty$ is bounded from above on this subset. The boundedness from below follows from the fact that $\exp(c\cdot p_\infty)$ is a probability density, and thus the measure $\app{\sem}^*_{\!\De\ta}\mu(\{c\cdot p_\infty\leq-n\})$ must be arbitrarily small for $n$ large enough. Finally, after rejecting all these small subsets, we can restrict the configuration space to an appropriate $\otilde{\consp}$, and as $\tp{\bsm{\la}_L}(\mbf{\resf}_L-\mbf{m}_L)$ converges uniformly to the bounded function $c\cdot p_\infty$, it is itself uniformly bounded.

Consider now the functions $\tp{\bsm{\la}_L}(\mbf{\resf}_L-\mbf{m}_L)$ and $c\cdot p_\infty$ as elements of the Lebesgue space $\leb{2}(\otilde{\consp},\app{\sem}^*_{\!\De\ta}\mu)$ with the scalar product $\bra\,\cdot,\,\cdot\,\ket_2$. By Assumption~\ref{ass:res_hier}, the functions $\resf_l-m_l$ are linearly independent on $\otilde{\consp}$. Let $\otilde{\mbf{\resf}}_L = Q_L(\mbf{\resf}_L-\mbf{m}_L)$ be the orthonormal system in $\leb{2}(\otilde{\consp},\app{\sem}^*_{\!\De\ta}\mu)$ obtained by the Gram-Schmidt procedure with with upper-triangular matrices $Q_L$. If we denote $\otilde{\bsm{\la}}_L=\tp{Q_L}\bsm{\la}_L$, the corresponding Lagrange multipliers, it holds $\tp{\otilde{\bsm{\la}}}_L\otilde{\mbf{\resf}}_L=\tp{\bsm{\la}}_L(\mbf{\resf}_L-\mbf{m}_L)$, and from uniform convergence and orthogonality we get
\begin{equation*}
\tp{\otilde{\bsm{\la}}}_L\otilde{\bsm{\la}}_L = \bra\tp{\otilde{\bsm{\la}}}_L\otilde{\mbf{\resf}}_L,\tp{\otilde{\bsm{\la}}}_L\otilde{\mbf{\resf}}_L\ket_2 = \bra\tp{\bsm{\la}}_L(\mbf{\resf}_L-\mbf{m}_L),\tp{\bsm{\la}}_L(\mbf{\resf}_L-\mbf{m}_L)\ket_2\longrightarrow c^2\bra p_\infty,p_\infty\ket_2.
\end{equation*}

In particular, the sequence $\tp{\otilde{\bsm{\la}}}_L\otilde{\bsm{\la}}_L = \tp{\bsm{\la}}_L(Q_L\tp{Q}_L)\bsm{\la}_L$ is bounded. Moreover, the quadratic form on $\R^L$ induced by the matrix $Q_L\tp{Q}_L$ satisfies
\begin{equation*}
\tp{\mbf{v}}_L(Q_L\tp{Q}_L)\mbf{v}_L\geq 
\mrm{spmin}(Q_L\tp{Q}_L)\|\mbf{v}\|^2,
\end{equation*}
for every $\mbf{v}_L\in\R^L$ where $\mrm{spmin}(\cdot)$ returns the smallest eigenvalue of a symmetric matrix. Applying the foregoing inequality with $\mbf{v}_L=\bsm{\la}_L$, we get
\begin{equation}\label{eq:Lbound}
\|\bsm{\la}_L\|^2\leq \frac{1}{\mrm{spmin}(Q_L\tp{Q}_L)}\tp{\otilde{\bsm{\la}}}_L\otilde{\bsm{\la}}_L.
\end{equation}
From the definition of matrices $Q_L$, it follows that $Q_L\tp{Q}_L=(H_L)^{-1}$, where $(H_L)_{k,l}=\bra\resf_k-m_k,\resf_l-m_l\ket_2$. Using the Spectral Mapping Theorem we obtain
\begin{align*}
\frac{1}{\mrm{spmin}(Q_L\tp{Q}_L)}=\mrm{spmax}(H_L)\leq\|H_L\|
&\leq\Big(\sum_{k,l=1}^L|\bra\resf_k-m_k,\resf_l-m_l\ket_2|^2\Big)^{1/2}\\[0.5em]
&\leq\Big(2\sum_{l=1}^L\|\resf_l-m_l\|^2_{\infty}\Big)^{1/2},
\end{align*}
where, in the last estimate, we used the fact that $\app{\sem}^*_{\!\De\ta}\mu$ is a sub-probability measure on $\otilde{\consp}$. Assumption~\ref{ass:res_hier} ensures that the sequence $\mrm{spmin}(Q_L\tp{Q}_L)^{-1}$ is bounded in $L$, and thus~\eqref{eq:Lbound} establishes the boundedness of $\{\bsm{\la}_L\}$, which leads to the announced contradiction.

Let us now consider the dependence on $(\De t,\De\ta,\mu)$. Since we already have established the boundedness of $\{\|\la_L(\De t,\De\ta,\mu)\|\}_{L\geq1}$, with fixed parameters,~\eqref{eq:Lconv2} yields further the bound
\begin{equation}\label{eq:Lbound2}
\|\la_L(\De t,\De\ta,\mu)\|\leq\mrm{const}\big(\|\mbf{\resf}_L\|, \|\mbf{m}_L(\De t,\De\ta,\mu)\|,\lre(\mu_\infty\|\app{\sem}^*_{\!\De\ta}\mu)\big),
\end{equation}
with $\mu_\infty$ minimising relative entropy to $\app{\sem}^*_{\!\De\ta}\mu$ constrained on $\Exp[\mu_\infty][\mbf{\resf}_L]=\mbf{m}_L(\De t,\De\ta,\mu),\ L\geq1$. We need to investigate the supremum over $L\geq1$ and $(\De t,\De\ta,\mu)\in\msr{T}\times\msr{C}_0$ on the right-hand side of~\eqref{eq:Lbound}. The value of $\sup_{L}\|\mbf{\resf}_L\|$ is finite by Assumption~\ref{ass:res_hier}; the extrapolated moments satisfy $\mbf{m}_L(\De t,\De\ta,\mu)\in\comp_L$ and, since the diameters of $\comp_L$ are bounded, the norms $\|\mbf{m}_L(\De t,\De\ta,\mu)\|$ are uniformly bounded both in $L$ and in $(\De t,\De\ta,\mu)$. Therefore, we can rewrite~\eqref{eq:Lbound2} as
\begin{equation}\label{eq:Lbound3}
\|\la_L(\De t,\De\ta,\mu)\|\leq\mrm{const}\cdotp\lre(\mu_\infty\|\app{\sem}^*_{\!\De\ta}\mu).
\end{equation}
The right-hand side does not depend on $L$ any more, it is a function of $(\De t,\De\ta,\mu)$ solely.

The value of the relative entropy $\lre(\mu_\infty\|\app{\sem}^*_{\!\De\ta}\mu)$ can be in general infinite. This can happen whenever there does not exist any $\mu_\infty$ that has correct moments, or such $\mu_\infty$ exists but has infinite relative entropy with respect to $\app{\sem}^*_{\!\De\ta}\mu$. These two cases are not ruled out by the definition and the properties of the matching operator established in Sections~\ref{sec:matchop} and~\ref{sec:match-prop}; the results presented there give a good control over the relative entropy minimisation only for finite and fixed number of extrapolated moments.

In our case however, assumption~\eqref{eq:Lbound_ent} guarantees that the micro-macro acceleration procedure advances within the values of $(\De t,\De\ta,\mu)$ for which the "infinite" relative entropy minimisation procedure is solvable, with the value of relative entropy bounded by $A$ uniformly in $(\De t,\De\ta,\mu)$. In consequence, \eqref{eq:Lbound3} ensures the uniform bound for the Lagrange multipliers in both $L$ and $(\De t,\De\ta,\mu)$. This in turn, as the derivations in Section~\ref{sec:inop_der} reveal, yields a Lipschitz constant in~\eqref{eq:inop_lip} that is independent of $L$.

\subsubsection{Boundedness with projected moments}\label{sec:Lbound_pro}
We finish our consideration related to the numerical stability with a short investigation of the influence of projection into the moment space on the whole procedure. The conditions in~\eqref{eq:Lcomp}, together with the following property of projections 
\begin{equation}\label{eq:Lcomp_pro}
(\proj{\!\comp_{L+1}}\mbf{m}_{L+1})_l = (\proj{\!\comp_{L}}\mbf{m}_{L})_l,\ l\leq L,\quad \text{if}\ \mbf{m}_{L+1}=(\mbf{m}_{L},m_{L+1}),\ \mbf{m}_L\notin\comp_L,
\end{equation}
allow us to prove the boundedness as follows.

First note that if the extrapolated moment vector $\mbf{m}_L=\mbf{m}_L(\De t,\De\ta,\mu)$ does not belong to $\comp_L$, then $\mbf{m}_{L+1}$ does not belong to $\comp_{L+1}$, a consequence of~\eqref{eq:Lcomp_res}. Thus either the sequence $\{\mbf{m}_L\}$ is admissible for all $L$, and we find ourselves in the setting of Section~\ref{sec:Lbound_adm}, or there is $L_0$ such that $\mbf{m}_L\notin\comp_L$ for all $L\geq L_0$. In the latter case, we perform matchings with the projected vectors $\otilde{\mbf{m}}_L=\proj{\comp_L}\mbf{m}_L\in\comp_L$. According to~\eqref{eq:Lcomp_pro}, the sequence $\{\otilde{\mbf{m}}_L\}_{L\geq L_0}$ forms a hierarchy of macroscopic states, that is $\otilde{\mbf{m}}_{L+1} = (\otilde{\mbf{m}}_L,\otilde{m}_{L+1})$, for $L\geq L_0$. Hence, we can apply the result of Section~\ref{sec:Lbound_adm}, this time with the the sequence $\{\otilde{\mbf{m}}_L\}_{L\geq L_0}$, which gives the boundedness in this case as well.

Finally, let us concisely present a construction of $\comp_L$ that satisfies~\eqref{eq:Lcomp}, along with projection operators such that~\eqref{eq:Lcomp_pro} holds. Assume that, for some $L\geq 1$, we already have such a set $\comp_L$. Consider the cylinder $\comp_L\times\R\subset\R^{L+1}$. Since $\ga_L\subset\inter\comp_L$, by~\eqref{eq:Lcomp_inc}, it follows that $\ga_{L+1}\subset\inter(\comp_L\times\R)$. Because $\ga_{L+1}$ is a compact subset of $\inter\momsp_{L+1}$, we can find small enough $\vep>0$ such that the $\vep$-deflation $(\momsp_{L+1})_{\vep}=\{\mbf{m}_{L+1}\in\momsp_{L+1}:\ \dist(\mbf{m}_{L+1},\bdy\momsp_{L+1})\geq\vep\}$ contains $\ga_{L+1}$ in its interior. Consequently, the intersection $\comp_{L+1,\vep}= (\comp_L\times\R)\cap(\momsp_{L+1})_{\vep}$ satisfies~\eqref{eq:Lcomp}. We can now smooth $\comp_{L+1,\vep}$ in a way that preserves the (already smooth) part of its boundary consisting of the boundary of the cylinder. The procedure to do this is more involved than the one described in Section~\ref{sec:inop}, and we only refer to the results presented in~\cite{Ghomi2002}. We define $\comp_{L+1}$ as given by this smoothing; $\comp_{L+1}$ clearly satisfies~\eqref{eq:Lcomp} as well.

Finally, to guarantee~\eqref{eq:Lcomp_pro}, we define $\proj{\comp_{L+1}}$ as follows. If the extrapolated macroscopic state $\mbf{m}_{L+1}=(\mbf{m}_{L},m_{L+1})$ is such that $\mbf{m}_L\notin\comp_L$, then $\mbf{m}_{L+1}\notin\comp_L\times\R$, and we first project $\mbf{m}_{L+1}$ onto $\bdy(\comp_L\times\R)$, the boundary of $\comp_L\times\R$. Then, if needed, we perform the projection inside $\bdy(\comp_L\times\R)$ onto the boundary of $\comp_{L+1}$. This procedure is smooth, as a composition of two smooth metric projections, and makes sure that~\eqref{eq:Lcomp_pro} holds. If $\mbf{m}_L\in\comp_L$ but $\mbf{m}_{L+1}\notin\comp_{L+1}$, it means that only the last coordinate $m_{L+1}$ is problematic. In this case, we know that $\mbf{m}_{L+1}\in\comp_L\times\R$, and we define $\proj{\comp_{L+1}}\mbf{m}_{L+1}$ as a projection along the line $\mbf{m}_L\times\R$ onto the nearest point on the boundary of $\comp_{L+1}$. This is well defined, by the convexity of $\comp_{L+1}$, and smoothly extends the previous projection mapping.

\section{Analysis of local errors and convergence of micro-macro acceleration method}
\label{sec:locerr}
Following the considerations in Section~\ref{sec:stab}, which led to the estimate~\eqref{eq:globerr_locerr} on the global error in total variation, it only remains to quantify the following local errors
\begin{equation}\label{eq:locerr}
\locerr(\p(t),\mbf{\resf};\De t,\De\ta)\doteq\big\|\p(t+\De t) - \match\big(\mbf{m}(\De t,\De\ta,\p(t)),\app{\p}(t+\De\ta)\big)\big\|_{TV}/\De t,
\end{equation}
where $\mbf{\resf}\in\contb(\consp,\R^{L})$ is a vector of restriction functions, $0<\De\ta<\De t$, and $\p(t)$ is the density of $\mu(t)$, the law of the exact solution $X_t$ of~\eqref{eq:sde_int}, with initial density $\p_0$, see Section~\ref{sec:conv}. Here and throughout this Section, $\app{\p}(t+\De\ta)$ is the density of $\app{\sem}_{\!\De\ta}^{*}\mu(t)$ obtained from the Euler method~\eqref{eq:em_trans} with the initial condition $\mu(t)$ and final time $\De\ta$.

As in Section~\ref{sec:inop_der}, we assume that $\De t\leq\De t_0$ with fixed $\De t_0>\De\ta$, but here we take $\De t_0$ small enough, so that the extrapolation satisfies $\mbf{m}(\De t,\De\ta,\p(t))\in\inter\momsp$ and the matching in~\eqref{eq:locerr} is well defined. The value of $\De t_0$ depends of course on the initial density $\p_0$, but we will show in Section~\ref{sec:locerr_extrap} that it can be chosen uniformly along the trajectory of diffusion semigroup on fixed time horizon $T>0$. This makes the moment projection in~\eqref{eq:extr_proj} redundant throughout the analysis of local errors.

Let us now shortly describe our strategy.
First, we are interested in the limit
\begin{equation}\label{eq:locerr_zero_ts}
\locerr(\p(t),\mbf{\resf})\doteq\limsup_{\substack{\De\ta,\De t\to0 \\ 0<\De\ta\leq\De t}}\locerr(\p(t),\mbf{\resf};\De t,\De\ta),
\end{equation}
that represents the error due to extrapolation with finite number of moments over an infinitesimal time step. The results in Section~\ref{sec:locerr_extrap} will imply in particular that
\begin{equation}\label{eq:locerr_zero_ts_est}
\locerr(\p(t),\mbf{\resf})\leq \sqrt{\finf(t) - \Pre[\p(t)](\mbf{\resf})\!\big[\res(\gen^*\p(t))\big]^2}.
\end{equation}
Here $\finf(t)$ is the Fisher information defined in~\eqref{eq:finf_sde} and the precision matrix $\Pre$ is well defined by Lemma~\ref{lem:hess_nonsing}.

Second, we consider the limit as the number of moments $L$ goes to infinity. This procedure requires employing the hierarchy of macroscopic variables $\{\mbf{\resf}_{L}\}_{L\geq1}$ from Section~\ref{sec:conv} and satisfying Assumption~\ref{ass:res_hier}. 
With this assumption at hand, we will demonstrate in Section~\ref{sec:locerr_cons} that first taking the time steps $\De\ta,\De t$ to zero, and then the number of moments $L$ to infinity, makes the cumulative local error vanish:
\begin{equation}\label{eq:locerr_sum_lim}
\limsup_{L\to+\infty}\limsup_{\substack{\De\ta,\De t\to0 \\ 0<\De\ta\leq\De t}}\,
\sum_{n=0}^{\mathclap{N(\De t)-1}} \locerr(\p(n\De t),\mbf{\resf}_{L};\De t,\De\ta)\De t = 0.
\end{equation}
Note that the sum in~\eqref{eq:locerr_sum_lim} is exactly the last term in~\eqref{eq:globerr_locerr} and, together with the considerations on the numerical stability presented in Section~\ref{sec:stab}, \eqref{eq:locerr_sum_lim} implies the convergence of the micro-macro acceleration method, as stated in Theorem~\ref{thm:conv}.

\subsection{Estimation of error due to finite dimensional extrapolation}
\label{sec:locerr_extrap}
Consider once again the convex and compact set $\comp\subset\inter\momsp$, introduced in Section~\ref{sec:inop}, whose interior contains the curve $\ga(\mu_0,T)$ of moments traced by the true evolution of SDE~\eqref{eq:sde_int} up to time $T$. By choosing $\De t_0$ small enough, we guarantee, due to Definition~\ref{dfn:extr_proj} (of extrapolation), that for all $t\in[0,T]$ we have $\mbf{m}(\De t,\De\ta,\p(t))\in\comp$. In consequence, the moments extrapolated from the solution $\p(t)$ are always feasible and uniformly bounded with respect to $t$. 
This justifies the definition of the local error in~\eqref{eq:locerr}, and now we can formulate the result that we prove in this Section.
\begin{thm}\label{thm:locerr_asym}
Let $\p(t)$ be the density of the solution to SDE~\eqref{eq:sde_int} with elliptic generator~$\gen$ (Assumption~\ref{ass:smooth_elliptic}), and with initial random variable whose law $\mu_0$ satisfies Assumption~\ref{ass:init_gauss_est}. Fix a vector of functions $\mbf{\resf}\in\contb(\consp,\R^{L})$, non-constant and independent modulo Lebesgue, which generate the restriction operator $\res$, and a time $\De t_0>0$ such that
\begin{equation*}
\ga(\mu_0,T) + \msr{B}(\De t_0)\subset\comp,
\end{equation*}
where $\msr{B}(\De t_0)\subset\R^L$ is a ball of radius $\De t_0$ centred at $\bsm{0}$. Then for all $t\in[0,T]$ and all $0<\De\ta<\De t\leq\De t_0$ we have
\begin{equation*}
\locerr(\p(t),\mbf{\resf};\De t,\De\ta)\leq \sqrt{\finf(t) - \Pre[\p(t)](\mbf{\resf})\big[\res(\gen^*\p(t))\big]^2} + \so\big((\De\ta)^0\big) + \bo\big(\sqrt{\De t}\big),
\end{equation*}
with constants uniform in $t,\De\ta,\De t$.
\end{thm}
 Note that Theorem~\ref{thm:locerr_asym} clearly yields~\eqref{eq:locerr_zero_ts_est}.
\begin{proof}
To begin with, we introduce the moments obtained from the exact flow
\begin{equation*}
\otilde{\mbf{m}}(\De t,\p(t))\doteq\res\p(t+\De t).
\end{equation*}
We use the triangle inequality thrice to split the numerator of~\eqref{eq:locerr} into the difference between: (i) the solution at $t+\De t$ and the matching of prior with the moments of the solution~\eqref{eq:locerr_split1}; (ii) two matchings with the same prior but different moments~\eqref{eq:locerr_split2}; and (iii) two matchings with different priors but same moments~\eqref{eq:locerr_split3}.
\begin{subequations}\label{eq:locerr_split}
\begin{align}
\|\p(t+\De t) - \match\big(&\mbf{m}(\p(t);\,\De t,\De\ta),\app{\p}(t+\De\ta)\big)\big\|_{TV}\nonumber\\[0.7em]
&\leq\big\|\p(t+\De t) - \match\big(\otilde{\mbf{m}}(\De t,\p(t)),\p(t+\De\ta)\big)\big\|_{TV}\label{eq:locerr_split1}\\[0.7em]
&+\big\|\match\big(\otilde{\mbf{m}}(\De t,\p(t)),\p(t+\De\ta)\big) - \match\big(\mbf{m}(\De t,\De\ta,\p(t)),\p(t+\De\ta)\big)\big\|_{TV}\label{eq:locerr_split2}\\[0.7em]
&+\big\|\match\big(\mbf{m}(\De t,\De\ta,\p(t)),\p(t+\De\ta)\big) - \match\big(\mbf{m}(\De t,\De\ta,\p(t)),\app{\p}(t+\De\ta)\big)\big\|_{TV}.\label{eq:locerr_split3}
\end{align}
\end{subequations}
To finish the proof, we establish the appropriate estimates for every term in~\eqref{eq:locerr_split}.

\underline{Estimate on~\eqref{eq:locerr_split3}.}
In this case, as we match with the same moments but different priors, we can use Theorem~\ref{thm:lip_match}. To this end define
\begin{equation*}
\Ga \doteq \comp\times\Big\{\mu\in\P(\consp):\ \frac{C^{-1}}{(1+2T)^{d/2}}\exp\big(-2c|x|^2) \leq \rnder{\mu}{x}(x) \leq C\exp\Big(-\frac{|x|^2}{c(1+2T)}\Big)\Big\}.
\end{equation*}
Lemma~\ref{lem:gauss_est} ensures that $\p(t+\De\ta)$ and $\app{p}(t+\De\ta)$ satisfy the bounds in the definition of $\Ga$. The compactness of $\comp$ combined with the uniform control on the tails of all densities $\p$, imply that $\Ga$ is compact in $\|\cdot\|\times weak$ topology on $\R^{L}\times\P(\consp)$, and the validity of inclusion $\Ga\subset\dmatch(\consp,\mbf{\resf})$ follows from~\eqref{eq:matchdom_simple}. As a result, we can apply Theorem~\ref{thm:lip_match} and Lemma~\ref{lem:euler_err_est}, to get
\begin{multline*}
\big\|\match\big(\mbf{m}(\De t,\De\ta,\p(t)),\p(t+\De\ta)\big) - \match\big(\mbf{m}(\De t,\De\ta,\p(t)),\app{\p}(t+\De\ta)\big)\big\|_{TV}\\[0.7em]
\leq C\|\p(t+\De\ta)-\app{\p}(t+\De\ta)\|_{TV}\leq C\frac{\sqrt{\De\ta}}{K} = \bo\big((\De\ta)^{3/2}\big),
\end{multline*}
with constant $C$ independent of $t$, $\De t$, and $K$. The last asymptotic equality holds due to the assumption that $K$ is of the order of $(\De\ta)^{-1}$.

\underline{Estimate on~\eqref{eq:locerr_split2}.}
Here, note that both matchings have the same prior. Based on the continuity of the matching in the weak topology on $\P(\consp)$ and its differentiability in moments (Theorem~\ref{thm:matchop_prop}), we can argue that the function
\begin{equation*}
\comp\times[0,T]\in(\mbf{m},t)\mapsto\bsm{\la}(\mbf{m},\p(t))
\end{equation*}
is bounded and globally Lipschitz in $\mbf{m}$ with constant uniform in $t$. This boundedness, combined with the exponential form of the matching and Lemma~\ref{lem:matchdens_bound}, leads to the estimate
\begin{gather}\label{eq:locerr_split21}
\begin{aligned}
\big\|\match\big(\otilde{\mbf{m}}(\p(t)
;\De t),\p&(t+\De\ta)\big) - \match\big(\mbf{m}(\De t,\De\ta,\p(t)),\p(t+\De\ta)\big)\big\|_{TV}\\[0.5em]
&\leq\Big\|\exp\!\Big(\bsm{\la}(\otilde{\mbf{m}}(\De t,\p(t)),\p(t+\De\ta))\\
&\qquad\qquad-A\big(\bsm{\la}(\otilde{\mbf{m}}(\De t,\p(t)),\p(t+\De\ta)),\p(t+\De\ta)\big)\Big)\\
&\phantom{\leq\,}-\exp\!\Big(\bsm{\la}(\mbf{m}(\De t,\De\ta,\p(t)),\p(t+\De\ta)\big)\\
&\qquad\qquad-A\big(\bsm{\la}(\mbf{m}(\De t,\De\ta,\p(t)),\p(t+\De\ta)),\p(t+\De\ta)\big)\Big)\Big\|_{\infty}\\[0.5em]
&\leq C\big\|\bsm{\la}(\otilde{\mbf{m}}(\De t,\p(t)),\p(t+\De\ta))- \bsm{\la}(\mbf{m}(\De t,\De\ta,\p(t)),\p(t+\De\ta))\big\|,
\end{aligned}
\end{gather}
where $C=C(\|\mbf{\resf}\|_{\infty})\cdot\exp\!\big(\sup_{\comp\times[0,T]}\|\bsm{\la}(\mbf{m},\p(t))\|^2\big)$. Moreover, the Lipschitz continuity of $\mbf{m}\mapsto\bsm{\la}(\mbf{m},\p(t))$ and the differentiability of matching with respect to the extrapolated moments (Theorem~\ref{thm:matchop_prop}\eqref{thm:matchop_prop_momdiff}) yields
\begin{multline}\label{eq:locerr_split22}
\|\bsm{\la}(\otilde{\mbf{m}}(\De t,\p(t)),\p(t+\De\ta))- \bsm{\la}(\mbf{m}(\De t,\De\ta,\p(t)),\p(t+\De\ta))\|\\[0.7em]
\leq\big(\sup_{\comp\times[0,T]}\|\sder[\mbf{m}]\bsm{\la}(\mbf{m},\p(t))\|\big)\|\otilde{\mbf{m}}(\De t,\p(t)) - \mbf{m}(\De t,\De\ta,\p(t))\|.
\end{multline}
To estimate the difference between the exact and extrapolated moments we employ elliptic regularity theory, see~\cite[Ch.~3]{Stroock2008a}, which implies that (i) the function $t\mapsto\res\p(t)=\Exp[\p(t)][\mbf{\resf}]$ is smooth for all $t>0$; (ii) its first derivative is $\res(\gen^*\p(t))$; and (iii) the higher order derivatives are bounded by $C\|\mbf{\resf}\|_{\infty}$, uniformly in $t\in(0,T]$. Hence, for all $\De t>0$, we have
\begin{equation*}
\res\p(t+\De t) = \res\p(t) + \De t\,\res(\gen^*\p(t)) + r_{\!\p(t)}\!\big((\De t)^2\big).
\end{equation*}
The remainder term $r$ is given by the expectations of products of $\mbf{\resf},a, b$, and its derivatives up to fourth order, evaluated on the process $X$ at some (random) time between $0$ and $\De t$. The explicit formula is rather complex, but can be conveniently presented using, for example, rooted tree theory, see~\cite{Rossler2004}. Since the derivatives of $\res\p(t)$ are bounded, we can estimate
\begin{equation*}
r_{\!\p(t)}\!\big((\De t)^2\big)\leq \bo\big((\De t)^2\big),
\end{equation*}
with constants independent of $\p(t)$.
Using expansion of $\res(\p(t+\De t))$ together with the estimate on the remainder $r$, we obtain
\begin{equation}\label{eq:locerr_split23}
\|\otilde{\mbf{m}}(\De t,\p(t)) - \mbf{m}(\De t,\De\ta,\p(t))\|\leq \De t\Big\|\res\Big(\gen^*\p(t)-\frac{\app{\sem}_{\!\De\ta}^{*}\p(t)-\p(t)}{\De\ta}\Big)\Big\| + \bo\big((\De t)^2\big),
\end{equation}
with constants in front of $(\De t)^2$ uniform in $t\in[0,T]$. The coefficient by $\De t$ is $\so(1)$ as $\De\ta\to0$. Indeed, adding and subtracting $\sem_{\!\De\ta}^*\p(t)$ in the numerator, we can estimate
\begin{align*}
\Big\|\res\Big(\gen^*\p(t)-\frac{\app{\sem}_{\!\De\ta}^{*}\p(t)-\p(t)}{\De\ta}\Big)\Big\|
&\leq
\Big\|\res\big(\gen^*\p(t)\big) - \frac{\res\p(t+\De\ta) - \res\p(t)}{\De\ta}\Big\|\\[0.5em]
&+\Big\|\frac{\Exp[\p(t)][\sem_{\!\De\ta}\mbf{\resf}] - \Exp[\p(t)][\app{\sem}_{\!\De\ta}\mbf{\resf}]}{\De\ta}\Big\|.
\end{align*}
Here, we employed~\eqref{eq:dual_sem} and its counterpart for $\app{\sem}_{\!\De\ta}$. The first summand vanishes as $\De\ta$ goes to zero, whereas for the second one we have, from Lemma~\ref{lem:euler_err_est},
\begin{equation*}
\big\|\Exp[\p(t)][\sem_{\!\De\ta}\mbf{\resf}] - \Exp[\p(t)][\app{\sem}_{\!\De\ta}\mbf{\resf}]\big\|\leq C\frac{\sqrt[]{\De\ta}}{K}=\bo\big((\De\ta)^{3/2}\big),
\end{equation*}
since $K$ is proportional to $(\De\ta)^{-1}$. Combining~\eqref{eq:locerr_split21},~\eqref{eq:locerr_split22}, and~\eqref{eq:locerr_split23} we obtain
\begin{equation*}
\|\match\big(\otilde{\mbf{m}}(\De t,\p(t)),\sem_{\!\De\ta}^*\p(t)\big) - \match\big(\mbf{m}(\De t,\De\ta,\p(t)),\sem_{\!\De\ta}^*\p(t)\big)\big\|_{TV} \leq \De t\cdot\so\big((\De\ta)^0\big) + \bo\big((\De t)^2\big),
\end{equation*}
with all constants uniform in $t\in[0,T]$.

\underline{Estimate on~\eqref{eq:locerr_split1}.}
From Pinsker's inequality, we get
\begin{equation*}
\big\|\p(t+\De t) - \match\big(\otilde{\mbf{m}}(\De t,\p(t)),\p(t+\De\ta)\big)\big\|_{TV}\leq \sqrt{2\,\lre\big(\p(t+\De t)\big\|\match\big(\otilde{\mbf{m}}(\De t,\p(t)),\p(t+\De\ta)\big)\big)}.
\end{equation*}
As we match with the exact moments of $\p(t+\De t)$, we can use~\eqref{eq:lre_pyth}, and apply expansion~\eqref{eq:entr_match_expan} to get
\begin{multline}\label{eq:locerr_ent_expan}
\lre\big(\p(t+\De t)\big\|\match\big(\otilde{\mbf{m}}(\De t,\p(t)),\p(t+\De\ta)\big)\big)
\\[0.7em]
=\frac{(\De t-\De\ta)^2}{2}\big(\finf(t) - \Pre[\p(t)](\mbf{\resf})\big[\res(\gen^*\p(t)) \big]^2\big) + \bo\big((\De t-\De\ta)^3\big).
\end{multline}
Note that, since the relative entropy is non-negative, the coefficient by $(\De t)^2$ has to be non-negative as well. Hence, we obtain the following estimate
\begin{equation*}
\|\p(t+\De t) - \match\big(\otilde{\mbf{m}}(\De t,\p(t)),\p(t+\De\ta)\big)\big\|_{TV}\leq(\De t)\sqrt{\finf(t) - \Pre[\p(t)](\mbf{\resf})\big[\res(\gen^*\p(t)) \big]^2} + \bo\big((\De t)^{3/2}\big),
\end{equation*}
with uniform constants in $\bo$ term resulting from the considerations in Section~\ref{sec:expansion}.
\end{proof}

\subsection{Consistency of local errors with hierarchies of moments}
\label{sec:locerr_cons}
In this final Section, we expose the proof of~\eqref{eq:locerr_sum_lim}. At this point, we adopt all the hypothesis in Theorem~\ref{thm:conv}, and consider first the cumulative error of local discretisation errors~\eqref{eq:locerr}. Theorem~\ref{thm:locerr_asym} yields
\begin{multline*}
\sum_{n=0}^{\mathclap{N(\De t)-1}}\,\locerr(\p(n\De t),\mbf{\resf};\De t,\De\ta)\De t\\[0.7em]
\leq \left(\quad \sum_{n=0}^{\mathclap{N(\De t)-1}}\sqrt{\finf(n\De t) - \Pre[\p(n\De t)](\mbf{\resf})\big[\res(\gen^*\p(n\De t)) \big]^2}\,\De t\right) + T\cdot\Big(\so\big((\De\ta)^{0}\big) + \bo\big(\sqrt{\De t}\big)\Big),
\end{multline*}
over $N(\De t)$ steps of micro-macro acceleration method with fixed initial condition $\p_0$. Note that, the first expression on the right-hand side is a Riemann sum on $[0,T]$ for the regular grid $\{n\De t:\ n=0,\ldots,N(\De t)-1\}$. Thus, in the limit as the time steps $\De\ta,\De t$ tend to zero, we get
\begin{equation}\label{eq:locerr_sum_tlim}
\limsup_{\substack{\De\ta,\De t\to0 \\ 0<\De\ta\leq\De t}}\,\sum_{n=0}^{\mathclap{N(\De t)-1}}\,\locerr(\p(n\De t),\mbf{\resf};\De t,\De\ta)\De t\leq \int_{0}^{T}\!\sqrt{\finf(t) - \Pre[\p(t)](\mbf{\resf})\big[\res(\gen^*\p(t)) \big]^2}\,\der{t}
\end{equation}
%

Next, we consider a hierarchy of restriction functions $\{\mbf{\resf}_{L}\}_{L\geq1}$ that defines a sequence of restriction operators $\res_{L}$, see Assumption~\eqref{ass:res_hier}. Let us fix $t$ and write~\eqref{eq:locerr_ent_expan} as follows
\begin{equation}\label{eq:finf_lre}
\finf(t) - \Pre[\p(t)](\mbf{\resf}_{L})\big[\res_{L}(\gen^*\p(t)) \big]^2 = \frac{\lre\big(\p(t+\De t)\big\|\match\big(\res_L\p(t+\De t),\p(t)\big)\big)}{(\De t)^2} + \bo_{L}\big(\De t\big),
\end{equation}
which is valid for all $\De t>0$ small enough, and where we write $\bo_{L}$ to indicate the dependence on the vector $\mbf{\resf}_{L}$ of the constant in front of $\De t$. This constant is bounded in $L$ because the left-hand side of~\eqref{eq:finf_lre} is bounded from above by $\finf(t)$, and the first term on the right-hand side is bounded by $\lre(\p(t+\De t)\|\p(t))/(\De t)^2$. Thus, we can take $\limsup$, as $L$ tends to infinity, on both sides of~\eqref{eq:finf_lre} to obtain
\begin{multline*}
\limsup_{L\to+\infty}\Big(\finf(t) - \Pre[\p(t)](\mbf{\resf}_{L})\big[\res_{L}(\gen^*\p(t)) \big]^2\Big)\\
\leq \limsup_{L\to+\infty}\frac{\lre\big(\p(t+\De t)\big\|\match\big(\res_L\p(t+\De t),\p(t)\big)\big)}{(\De t)^2} + \bo\big(\De t\big).
\end{multline*}
Assumption~\ref{ass:res_hier} guarantees that the entropy between the target measure $\p(t+\De t)$ and the matching of prior $\p(t)$ with the moments coming from the target, goes to zero (even monotonically) as the number of moments used increases, see~\cite[Cor.~3.3]{BorLew1991a}. This means that
\begin{equation*}
\limsup_{L\to+\infty}\Big(\finf(t) - \Pre[\p(t)](\mbf{\resf}_{L})\big[\res_{L}(\gen^*\p(t)) \big]^2\Big)\leq \bo\big(\De t\big),
\end{equation*}
and since $\De t$ can be arbitrarily small, we conclude that this limit is zero.

We just demonstrated that the integrand in~\eqref{eq:locerr_sum_tlim} converges pointwise to zero as $L$ goes to infinity. Since the integrand in~\eqref{eq:locerr_sum_tlim} is bounded by the continuous function $t\to\sqrt{\finf(t)}$, the whole integral is zero in the same limit and~\eqref{eq:locerr_sum_tlim} yields the validity of~\eqref{eq:locerr_sum_lim}.


\section{Conclusions and outlook}
We presented a detailed study of a micro-macro acceleration method for the simulation of stiff SDEs. The method combines short bursts of path simulations with forward in time extrapolation of a few macroscopic state variables. It relies crucially on the constrained minimisation of relative entropy to obtain a new microscopic distribution consistent with the extrapolated macroscopic states.

The nexus of our studies is Theorem~\ref{thm:conv}. This result establishes the convergence, under a number of assumptions, of the micro-macro acceleration method to the exact dynamics of the SDE, in the limit when the extrapolation time step vanishes and the number of macroscopic state variables tends to infinity. Besides that, we grouped the manuscript into three distinct parts: expansion of relative entropy in the extrapolation time step, numerical stability of the method, and the asymptotic behaviour of local errors with vanishing extrapolation time step. The proof of convergence relies on all three elements, but these results are of interest on their own and, especially for a fixed number of macroscopic state variables, have been proved under less strict assumptions than Theorem~\ref{thm:conv}.

The present study revealed many challenges in the theoretical analysis of the micro-macro acceleration method, like the need to deal properly with the infeasibility of extrapolated macroscopic states and with a non-compact configuration space. To pursue this track of research, we will need to study the method with adaptive extrapolation time step and investigate the properties of relative entropy minimisation procedure based on unbounded restriction functions. From the numerical perspective, this work can be complemented with the analysis of adaptive selection of all method parameters, and in particular, on the simultaneous choice of the number of macroscopic state variables as a function of extrapolation time step for a given accuracy. In \cite{DebSamZie2017}, the method was tested on the FENE dumbbells model, where stiffness comes from the boundedness of the configuration space. To further study the efficiency of the method, we should also consider problems with an explicitly present time scale separation, using slow-fast systems of SDEs as model problems.

\appendix
\section{Estimates for the densities of the process and the convergence of the Euler scheme on $\R^d$}\label{sec:estimates_dens}
In this Appendix, we work in the case $\consp=\R^d$ and derive some consequences of the bounds
\begin{equation}\label{eq:init_gauss_est}
C^{-1}\exp(-c|x|^2)\leq\p_0(x)\leq C\exp(-|x|^2/c),
\end{equation}
that we put as Assumption~\ref{ass:init_gauss_est} in Section~\ref{sec:conv} on the density of the law $\mu_0$ of the initial random variable $X_0$.
First, note that Assumption~\ref{ass:smooth_elliptic} guarantees that the laws of the process $(X_t)_{0:T}$, satisfying~\eqref{eq:sde_int} with initial condition $X_0$, have densities for all $t\in[0,T]$, and
\begin{equation*}
\p(t,x) = \int_{\R^d}p(t,x;\xi)\p_0(\xi)\,\der{\xi}.
\end{equation*}
\begin{lem}\label{lem:gauss_est}
If Assumptions~\ref{ass:smooth_elliptic} and bounds~\eqref{eq:init_gauss_est} hold, we have the following Gaussian estimate for all $t\in[0,T]$
\begin{equation*}
\frac{C^{-1}}{(1+2t)^{d/2}}\exp\!\big(-2c|x|^2) \leq \p(t,x) \leq C\exp\!\Big(-\frac{|x|^2}{c(1+2t)}\Big).
\end{equation*}
\end{lem}

\begin{proof}
Fix $t>0$ and $x\in\R^d$. Combining~\eqref{eq:gest} with~\eqref{eq:init_gauss_est} gives
\begin{equation}\label{eq:gauss_est1}
\p(t,x)\leq\frac{C}{t^{d/2}}\int_{\R^d}\exp\!\Big(-\frac{|x-\xi|^2+t|\xi|^2}{ct}\Big)\der{\xi}.
\end{equation}
Note, that the estimate $|x-\xi|\geq\big||x|-|\xi|\big|$, and the Cauchy inequality with $\vep>0$, yield together $|x-\xi|^2\geq (1-\vep)|x|^2 -((1-\vep)/\vep)|\xi|^2$. So, whenever $\vep<1$, we have
\begin{equation*}
|\xi|^2\geq \vep|x|^2 - (\vep/(1-\vep))|x-\xi|^2,
\end{equation*}
and plugging this into~\eqref{eq:gauss_est1} produces
\begin{equation}\label{eq:gauss_est2}
\p(t,x)\leq\frac{C}{t^{d/2}}\exp\!\Big(-\frac{\vep|x|^2}{c}\Big)\int_{\R^d}\exp\!\Big(-\frac{\big(1-t\vep/(1-\vep)\big)|x-\xi|^2}{ct}\Big)\der{\xi}.
\end{equation}
We choose $\vep = 1/(1+2t)<1$, to get rid of $t$ from the numerator of the integrand in~\eqref{eq:gauss_est2}, and use polar coordinates to get
\begin{align*}
\int_{\R^d}\exp\!\Big(-\frac{|x-\xi|^2}{2ct}\Big)\der{\xi} 
&= \int_{0}^{\infty}\left(\int_{\bdy{\msr{B}}(x,r)}\exp\!\Big(-\frac{r^2}{2ct}\Big)\der{S}\right)\der{r}\\
&=d\cdot|\msr{B}(0,1)|\int_{0}^{\infty}\exp\!\Big(-\frac{r^2}{2ct}\Big)\cdot r^{d-1}\,\der{r}\\
&= t^{d/2}\cdot d\cdot|\msr{B}(0,1)|\int_{0}^{\infty}\exp\!\Big(-\frac{r^2}{2c}\Big)\cdot r^{d-1}\der{r},
\end{align*}
where $\msr{B}(x,r)\subset\R^d$ is a ball of radius $r$ centred at $x$. Combining all the expressions independent of $t$ into $C$, we obtain from~\eqref{eq:gauss_est2} the upper bound.

Now, we consider the lower bounds in~\eqref{eq:gest} and~\eqref{eq:init_gauss_est} to estimate
\begin{equation}\label{eq:gauss_est3}
\p(t,x)\geq \frac{C^{-1}}{t^{d/2}}\int_{\R^d}\exp\!\Big(-c\frac{|x-\xi|^2+t|\xi|^2}{t}\Big)\,\der{\xi}.
\end{equation}
Using the standard Cauchy inequality we can verify $|\xi|^2\leq(|x-\xi|+|x|)^2\leq 2|x-\xi|^2 + 2|x|^2$, which together with~\eqref{eq:gauss_est3} produces
\begin{equation}\label{eq:gauss_est4}
\p(t,x)\geq \frac{C^{-1}}{t^{d/2}}\exp\!\big(-2c|x|^2)\int_{\R^d}\exp\!\Big(-c\frac{(1+2t)|x-\xi|^2}{t}\Big)\,\der{\xi}.
\end{equation}
Now, integration through polar coordinates gives
\begin{align*}
\int_{\R^d}\exp\!\Big(-c\frac{(1+2t)|x-\xi|^2}{t}\Big)\,\der{\xi}
&= d\cdot|\msr{B}(0,1)|\int_{0}^{\infty}\exp\!\Big(-c\frac{(1+2t)r^2}{t}\Big)\cdot r^{d-1}\,\der{r}\\
&= \frac{t^{d/2}}{(1+2t)^{d/2}}\,d\cdot|\msr{B}(0,1)|\int_{0}^{\infty}\exp\!\big(-cr^2\big)\cdot r^{d-1}\,\der{r}.
\end{align*}
Thus, from~\eqref{eq:gauss_est4} we finally obtain the lower bound.
\end{proof}

Concerning the densities of the Euler scheme on the small time horizon $\De\ta>0$ with $K$ steps, which are given by
\begin{equation*}
\app{\p}(t_k,x) = \int_{\R^d}\app{p}(t_k,x;\xi)\p_0(\xi)\,\der{\xi},\quad k=1,\dotsc,K,
\end{equation*}
we have, as a consequence of~\eqref{eq:gest_em}, the following result.
\begin{lem}\label{lem:euler_err_est}
If Assumption~\ref{ass:smooth_elliptic} holds and $\De\ta_0>0$, there is a constant $C$ such that for every initial variable $X_0$ with density $\p_0$ satisfying~\eqref{eq:init_gauss_est}
\begin{equation*}
\|\p(\De\ta)-\app{\p}(\De\ta)\|_{TV}\leq C\frac{\sqrt{\De\ta}}{K},
\end{equation*}
and for every $f\in\contb(\R^d)$
\begin{equation*}
\big|\Exp[][f(X_{\De\ta})]-\Exp[][f(\app{X}_{K})]\big|\leq C\|f\|_{\infty}\,\frac{\sqrt{\De\ta}}{K},
\end{equation*}
with $\De\ta\leq\De\ta_0$.
\end{lem}
\begin{proof}
Employing~\eqref{eq:gest_em}, we have the following estimate
\begin{align*}
\int_{\R^d}|\p(t_k,x)-\app{\p}(t_k,x)|\,\der{x}
&\leq \int_{\R^d}\int_{\R^d}|p(t_k,x;\xi)-\app{p}(t_k,x;\xi)|\p_0(\xi)\,\der{\xi}\der{x}\\
&\leq C\frac{\De\ta}{Kt_k^{(d+1)/2}} \int_{\R^d}\p_0(\xi)\int_{\R^d}\exp\Big(-\frac{c|x-\xi|^2}{t_k}\Big)\der{x}\,\der{\xi}\\
&= C\frac{\De\ta}{Kt_k^{(d+1)/2}}|\bdy \msr{B}(0,1)|t_k^{d/2}\int_{\R^d}\exp(-cr^2)\cdot r^{d-1}\,\der{r}.
\end{align*}
Thus for $\De\ta=t_K$ we compute
\begin{equation*}
\|\p(\De\ta)-\app{\p}(\De\ta)\|_{TV} = \int_{\R^d}|\p(\De\ta,x)-\app{\p}(\De\ta,x)|\,\der{x}\leq C\frac{\sqrt{\De\ta}}{K}.
\end{equation*}
The bound for the expectation follows now easily from
\begin{equation*}
\big|\Exp[][f(X_{\De\ta})]-\Exp[][f(\app{X}_{K})]\big|\leq\|f\|_{\infty}\cdot \|\p(\De\ta)-\app{\p}(\De\ta)\|_{TV}.\qedhere
\end{equation*}
\end{proof}


\section{Properties of the matching operator: proofs}\label{sec:match-prop_proofs}

\begin{proof}[Proof of Theorem~\ref{thm:matchop_prop}]
For the proof of~\eqref{thm:matchop_prop_pyth}, see~\cite{Csiszar1975}. The proof of item~\eqref{thm:matchop_prop_cont}~can be found in~\cite{Kruk2004}.

\eqref{thm:matchop_prop_momdiff} Note that, according to Lemma~\ref{lem:lpart} and~\eqref{eq:match_lagr}, the mapping $\mbf{m}\mapsto\bsm{\la}(\mbf{m},\mu)$ is, on the open set $\inter\momsp\big(\supp(\mu),\mbf{\resf}\big)$, the inverse of the function $\bsm{\la}\mapsto F(\bsm{\la})=\grad[\bsm{\la}]A\big(\bsm{\la},\mu\big)$. From Lemma~\ref{lem:lpart}, we know also that $F$ is smooth, and thus we can apply Inverse Function Theorem to get
\begin{equation*}
\sder[\mbf{m}]\bsm{\la}(\mbf{m},\mu)= \big(\sder[\bsm{\la}]F\big(\bsm{\la}(\mbf{m},\mu)\big)\big)^{-1},
\end{equation*}
which is exactly~\eqref{eq:lagr_sder_m}.

\eqref{thm:matchop_prop_priordiff} Denote $\mu_{\vep}\doteq\mu+\vep(\nu - \mu)\in\P(\consp)$, for $\vep\in[0,1]$. Then $(\mbf{m},\mu_{\vep})\in\dmatch$ for all $\vep\in(0,1)$. Indeed, for each $U\in\Bor(\consp)$ we have
\begin{equation*}
\mu_{\vep}(U) = (1-\vep)\mu(U) + \vep\nu(U),
\end{equation*}
and to prove that $\{1,\resf_1,\ldots,\resf_L\}$ is independent modulo $\mu_{\vep}$, consider, for any $(\la_0,\bsm{\la})\in\R^{L+1}\bsh(0,\bsm{0})$, the set $U=\{x\in\consp:\ \la_0+\tp{\bsm{\la}}\mbf{\resf}(x)=0\}$ (see Definition~\ref{dfn:psH}). It is also clear that $\supp(\mu_{\vep})=\supp(\mu)\cup\supp(\nu)$,
which ensures that $\mbf{m}\in\inter\momsp(\supp(\mu_{\vep}),\mbf{\resf})$, for each $\vep$. Hence, both requirements in the formula for~$\dmatch$ in Definition~\ref{dfn:matchop} are fulfilled.

Now let us consider $F\from[0,1]\times\R^{L}\to\R^{L}$ given by $F(\vep,\bsm{\la})=\grad[\bsm{\la}]A(\bsm{\la},\mu_{\vep}) - \mbf{m}$. Because $(\mbf{m},\mu_{\vep})\in\dmatch$, we know from~\eqref{eq:match_lagr}, that the function $\vep\mapsto\bsm{\la}(\vep)\doteq\bsm{\la}(\mbf{m},\mu_{\vep})$ is the implicit solution of the equation $F(\vep,\bsm{\la})=\bsm{0}$. We cannot employ the chain rule directly to this equation, since we have not yet established the differentiability of $\bsm{\la}(\cdot)$ and we are concerned with the point $(0,\bsm{\la}(0))$, which lies on the boundary of the domain of $F$. However, employing the ideas from the proof of the Implicit Function Theorem, we can obtain the desired result for the directional derivatives.

To this end, let us denote $\De\bsm{\la}(\vep) = \bsm{\la}(\vep) - \bsm{\la}(0)$. We need to show (see Definition~\ref{dfn:dirder}) that the limit
\begin{equation*}
\lim_{\vep\searrow0}\frac{\De\bsm{\la}(\vep)}{\vep}
\end{equation*}
exists and is equal to the right-hand side of~\eqref{eq:lagr_dirder}. First note that, according to Lemmas~\ref{lem:lpart} and~\ref{lem:hess_nonsing}, $\sder[\bsm{\la}]F(\vep,\bsm{\la})$ is continuous and non-singular for every $(\vep,\bsm{\la})\in[0,1]\times\R^{L}$. Thus we can put $B\doteq\big(\sder[\bsm{\la}]F(0,\bsm{\la}(0))\big)^{-1}=\sder[\mbf{m}]\bsm{\la}(\mbf{m},\mu)$. Moreover, owing to Lemmas~\ref{lem:lpart} and~\ref{lem:lpart_dirder}, we compute by the chain rule for directional derivatives (recall that $\et=\nu-\mu$)
\begin{align*}
\mbf{a}&\doteq\pder[\vep] F(0,\bsm{\la}(0);+1)\\
&= \pder[\mu]\big(\Exp[\mu]\big[\exp\big(\tp{\bsm{\la}(0)}\mbf{\resf}-A(\bsm{\la(0)},\mu)\big)\mbf{\resf}\big];\et\big)\\
&= \big\bra\exp\big(\tp{\bsm{\la}(0)}\mbf{\resf}-A(\bsm{\la}(0),\mu)\big)\mbf{\resf}\,|\,\et\big\ket - \pder[\mu]A(\bsm{\la}(0),\mu;\et)\, \Exp[\mu]\big[\exp\big(\tp{\bsm{\la}(0)}\mbf{\resf}-A(\bsm{\la}(0),\mu)\big)\mbf{\resf}\big]\\
&= \big\bra \exp\big(\tp{\bsm{\la}(0)}\mbf{\resf}-A(\bsm{\la}(0),\mu)\big)\mbf{\resf}\,|\,\et\big\ket - \big\bra \exp\big(\tp{\bsm{\la}(0)}\mbf{\resf}-A(\bsm{\la}(0),\mu)\big)\,|\,\et\big\ket\mbf{m}\\
&= \big\bra \exp\!\big(\tp{\bsm{\la}(0)}\mbf{\resf}-A(\bsm{\la}(0),\mu)\big)(\mbf{\resf}-\mbf{m})\,\big|\,\et\big\ket,
\end{align*}
where we used~\eqref{eq:match_lagr} in the next to last line. The function $\bsm{\la}(\vep)$, as the implicit solution, satisfies for every $\vep>0$
\begin{equation*}
\De\bsm{\la}(\vep) = -\vep B\mbf{a} + f(\vep, \bsm{\la}(\vep)),
\end{equation*}
where $f(\vep,\bsm{\la}) \doteq B\big(\vep\mbf{a} + B^{-1}(\bsm{\la}(\vep)-\bsm{\la}(0)) - F(\vep,\la)\big)$. The properties of $F$ imply that $\sder[\bsm{\la}]f$ exists for every $(\vep,\bsm{\la})$, is continuous on $[0,1]\times\R^{L}$ with $\sder[\bsm{\la}]f(0,\bsm{\la}(0))=\bsm{0}$, and the function $f$ has the directional derivative $\pder[\vep]f(0,\bsm{\la}(0);+1)=\bsm{0}$. Using the first-order Taylor expansion in $\bsm{\la}$, for every $\vep>0$ we get
\begin{align*}
\frac{\De\bsm{\la}(\vep)}{\vep}
&= -B\mbf{a} + \frac{f(\vep, \bsm{\la}(\vep))}{\vep}\\
&= -B\mbf{a} + \frac{f(\vep, \bsm{\la}(\vep))-f(\vep,\bsm{\la}(0))}{\vep} + \frac{f(\vep,\bsm{\la}(0))-f(0,\bsm{\la}(0))}{\vep}\\
&= -B\mbf{a} + \big(\sder[\bsm{\la}]f(\vep,\bsm{\la}(0)) + r(\vep,\De\bsm{\la}(\vep))\big)\frac{\De\bsm{\la}(\vep)}{\vep} + \frac{f(\vep,\bsm{\la}(0))-f(0,\bsm{\la}(0))}{\vep}
\end{align*}
Rearranging, leads to
\begin{equation}\label{eq:minxent_match_prop1}
\frac{\De\bsm{\la}(\vep)}{\vep} = \big(I - \sder[\bsm{\la}]f(\vep,\bsm{\la}(0)) + r(\vep,\De\bsm{\la}(\vep))\big)^{-1}\Big(\!-B\mbf{a} + \frac{f(\vep,\bsm{\la}(0))-f(0,\bsm{\la}(0))}{\vep}\Big),
\end{equation}
where $I$ denotes the identity matrix. The continuity of $\bsm{\la}(\vep)$, which follows from part~\eqref{thm:matchop_prop_cont}, yields $\De\bsm{\la}(\vep)\to 0$ as $\vep\to0$. Hence, the remainder $r(\vep,\De\bsm{\la}(\vep))$ vanishes as $\vep$ goes to zero, and same is true for $\sder[\bsm{\la}]f(\vep,\bsm{\la}(0))$ and the ratio $(f(\vep,\bsm{\la}(0))-f(0,\bsm{\la}(0)))/\vep$. In consequence, for $\vep$ small enough, we can indeed invert the matrix $I - \sder[\bsm{\la}]f(\vep,\bsm{\la}(0)) + r(\vep,\De\bsm{\la}(\vep))$, and passing to the limit on the right-hand side of~\eqref{eq:minxent_match_prop1} concludes the proof.
\end{proof}

\begin{proof}[Proof of Theorem~\ref{thm:lip_match}]
For every $(\mbf{m},\mu)\in\dmatch$, let us put $p(\mbf{m},\mu)\doteq\tp{\bsm{\la}(\mbf{m},\mu)}\mbf{\resf} - A(\bsm{\la}(\mbf{m},\mu),\mu)$. According to Lemma~\ref{lem:lpart} and Theorem~\ref{thm:matchop_prop}(\ref{thm:matchop_prop_cont}), the mapping $(\mbf{m},\mu)\mapsto p(\mbf{m},\mu)\in\Bm(\consp)$ is continuous on~$\dmatch$.

Take $(\mbf{m},\mu_1),(\mbf{m},\mu_2)\in\Ga$. From the definition of the matching operator $\match$ we have the following estimate
\begin{equation}\label{eq:lip_match1}
\|\match(\mbf{m},\mu_1)-\match(\mbf{m},\mu_2)\|_{TV} \leq \|e^{p(\mbf{m},\mu_1)} - e^{p(\mbf{m},\mu_2)}\|_{\infty} + \|e^{p(\mbf{m},\mu_2)}\|_{\infty}\|\mu_1-\mu_2\|_{TV}.
\end{equation}
The sup norm in the second term can be bounded by $C=\sup_{\Ga}\|e^{p(\mbf{m},\mu)}\|_{\infty}$, which is finite due to the continuity of $p$ and the compactness of $\Ga$.

Let us now consider the first summand in~\eqref{eq:lip_match1}. The exponential function satisfies the Lipschitz condition on bounded domains. Thus, using once more the uniform boundedness of $p$ on $\Ga$, we can find a constant $C$ to obtain pointwise
\begin{gather}\label{eq:lip_match2}
\begin{aligned}
\big|e^{p(\mbf{m},\mu_1)}& - e^{p(\mbf{m},\mu_2)}\big|
\leq\\ 
&\leq C\big|\tp{\big(\bsm{\la}(\mbf{m},\mu_1)-\bsm{\la}(\mbf{m},\mu_2)\big)}\mbf{\resf} + \big(A(\bsm{\la}(\mbf{m},\mu_2),\mu_2) - A(\bsm{\la}(\mbf{m},\mu_1),\mu_1)\big)\big|\\[1em]
&\leq C\big(\|\mbf{\resf}\|_{\infty}\|\bsm{\la}(\mbf{m},\mu_1)-\bsm{\la}(\mbf{m},\mu_2)\|+ \big|A(\bsm{\la}(\mbf{m},\mu_1),\mu_1) - A(\bsm{\la}(\mbf{m},\mu_2),\mu_2)\big|\big)
\end{aligned}
\end{gather}
Note also that for all $\la_1,\la_2$ in a bounded set $B\subset\R^{L}$, we have
\begin{gather}\label{eq:lip_match3}
\begin{aligned}
|Z(\bsm{\la}_1,\mu_1) - Z(\bsm{\la}_2,\mu_2)\big|&\leq \big|\bra e^{\tp{\bsm{\la}_1}\mbf{\resf}},\mu_1-\mu_2\ket\big| + \Exp[\mu_2]\big|e^{\tp{\bsm{\la}_1}\mbf{\resf}} - e^{\tp{\bsm{\la}_2}\mbf{\resf}}\big|\\[1em]
&\leq \sup_{\bsm{\la}\in B}\|e^{\tp{\bsm{\la}}\mbf{\resf}}\|_{\infty}\|\mu_1-\mu_2\|_{TV} + \sup_{\bsm{\la}\in B}\|e^{\tp{\bsm{\la}}\mbf{\resf}}\mbf{\resf}\|_{\infty}|\tp{\bsm{\la}_1}\mbf{\resf} - \tp{\bsm{\la}_2}\mbf{\resf}|\\[1em]
&\leq \sup_{\bsm{\la}\in B}\|e^{\tp{\bsm{\la}}\mbf{\resf}}\|_{\infty}\big(\|\mu_1-\mu_2\|_{TV} + \|\mbf{\resf}\|_{\infty}^{2}\|\bsm{\la}_1-\bsm{\la}_2\|\big),
\end{aligned}
\end{gather}
where we used Lipschitz continuity of the exponential mapping on bounded sets. We can transfer~\eqref{eq:lip_match3} to the estimate for the log-partition function $A$ using the Lipschitz continuity of the logarithm on compact subsets of the positive line. Therefore, in view of the compactness of $\Ga$ and the continuity of $(\mbf{m},\mu)\mapsto\bsm{\la}(\mbf{m},\mu)$, we can combine this with~\eqref{eq:lip_match2}, with $\bsm{\la}_i=\bsm{\la}(\mbf{m},\mu_i)$, and take the supremum over $\consp$ on the left-hand side, to obtain
\begin{equation}
\|e^{p(\mbf{m},\mu_1)} - e^{p(\mbf{m},\mu_2)}\|_{\infty}
\leq C\big(\|\bsm{\la}(\mbf{m},\mu_1)-\bsm{\la}(\mbf{m},\mu_2)\| + \|\mu_1-\mu_2\|_{TV}\big),
\end{equation}
where $C$ depends only on $\Ga$ and $\|\mbf{\resf}\|_{\infty}$.

Finally, we need to estimate the distance between the Lagrange multipliers. To this end, we can apply the mean value inequality for the directional derivatives. According to~\eqref{eq:lagr_dirder}, we have
\begin{equation}
\|\bsm{\la}(\mbf{m},\mu_1)-\bsm{\la}(\mbf{m},\mu_2)\|\leq \sup_{\mu\in[\mu_1,\mu_2]}\|\sder[\mbf{m}]\bsm{\la}(\mbf{m},\mu)\,e^{p(\mbf{m},\mu)} \big(\mbf{\resf}-\mbf{m}\big)\|_{\infty}\cdot\|\mu_1-\mu_2\|_{TV}.
\end{equation}
The supremum over the segment $[\mu_1,\mu_2]$ is clearly bounded by
\begin{equation}\label{eq:lip_match4}
\sup_{\Ga}\|\sder[\mbf{m}]\bsm{\la}(\mbf{m},\mu)\|_{op}\cdot \|e^{p(\mbf{m},\mu)}\|_{\infty} \cdot \big(\|\mbf{\resf}\|_{\infty}+\|\mbf{m}\|\big),
\end{equation}
and the only new ingredient here is the operator norm $\|\sder[\mbf{m}]\bsm{\la}(\mbf{m},\mu)\|_{op}$. From~\eqref{eq:lagr_sder_m} we know that $\sder[\mbf{m}]\bsm{\la}(\mbf{m},\mu)$ is equal to the inverse of the Hessian $\hess[\bsm{\la}]A$, which by Lemma~\ref{lem:lpart} is given by $\Exp[\mu]\big[e^{ p(\mbf{m},\mu)}\mbf{\resf}\tp{\mbf{\resf}}\big]$. The Hessian is positive-definite by Lemma~\ref{lem:hess_nonsing}, and the continuity properties of $p$ and the expectation yield the uniform lower bound on $\Ga$ for the smallest eigenvalue of $(\mbf{m},\mu)\mapsto\hess[\bsm{\la}]A(\mbf{m},\mu)$. This in turn guarantees the boundedness of the operator norm in formula~\eqref{eq:lip_match4}. 
\end{proof}

\printbibliography

\end{document}